\documentclass[10pt, a4paper, oneside]{amsart}

\usepackage[stretch=10]{microtype}
\usepackage{hyperref}

\providecommand{\href}[2]{#2}
\providecommand{\texorpdfstring}[2]{#1}

\providecommand*{\backref}{}
\providecommand*{\backrefalt}{}
\renewcommand*{\backref}[1]{}
\renewcommand*{\backrefalt}[4]{%
	\ifcase #1 %
	\or
	  Cited page #2.
	\else
	  Cited pages #2.
	\fi
}

\newcommand{\N}{\mathbb{N}}
\newcommand{\Z}{\mathbb{Z}}
\newcommand{\R}{\mathbb{R}}
\newcommand{\C}{\mathbb{C}}
\newcommand{\HH}{\mathbb{H}}
\newcommand{\Sbb}{\mathbb{S}}
\newcommand{\dd}{\;{\rm d}}
\newcommand{\de}{{\rm d}}
\newcommand{\norm}[1]{\left\| #1 \right\|}

\newcommand{\SL} {{\mathrm{SL}(2,\R)}}
\newcommand{\SU} {{\mathrm{SU}(2)}}
\newcommand{\SO} {{\mathrm{SO}(2,\R)}}
\newcommand{\SLZ}{{\mathrm{SL}(2,\Z)}}
\newcommand{\GL} {{\mathrm{GL}^+(2,\R)}}
\newcommand{\boD}{\mathcal{D}}
\newcommand{\boL}{\mathcal{L}}
\newcommand{\boT}{\mathcal{T}}
\newcommand{\boQ}{\mathcal{Q}}
\newcommand{\boP}{\mathcal{P}}
\newcommand{\boC}{\mathcal{C}}
\newcommand{\boB}{\mathcal{B}}
\newcommand{\boF}{\mathcal{F}}
\newcommand{\st}{\;:\;}
\newcommand{\ic}{\mathbf{i}}
\newcommand{\smat}[4]{\left( \begin{smallmatrix} #1 & #2 \\ #3 & #4 \end{smallmatrix}\right)}
\newcommand{\Teich}{\mathrm{Teich}}
\newcommand{\Card}{\#}
\DeclareMathOperator{\length}{length}
\DeclareMathOperator{\supp}{supp}
\DeclareMathOperator{\vol}{vol}
\DeclareMathOperator{\Ima}{Im}
\DeclareMathOperator{\sys}{sys}
\newcommand{\boDb}{\overline{\boD^\Gamma}}
\newcommand{\boM}{\mathcal{M}}

\newcommand{\paths}{\kappa}
\newcommand{\partition}{\rho}
\newcommand{\normdeux}[1]{\left\| #1 \right\|'}
\newcommand{\extension}[1]{\overline{#1}}
\newcommand{\ress}{r_{\mathrm{ess}}}
\newcommand{\sigmaess}{\sigma_{\mathrm{ess}}}

\def\MM{\mathcal {M}}
\def\CC{\mathcal {C}}

\DeclareMathOperator{\Leb}{Leb}
\DeclareMathOperator{\dLeb}{dLeb}
\DeclareMathOperator{\id}{id}
\newcommand{\coloneqq}{\mathrel{\mathop:}=}

\newcommand{\boLo}{\overline{\boL}}
\newcommand{\Bo}{\overline{B}}

\newtheorem*{mainthm}{Main Theorem}
\newtheorem{thm}{Theorem}[section]
\newtheorem{prop}[thm]{Proposition}
\newtheorem{definition}[thm]{Definition}
\newtheorem{lem}[thm]{Lemma}
\newtheorem{cor}[thm]{Corollary}
\newtheorem*{prop*}{Proposition}

\theoremstyle{definition}
\newtheorem{example}[thm]{Example}
\newtheorem{rmk}[thm]{Remark}

\numberwithin{equation}{section}

\title[Small eigenvalues of the Laplacian in moduli space]
{Small eigenvalues of the Laplacian for algebraic measures in moduli space, and mixing properties of the Teichm\"uller flow}

\author{Artur Avila and S\'ebastien Gou\"ezel}

\address{CNRS UMR 7586, Institut de Math\'ematiques de Jussieu,
175 rue du Chevaleret, 75013, Paris, France \&
IMPA, Estrada Dona Castorina 110, 22460-320, Rio de Janeiro, Brazil}
\email{artur@math.jussieu.fr}
\address{IRMAR, CNRS UMR 6625,
Universit\'e de Rennes 1, 35042 Rennes, France}
\email{sebastien.gouezel@univ-rennes1.fr}

\date{November 24, 2010}

\begin{document}

\begin{abstract}
We consider the $\SL$ action on moduli spaces of quadratic
differentials.  If $\mu$ is an $\SL$-invariant probability
measure, crucial information about the associated
representation on $L^2(\mu)$ (and in particular, fine
asymptotics for decay of correlations of the diagonal action,
the Teichm\"uller flow) is encoded in the part of the spectrum of
the corresponding foliated hyperbolic Laplacian
%(acting on $\SO(2,\R)$-invariant functions)
that lies in $(0,1/4)$ (which controls the contribution of the
complementary series).  Here we prove that the essential spectrum of
an invariant algebraic measure is contained in $[1/4,\infty)$, i.e.,
for every $\delta>0$, there are only finitely many eigenvalues
(counted with multiplicity) in $(0,1/4-\delta)$. In particular, all
algebraic invariant measures have a spectral gap.
\end{abstract}

\maketitle

\section{Introduction}

For any lattice $\Gamma \subset \SL$, the irreducible decomposition of
the unitary representation of $\SL$ on $L^2(\SL/\Gamma)$ consists almost
entirely of tempered representations (with fast decay of matrix
coefficients): only finitely many non-tempered representations may
appear, each with finite multiplicity.  This corresponds to the well
known result of Selberg (see, e.g.,~\cite{Iwaniec})
that in an hyperbolic surface of finite volume,
the Laplacian has only finitely many eigenvalues, with finite
multiplicity, in $(0,1/4)$.  This has several remarkable consequences,
for instance, on the asymptotics of the number of closed geodesics,
the main error terms of which come from the small eigenvalues of the
Laplacian (by Selberg's trace formula, see~\cite{hejhal_2}), or for the
asymptotics of the correlations of smooth functions under the
diagonal flow~\cite{ratner}.

For a more general ergodic action of $\SL$, the situation can be much
more complicated: in general, one may even not have a spectral gap
($\SL$ does not have Kazhdan's property $(T)$).  Even in the
particularly nice situation of the $\SL$ action on a homogeneous
space $G/\Gamma$ with $G$ a semi-simple Lie group containing $\SL$
and $\Gamma$ an irreducible lattice in $G$ (a most natural
generalization the case $G=\SL$ above), non-tempered representations
may have a much heavier contribution: for instance, \cite[Theorem
1]{sarnak_strong_spectral_gap} constructs examples (with $G=
\SL\times \SU$) where the spectrum of the foliated (along $\SO \backslash \SL$
orbits) Laplacian on $\SO \backslash G/\Gamma$ has an accumulation point in
$(0,1/4)$.  In fact, whether there is always a spectral gap at all
remains an open problem for $G=\SL \times \SU$.  While one does expect better
behavior in the case where $G$ has no compact factor, it too remains far
from fully understood.
%Although this should
%probably not happen when $G$ has no compact factor, this question is
%still widely open.

Moduli spaces of quadratic differentials present yet another natural
generalization of $\SL/\Gamma$, with different challenges.  Let $g,n \geq 0$
with $3g-3+n>0$, let $\MM_{g,n}$
be the moduli space of quadratic differentials on a genus one Riemann
surface with $n$ punctures, and with at most simple poles at the punctures
(alternatively, it is the cotangent bundle of the moduli space of Riemann
surfaces),
and let $\MM^1_{g,n} \subset \MM_{g,n}$ be the subspace of area one
quadratic differentials.  There is a natural $\SL$ action on $\MM^1_{g,n}$,
which has been intensively studied, not least because the corresponding
diagonal action gives the Teichm\"uller geodesic flow.  If $g=0$ and $n=4$
or if $g=1$ and $n=1$, $\MM^1_{g,n}$ turns out to be of the form
$\SL/\Gamma$.  In higher genus the $\MM^1_{g,n}$ are not homogeneous spaces,
and it is rather important to understand to which extent they may still behave
as such.

Recall that $\MM_{g,n}$ is naturally stratified by the ``combinatorial
data'' of the quadratic differential $q$ (order of zeros, number of poles, and
whether or not $q$ is a square of an Abelian differential).  Each stratum
has a natural complex affine structure, though it is not necessarily
connected, the (finitely many) connected
components having been classified by
Kontsevich-Zorich~\cite{Kontsevich_Zorich} and
Lanneau~\cite{Lanneau_connected}.
Each connected component $\CC$ carries a unique (up to scaling) finite
invariant measure $\mu$ which is $\SL$ invariant and
absolutely continuous with respect to $\CC
\cap \MM^1_{g,n}$ (in case of
the largest, ``generic'', stratum, which is connected, $\mu$ coincides with
the Liouville measure in $\MM^1_{g,n}$).  Those measures were constructed,
and shown to be ergodic, by Masur~\cite{Masur} and Veech~\cite{Veech_flow}.
In~\cite{AGY_teich} and~\cite{avila_resende}, it is shown that for
such a {\it Masur-Veech
measure} $\mu$ the $\SL$ action on $L^2(\mu)$ has a spectral gap.

There are many more ergodic $\SL$ invariant measures beyond the Masur-Veech
measures, which can be expected to play an important role in the analysis of
non-typical $\SL$ orbits (the consideration of non-typical orbits arises, in
particular, when studying billiards in rational polygons).  While all such
measures have not yet been
classified, it has been recently announced by Eskin and Mirzakhani that they are
all ``algebraic'',\footnote {Here we use the term algebraic in a rather lax
sense.  What has actually been shown is that the corresponding $\GL$ invariant
measure is supported on an affine submanifold of some stratum, along which
it is absolutely continuous (with locally constant density in affine charts).}
a result analogue to one of Ratner's Theorems (classifying $\SL$ invariant
measures in an homogeneous space~\cite{Ratner_sl}).
(For squares of Abelian differentials in
$\MM_{2,0}$, a stronger version of this result, including the classification
of the algebraic invariant measures,
was obtained earlier by McMullen \cite {mcmullen}.)

\medskip

Let $\mu$ be an algebraic $\SL$-invariant measure in some
$\MM^1_{g,n}$. Our goal in this paper is to see to what extent the
action of $\SL$ on $L^2(\mu)$ looks like an action on an homogeneous
space, especially concerning small eigenvalues of the associated
Laplacian acting on the subspace of $\SO$ invariant functions in
$L^2(\mu)$.  Our main theorem states that the situation is almost
identical to the $\SL/\Gamma$ case (the difference being that we are
not able to exclude the possibility that the eigenvalues accumulate
at $1/4$):

\begin{mainthm}
Let $\mu$ be an $\SL$-invariant algebraic probability measure
in the moduli space of quadratic differentials. For any
$\delta>0$, the spectrum of the associated Laplacian in
$[0,1/4-\delta]$ is made of finitely many eigenvalues, of
finite multiplicity.
\end{mainthm}

This theorem can also be formulated as follows: in the
decomposition of $L^2(\mu)$ into irreducible components, the
representations of the complementary series occur only
discretely, with finite multiplicity. More details are given in
the next section.

Our result is independent of the above mentioned theorem of
Eskin and Mirzakhani.
With their theorem, we obtain that our result in fact applies to all
$\SL$-invariant probability measures.

As mentioned before, the spectral gap (equivalent to the
absence of spectrum in
$(0,\epsilon)$ for some $\epsilon>0$) had been previously established in the
particular case of Masur-Veech measures (\cite {AGY_teich}, \cite
{avila_resende}),
but without any control of the
spectrum beyond a neighborhood of $0$ (which moreover degenerates
as the genus increases).  Here we not only obtain very detailed information
of the spectrum up to the $1/4$ barrier (beyond which the statement is
already false even for the modular surface $\MM_{1,1}$), but manage to
address all algebraic measures, even in the absence of a classification.
This comes from the implementation of a rather different, geometric
approach, in contrast with the combinatorial one used to establish
the spectral gap for Masur-Veech measures (heavily dependent on
the precise combinatorial description, in terms of Rauzy diagrams,
of the Teichm\"uller flow restricted to connected components of stratum).

\medskip

An interesting question is whether there are indeed eigenvalues in
$(0,1/4)$. It is well known that there is no such eigenvalue in
$\SL/\Gamma$ for $\Gamma = \SLZ$, and by
Selberg's Conjecture~\cite{Selberg} the
situation should be the same for any congruence subgroup. It is
tempting to conjecture that, in our non-homogeneous situation, there
is no eigenvalue either, at least when $\mu$ is the Masur--Veech
measure. We will however refrain from doing so since we have no
serious evidence in one direction or the other. Let us note however
that, for some measures $\mu$, there are indeed eigenvalues: for any
finite index subgroup $\Gamma$ of the congruence subgroup $\Gamma(2)$
containing $\{\pm 1\}$, the curve $\SL/\Gamma$ can be realized as a
Teichm\"uller curve by \cite{every_curve_is_teich}. Suitably choosing
$\Gamma$
% since $\Gamma(2)$ is a free group generated by a,b, one can take
% an increasing sequence
% of cyclic covers, in which the smallest eigenvalue tends to $0$.
and taking for $\mu$ the Liouville measure on the resulting Teichm\"uller
curve, we get an example with eigenvalues.  Notice that this shows
indeed that there can be no uniform spectral gap for all algebraic measures
in all moduli spaces (it is unknown whether there is a
uniform spectral gap in each fixed moduli space).

\medskip

A consequence of our main theorem is that the correlations of well
behaved functions have a nice asymptotic expansion (given by the
spectrum of the Laplacian). For instance, if $f_1$ and $f_2$ are
square-integrable $\SO$-finite functions (i.e., $f_1$ and $f_2$ have
only finitely many nonzero Fourier coefficients for the action of
$\SO$), then their correlations $\int f_1\cdot f_2\circ g_t \dd \mu$ with
respect to the Teichm\"uller flow $g_t=\smat {e^t} 0 0 {e^{-t}}$ can be
written, for every $0<\delta<1$,
as $\sum_{i=0}^{M-1} c_i(f_1,f_2) e^{-a_i t} +
o(e^{-(1-\delta)t})$, where $0=a_0<\dots < a_{M-1} \leq 1-\delta$ are
the numbers $1-\sqrt{1-4\lambda}$ for $\lambda$ an eigenvalue of
$\Delta$ in $[0, (1-\delta^2)/4]$. This follows at once from the
asymptotic expansion of matrix coefficients of $\SO$-finite functions
in \cite[Theorem 5.6]{casselman_milicic}. A similar expansion
certainly holds if $f_1$ and $f_2$ are only compactly supported
$C^\infty$ functions, but its proof would require more detailed
estimates on matrix coefficients.

\medskip

We expect that our techniques will also be useful in the study of
the Ruelle zeta function $\zeta_{\mathrm{Ruelle}}(z)=\prod_\tau
(1-e^{-z |\tau|})$ (where $\tau$ runs over the prime closed orbits of
the flow $g$, and $|\tau|$ is the length of $\tau$).  Recall that
$\zeta_{\mathrm{Ruelle}}(z)$ can be expressed
as an alternating product $\prod \zeta_k(z)^{(-1)^k}$, where
$\zeta_k$ is a dynamical zeta function related to the action of $g_t$
on the space of $k$-forms (see for instance \cite{fried_analytique}).  Along
the proof of the main theorem, we obtain considerable information for the
action of the Teichm\"uller flow in suitably defined Banach spaces,
which goes in the direction of providing
meromorphic extensions of the functions $\zeta_k$ (and therefore also
of the Ruelle zeta functions), hence opening the way to precise asymptotic
formulas (which should include correction terms coming from small eigenvalues
of the Laplacian) for the number of
closed geodesics in the support of any algebraic invariant measure.

\section{Statements of results}

Our results will be formulated in moduli spaces of flat surfaces, as
follows. Fix a closed surface $S$ of genus $g\geq 1$, a subset
$\Sigma=\{\sigma_1,\dots,\sigma_j\}$ of $S$ and multiplicities
$\kappa=(\kappa_1,\dots, \kappa_j)$ with $\sum (\kappa_i -1) = 2g-2$.
We denote by $\Teich = \Teich(S, \Sigma, \kappa)$ the set of
translation structures on $S$ such that the cone angle around each
$\sigma_i$ is equal to $2\pi\kappa_i$, modulo isotopy. Equivalently,
this is the space of abelian differentials with zeroes of order
$\kappa_i-1$ at $\sigma_i$. Let also $\Teich_1\subset \Teich$ be the
set of area one surfaces.

Given a translation surface $x$, one can develop closed paths
(or more generally paths from singularity to singularity) from
the surface to $\C$, using the translation charts. This defines
an element $\Phi(x)\in H^1(M,\Sigma;\C)$. The resulting
\emph{period map} $\Phi: \Teich \to H^1(M,\Sigma; \C)$ is a
local diffeomorphism, and endows $\Teich$ with a canonical
complex affine structure.

The mapping class group $\Gamma$ of $(S,\Sigma,\kappa)$ is the group
of homeomorphisms of $S$ permuting the elements of $\Sigma$ with the
same $\kappa_i$. It acts on $\Teich$ and on $\Teich_1$. The space
$\Teich$ is also endowed with an action of $\GL$, obtained by
postcomposing the translation charts by $\GL$ elements. The action of
the subgroup $\SL$ of $\GL$ leaves $\Teich_1$ invariant. Since the
actions of $\GL$ and $\Gamma$ commute, we may write the former on the
left and the latter on the right.

\begin{definition}
\label{def_proj_measure}
A measure $\tilde\mu$ on $\Teich_1$ is admissible if it
satisfies the following conditions:
\begin{itemize}
\item The measure $\tilde\mu$ is $\SL$ and $\Gamma$-invariant.
\item There exists a $\Gamma$-invariant linear submanifold $Y$ of $\Teich$ such
that $\tilde\mu$ is supported on $X=Y\cap \Teich_1$, and
the measure $\tilde\mu\otimes \Leb$ on $X\times \R_+^* =Y$
is locally a multiple of the linear Lebesgue measure on
$Y$.
\item The measure $\mu$ induced by $\tilde\mu$ on $X/\Gamma$ has
finite mass, and is ergodic under the action of $\SL$ on
$X/\Gamma$.
\end{itemize}
\end{definition}

Although this definition may seem quite restrictive, it follows from
the above mentioned theorem of Eskin and Mirzakhani that
ergodic $\SL$-invariant
measures are automatically admissible. The following proposition is
much weaker, but we nevertheless include it since its proof is
elementary, and is needed to obtain further information on admissible
measures (in particular on their local product structure, see
Proposition \ref{prop_local_product} below).

\begin{prop}
\label{prop_smooth_measure_is_projective}
Let $X$ be a $\Gamma$-equivariant $C^1$ submanifold of
$\Teich_1$ such that $X/\Gamma$ is connected, and let
$\tilde\mu$ be a $\SL$ and $\Gamma$-invariant measure on $X$
such that $\tilde\mu$ is equivalent to Lebesgue measure, and
the induced measure $\mu$ in $X/\Gamma$ is a Radon measure,
i.e., it gives finite mass to compact subsets of $X/\Gamma$.
Then $\tilde\mu$ is admissible.
\end{prop}
This proposition should be compared to a result of Kontsevich
and M\"oller in \cite{moller_linear}: any $\GL$-invariant
algebraic submanifold of $\Teich$ is linear. Here, we obtain
the same conclusion if $X$ is only $C^1$, but we additionally
assume the existence of an invariant absolutely continuous
Radon measure on $X$.

\medskip

Let $\tilde\mu$ be an admissible measure, supported by a
submanifold $X$ of $\Teich_1$. Every $\SL$-orbit in $X/\Gamma$
is isomorphic to a quotient of $\SL$. Therefore, the image of
every such orbit in $\SO \backslash X /\Gamma$ (the set of
translations surfaces in $X$, modulo the mapping class group,
and in which the vertical direction is forgotten) is a quotient
of the hyperbolic plane, and is canonically endowed with the
hyperbolic Laplacian. Gluing those operators together on the
different orbits, we get a Laplacian $\Delta$ on $\SO
\backslash X /\Gamma$, which acts (unboundedly) on $L^2( \SO
\backslash X /\Gamma, \mu)$, where $\mu$ is $\tilde\mu$ mod
$\Gamma$. Our main theorem describes the spectrum of this
operator:

\begin{thm}
\label{global_main_thm}
Let $\tilde\mu$ be an admissible measure, supported by a
manifold $X$. Denote by $\mu$ the induced measure on
$X/\Gamma$. Then, for any $\delta>0$, the spectrum of the
Laplacian $\Delta$ on $L^2( \SO \backslash X /\Gamma, \mu)$,
intersected with $(0,1/4-\delta)$, is made of finitely many
eigenvalues of finite multiplicity.
\end{thm}

This theorem can also be formulated in terms of the spectrum of
the Casimir operator, or in terms of the decomposition of
$L^2(X/\Gamma,\mu)$ into irreducible representations under the
action of $\SL$: for any $\delta>0$, there is only a finite
number of representations in the complementary series with
parameter $u\in (\delta,1)$ appearing in this decomposition,
and they have finite multiplicity. See \S
\ref{subsec_background_SL} for more details on these notions
and their relationships.

\begin{rmk}
We have formulated the result in the space $X/\Gamma$ where $\Gamma$
is the mapping class group. However, if $\Gamma'$ is a subgroup of
$\Gamma$ of finite index, then the proof still applies in $X/\Gamma'$
(of course, there may be more eigenvalues in $X/\Gamma'$ than in
$X/\Gamma$). This applies for instance to $\Gamma'$ the set of
elements of $\Gamma$ that fix each singularity $\sigma_i$.
\end{rmk}

\begin{rmk}
In compact hyperbolic surfaces, the spectrum of the Laplacian is
discrete. Therefore, the essential spectrum of the Laplacian in
$[1/4,\infty)$ in finite volume hyperbolic surfaces comes from
infinity, i.e., the cusps. Since the geometry at infinity of moduli
spaces of flat surfaces is much more complicated than cusps, one
might expect more essential spectrum to show up, and Theorem
\ref{global_main_thm} may come as a surprise. However, from the point
of view of measure, infinity has the same weight in hyperbolic
surfaces and in moduli spaces: the set of points at distance at least
$H$ in a cusp has measure $\sim c H^{-2}$, while its analogue in a
moduli space is the set of surfaces with systole at most $H^{-1}$,
which also has measure of order $c' H^{-2}$ by the Siegel-Veech
formula \cite{eskin_masur}. This analogy (which also holds for
recurrence speed to compact sets) justifies heuristically
Theorem~\ref{global_main_thm}.
\end{rmk}

\smallskip

\emph{Quadratic differentials.}
Let $g,n \geq 0$ be integers such that $3g-3+n>0$ and let
$\boT_{g,n}$ be the Teichm\"uller space of Riemann surfaces of genus
$g$ with $n$ punctures. Its cotangent space is the space $\boQ_{g,n}$
of quadratic differentials with at most simple poles at the
punctures.  It is stratified by fixing some appropriate combinatorial
data (the number of poles, the number of zeros of each given order,
and whether the quadratic differential is a square of an Abelian
differential or not).  Much of the theory of quadratic differentials
is parallel to the one of Abelian differentials, in particular, each
stratum in $\boQ_{g,n}$ can be seen as a Teichm\"uller space
$\widetilde \Teich= \widetilde \Teich(\tilde S,\tilde \Sigma,\tilde
\kappa)$ of {\it half-translation structures}, which allows one to
define a natural action of $\GL$.  Moreover, strata are endowed with
a natural affine structure, which allows one to define the notion of
admissible measure (in particular, the Liouville measure in
$\boQ_{g,n}$ is admissible). Thus the statement of Theorem
\ref{global_main_thm} still makes sense in the setting of quadratic
differentials.  As it turns out, it can also be easily derived from
the result about Abelian differentials.

This is most immediately seen for strata of squares, in which case
$\widetilde \Teich$ is the quotient of a Teichm\"uller space of Abelian
differentials $\Teich$ by an involution (the rotation of angle
$\pi$). Taking the quotient by $\SO$, we see that the spectrum of the
Laplacian for some $\SL$-invariant measure in $\widetilde \Teich$ is
the same as the one for its (involution-symmetric) lift to $\Teich$,
to which Theorem \ref{global_main_thm} applies.

Even when $\widetilde \Teich$ is not a stratum of squares, it can
still be analyzed in terms of certain Abelian differentials (the
well-known double cover construction also used in
\cite{avila_resende}).  Indeed in this case the Riemann surface with
a quadratic differential admits a (holomorphic, ramified, canonical)
connected double cover (constructed formally using the doubly-valued
square-root of the quadratic differential), to which the quadratic
differential lifts to the square of an (also canonical) Abelian
differential.  This double cover carries an extra bit of information,
in the form of a canonical involution, so that $\widetilde \Teich$
gets identified with a Teichm\"uller space of ``translation surfaces
with involution''.  Forgetting the involution, the latter can be seen
as an affine subspace of a Teichm\"uller space of translation surfaces,
allowing us to apply Theorem \ref{global_main_thm}.

\smallskip

\emph{Notations.}
Let us introduce notations for convenient elements of $\SL$. For
$t\in \R$, let $g_t=\smat{e^t}{0}{0}{e^{-t}}$. Its action on $\boQ_g$
is the geodesic flow corresponding to the Teichm\"uller distance on
$\boT_g$, and its action in different strata (that we still call the
Teichm\"uller flow) will play an essential role in the proof of our
main theorem. We also denote $h_r=\smat 1 r 0 1$ and $\tilde
h_r=\smat 1 0 r 1$ the horocycle actions, and $k_\theta = \smat{\cos
\theta}{\sin \theta}{-\sin\theta}{\cos\theta}$ the circle action.
Throughout this article, the letter $C$ denotes a constant whose
value is irrelevant and can change from line to line.

\smallskip

\emph{Sketch of the proof.}
The usual strategy to prove that the spectrum of the Laplacian is
finite in $[0,1/4]$ in a finite volume surface $S=\SO\backslash
\SL/\Gamma$ is the following: one decomposes $L^2(S)$ as
$L^2_{\text{cusp}}(S)\oplus L^2_{\text{eis}}(S)$ where
$L^2_{\text{cusp}}(S)$ is made of the functions whose average on all
closed horocycles vanishes, and $L^2_{\text{eis}}(S)$ is its
orthogonal complement. One then proves that the spectrum in
$L^2_{\text{eis}}(S)$ is $[1/4,\infty)$ by constructing a basis of
eigenfunctions using Eisenstein series, and that the spectrum in
$L^2_{\text{cusp}}(S)$ is discrete since convolution with smooth
compactly supported functions in $\SL$ is a compact operator.

There are two difficulties when trying to implement this strategy in
nonhomogeneous situations. Firstly, since the geometry at infinity is
very complicated, it is not clear what the good analogue of
$L^2_{\text{eis}}(S)$ and Eisenstein series would be. Secondly, the
convolution with smooth functions in $\SL$ only has a smoothing
effect in the direction of the $\SL$ orbits, and not in the
transverse direction (and this would also be the case if one directly
tried to study the Laplacian); therefore, it is very unlikely to be
compact.

To solve the first difficulty, we avoid completely the decomposition
into Eisenstein and cuspidal components and work in the whole $L^2$
space. This means that we will not be able to exhibit compact
operators (since this would only yield discrete spectrum), but we
will rather construct quasi-compact operators, i.e., operators with
finitely many large eigenvalues and the rest of the spectrum
contained in a small disk. The first part will correspond to the
spectrum of the Laplacian in $[0,1/4-\delta]$ and the second part to
the non-controlled rest of the spectrum.

Concerning the second difficulty, we will not study the Laplacian nor
convolution operators, but another element of the enveloping algebra:
the differentiation $L_\omega$ in the direction $\omega$ of the flow
$g_t$. Of course, its behavior on the space $L^2(X/\Gamma,\mu)$ is
very bad, but we will construct a suitable Banach space $\boB$ of
distributions on which it is quasi-compact. To relate the spectral
properties of $g_t$ on $\boB$ and of $\Delta$ on $L^2$, we will rely
on fine asymptotics of spherical functions in irreducible
representations of $\SL$ (this part is completely general and does
not rely on anything specific to moduli spaces of flat surfaces).

The main difficulty of the article is the construction of $\boB$ and
the study of $L_\omega$ on $\boB$. We rely in a crucial way on the
hyperbolicity of $g_t$, that describes what happens in all the
directions of the space under the iteration of the flow. If $\boB$ is
carefully tuned (its elements should be smooth in the stable
direction of the flow, and dual of smooth in the unstable direction),
then one can hope to get smoothing effects in every direction, and
therefore some compactness. This kind of arguments has been developed
in recent years for Anosov maps or flows in compact manifolds and has
proved very fruitful (see among others \cite{liverani_contact,
gouezel_liverani, bt_aniso, butterley_liverani}). We use in an
essential way the insights of these papers. However, the main
difficulty for us is the non-compactness of moduli space: since we
can not rely on an abstract compactness argument close to infinity,
we have to get explicit estimates there (using a quantitative
recurrence estimate of Eskin-Masur \cite{eskin_masur}). We should
also make sure that the estimates do not diverge at infinity.
Technically, this is done using the Finsler metric of
Avila-Gou\"ezel-Yoccoz \cite{AGY_teich} (that has good regularity
properties uniformly in the Teichm\"uller space) to define the Banach
space $\boB$, and plugging the Eskin-Masur function $V_\delta$ into
the definition of $\boB$. On the other hand, special features of the
flow under study are very helpful: it is affine (hence no distortion
appears), and its stable and unstable manifolds depend smoothly on
the base point and are affine. Moreover, it is endowed in a $\SL$
action, which implies that its spectrum can not be arbitrary:
contrary to \cite{liverani_contact}, we will not need to investigate
spectral values with large imaginary part.

Let us quickly describe a central step of the proof. At some point,
we need to study the iterates $\boL_{T_0}^n$ of the operator
$\boL_{T_0}f = f\circ g_{T_0}$, for a suitably chosen $T_0$. Using a
partition of unity, we decompose $\boL_{T_0}$ as $\tilde\boL_1 +
\tilde\boL_2$ where $\tilde\boL_1$ corresponds to what is going on in
a very large compact set $K$, and $\tilde\boL_2$ takes what happens
outside $K$ into account. We expand $\boL_{T_0}^n =\sum_{\gamma_i\in
\{1,2\}} \tilde \boL_{\gamma_1}\cdots \tilde\boL_{\gamma_n}$. In this
sum, if most $\gamma_i$s are equal to $2$, we are spending a lot of
time outside $K$, and the Eskin-Masur function gives us a definite
gain. Otherwise, a definite amount of time is spent inside $K$, where
the flow is hyperbolic, and we get a gain $\lambda$ given by the
hyperbolicity constant of the flow inside $K$. Unfortunately, we only
know that $\lambda$ is strictly less than $1$ (and $K$ is very large,
so it is likely to be very close to $1$). This would be sufficient to
get a spectral gap, but not to reach $1/4$ in the spectrum of the
Laplacian. A key remark is that, if we define our Banach space
$\boB$ using $C^k$ regularity, then the gain is better, of order
$\lambda^{k}$. Choosing $k$ large enough (at the complete end of the
proof), we get estimates as precise as we want, getting arbitrarily
close to $1/4$.

In view of this argument, two remarks can be made. Firstly,
since we need to use very high regularity, our proof can not be
done using a symbolic model since the discontinuities at the
boundaries would spoil the previous argument. Secondly, since
$k$ is chosen at the very end of the proof, we have to make
sure that all our bounds, which already have to be uniform in
the non-compact space $X/\Gamma$, are also uniform in $k$.

\medskip

The paper is organized as follows. In Section
\ref{conclusive_section}, we introduce necessary background on
irreducible unitary representations of $\SL$, and show that
Theorem~\ref{global_main_thm} follows from a statement on
spectral properties of the differentiation $L_\omega$ in the
flow direction (Theorem
\ref{thm_control_spectral_radius_annonce}). In Section
\ref{section_local_product}, we get a precise description of
admissible measures, showing that they have a nice local
product structure. Along the way, we prove
Proposition~\ref{prop_smooth_measure_is_projective}. In Section
\ref{sec_Finsler}, we establish several technical properties of
the $C^k$ norm with respect to the Finsler metric of
\cite{AGY_teich} that will be instrumental when defining our
Banach space $\boB$. In Section \ref{section_recurrence}, we
reformulate the recurrence estimates of Eskin-Masur
\cite{eskin_masur} in a form that is convenient for us.
Finally, we define the Banach space $\boB$ in Section
\ref{par_construct_norms}, and prove Theorem
\ref{thm_control_spectral_radius_annonce} in Section
\ref{sec_first_step}.

\section{Proof of the main theorem: the general part}
\label{conclusive_section}

\subsection{Functional analytic prerequisites}

Let $\boLo$ be a bounded operator on a complex Banach space
$(\Bo,\norm{\cdot})$. A complex number $z$ belongs to the spectrum
$\sigma(\boLo)$ of $\boLo$ if $zI-\boLo$ is not invertible. If $z$ is
an isolated point in the spectrum of $\boLo$, we can define the
corresponding spectral projection $\Pi_z\coloneqq
\frac{1}{2\ic\pi}\int_C (wI-\boLo)^{-1} \dd w$, where $C$ is a small
circle around $z$ (this definition is independent of the choice of
$C$). Then $\Pi_z$ is a projection, its image and kernel are
invariant under $\boLo$, and the spectrum of the restriction of
$\boLo$ to the image is $\{z\}$, while the spectrum of the
restriction of $\boLo$ to the kernel is $\sigma(\boLo)-\{z\}$. We say
that $z$ is an isolated eigenvalue of finite multiplicity of $\boLo$
if the image of $\Pi_z$ is finite-dimensional, and we denote by
$\sigmaess(\boLo)$ the essential spectrum of $\boLo$, i.e., the set
of elements of $\sigma(\boLo)$ that are not isolated eigenvalues of
finite multiplicity.

The spectral radius of $\boLo$ is $r(\boLo)\coloneqq \sup\{|z| \st
z\in \sigma(\boLo)\}$, and its essential spectral radius is
$\ress(\boLo) \coloneqq \sup\{|z| \st z\in \sigmaess(\boLo)\}$. These
quantities can also be computed as follows: $r(\boLo) = \inf_{n\in
\N} \norm{\boLo^n}^{1/n}$, and $\ress(\boLo) = \inf \norm{\boLo^n
-K}^{1/n}$, where the infimum is over all integers $n$ and all
compact operators $K$. In particular, we get that the essential
spectral radius of a compact operator is $0$, i.e., the spectrum of a
compact operator is made of a sequence of isolated eigenvalues of
finite multiplicity tending to $0$, as is well known.

So-called Lasota-Yorke inequalities can also be used to
estimate the essential spectral radius:
\begin{lem}
\label{lemme_hennion}
Assume that, for some $n>0$ and for all $x\in \Bo$, we have
  \begin{equation*}
  \norm{\boLo^n x} \leq M^n \norm{x} + \normdeux{x},
  \end{equation*}
where $\normdeux{\cdot}$ is a seminorm on $\Bo$ such that the unit ball of
$\Bo$ (for $\norm{\cdot}$) is relatively compact for $\normdeux{\cdot}$. Then
$\ress(\boLo) \leq M$.
\end{lem}
This has essentially been proved by Hennion in \cite{hennion},
the statement in this precise form can be found in \cite[Lemma
2.2]{BGK_coupling}.

Assume now that $\boL$ is a bounded operator on a complex
normed vector space $(B,\norm{\cdot})$, but that $B$ is not
necessarily complete. Then $\boL$ extends uniquely to a bounded
operator $\boLo$ on the completion $\Bo$ of $B$ for the norm
$\norm{\cdot}$. We will abusively talk about the spectrum,
essential spectrum or essential spectral radius of $\boL$,
thinking of the same data for $\boLo$.

\subsection{Main spectral result}
\label{subsec_main_result_first_step}

Let $\tilde\mu$ be an admissible measure supported on a manifold $X$,
and let $\mu$ be its projection in $X/\Gamma$.

We want to study the spectral properties of the differentiation
operator $L_\omega$ in the direction $\omega$ of the flow
$g_t$. As in \cite{liverani_contact}, it turns out to be easier
to study directly the resolvent of this operator, given by
$R(z) f = \int_{t=0}^\infty e^{-z t} f\circ g_t \dd t$.

Given $\delta>0$, we will study the operator $\boM=R(4\delta)$
on the space $\boD^\Gamma$ of $C^\infty$ functions on $X$,
$\Gamma$--invariant and compactly supported in $X/\Gamma$. Of
course, $\boM f$ is not any more compactly supported, so we
should be more precise.

We want to define a norm $\norm{\cdot}$ on $\boD^\Gamma$ such
that, for any $f \in \boD^\Gamma$, the function $f\circ g_t$
(which still belongs to $\boD^\Gamma$) satisfies $\norm{f \circ
g_t} \leq C \norm{f}$, for some constant $C$ independent of
$t$. Denoting by $\boDb$ the completion of $\boD^\Gamma$ for
the norm $\norm{\cdot}$, the operator $\boL_t : f\mapsto f\circ
g_t$ extends continuously to an operator on $\boDb$, whose norm
is bounded by $C$. Therefore, the operator $\boM \coloneqq
\int_{t=0}^\infty e^{-4\delta t} \boL_t$ acts continuously on
the Banach space $\boDb$, and it is meaningful to consider its
essential spectral radius. We would like this essential
spectral radius to be quite small. Since $\norm{f\circ g_t}\leq
C \norm{f}$, the trivial estimate on the spectral radius of
$\boM$ is $C\int_{t=0}^\infty e^{-4\delta t}\dd t=
C/(4\delta)$. This blows up when $\delta$ tends to $0$. We will
get a significantly better bound on the essential spectral
radius in the following theorem.

\begin{thm}
\label{thm_control_spectral_radius_annonce}
There exists a norm on $\boD^\Gamma$ satisfying the requirement
$\norm{f \circ g_t} \leq C \norm{f}$ (uniformly in $f\in
\boD^\Gamma$ and $t\geq 0$), such that the essential spectral
radius of $\boM$ for this norm is at most $1+\delta$.

Moreover, for any $f_1 \in \boD^\Gamma$, the linear form
$f\mapsto \int_{X/\Gamma} f_1 f\dd\mu$ extends continuously
from $\boD^\Gamma$ to its closure $\boDb$.
\end{thm}
This theorem is proved in Section \ref{sec_first_step}. The
main point is of course the assertion on the essential spectral
radius, the last one is a technicality that we will need later
on.

Let us admit this result for the moment, and see how it implies
our main result, Theorem \ref{global_main_thm}. Since Theorem
\ref{thm_control_spectral_radius_annonce} deals with the
spectrum of $L_\omega$, it is not surprising that it implies a
description of the spectrum of the action of $\SL$. However, we
only control the spectrum of $L_\omega$ on a quite exotic
Banach space of distributions. To obtain information on the
action of $\SL$, we will therefore follow an indirect path,
through meromorphic extensions of Laplace transforms of
correlation functions. (It seems desirable to find a more
direct and more natural route.)

\subsection{Meromorphic extensions of Laplace transforms}

From Theorem \ref{thm_control_spectral_radius_annonce}, we will
obtain in this section a meromorphic extension of the Laplace
transform of the correlations of smooth functions, to a
suitable domain described as follows. For $\delta,\epsilon>0$,
define $D_{\delta,\epsilon}\subset \C$ as the set of points
$z=x+\ic y$ such that either $x>0$, or $(x,y)\in [-1+6\delta,0]
\times [-\epsilon,\epsilon]$.

\begin{prop}
\label{prop_mero_ext}
Let $\delta>0$. Let $f_1, f_2 \in \boD^\Gamma$, define for $\Re
z>0$ a function $F(z)=F_{f_1,f_2}(z)= \int_{t=0}^\infty e^{-zt}
\left(\int_{X/\Gamma}f_1\cdot f_2 \circ g_t\dd\mu\right) \dd
t$. Then, for some $\epsilon>0$, the function $F$ admits a
meromorphic extension to (a neighborhood of) $D_{\delta,\epsilon}$.

Moreover, the poles of $F$ in $D_{\delta,\epsilon}$ are located in
the set $\{4\delta-1/\lambda_1, \dots,4\delta-1/\lambda_I\}$, where
the $\lambda_i$ are the finitely many eigenvalues of modulus at least
$1+2\delta$ of the operator $\boM=R(4\delta)$ acting on the space
constructed in Theorem \ref{thm_control_spectral_radius_annonce}. The
residue of $F$ at such a point $4\delta-1/\lambda_i$ is equal to $
\int_{X/\Gamma} f_1 \cdot \Pi_{\lambda_i} f_2 \dd\mu$, where
$\Pi_{\lambda_i}$ is the spectral projection of $\boM$ associated to
$\lambda_i \in \sigma(\boM)$.
\end{prop}
\begin{proof}
Heuristically, we have $F(z)=\int_{X/\Gamma} f_1 R(z)
f_2\dd\mu$ where $R(z)=\int_{t=0}^\infty e^{-zt} f\circ g_t$,
and moreover $R(z)=(z-L_\omega)^{-1}$ where $L_\omega$ is the
differentiation in the direction $\omega$. Let us fix
$z_0=4\delta$. The spectral properties of
$\boM=R(z_0)=(z_0-L_\omega)^{-1}$ are well controlled by
Theorem \ref{thm_control_spectral_radius_annonce}. In view of
the formal identity
  \begin{equation*}
  (z-L_\omega)^{-1} = (z_0-z)^{-1} (z_0-L_\omega)^{-1} ( (z_0-z)^{-1} - (z_0-L_\omega)^{-1})^{-1},
  \end{equation*}
we are led to define an operator
  \begin{equation}
  \label{defS}
  S(z) = \frac{1}{z_0-z} \boM \left( \frac{1}{z_0-z} -\boM \right)^{-1},
  \end{equation}
which should coincide with $R(z)$. In particular, we should
have the equality $F(z)= \int_{X/\Gamma} f_1 S(z)f_2 \dd\mu$.
Since $S(z)$ is defined for a large set of values of $z$, this
should define the requested meromorphic extension of $F$ to a
larger domain.

Let us start the rigorous argument. Let $\boDb$ be the Banach space
constructed in Theorem \ref{thm_control_spectral_radius_annonce}, and
let $\lambda_1,\dots,\lambda_I$ be the finitely many eigenvalues of
modulus $\geq 1+2\delta$ of $\boM$ acting on $\boDb$. For $z$ with
$1/|z_0-z|\geq 1+2\delta$ and $1/(z_0-z) \not\in
\{\lambda_1,\dots,\lambda_I\}$, we can define on $\boDb$ an operator
$S(z)$ by the formula \eqref{defS}. It is holomorphic on
$D_{\delta,\epsilon}\backslash\{4\delta-1/\lambda_1,
\dots,4\delta-1/\lambda_I\}$. Since the points $4\delta -1/\lambda_i$
are poles of finite order (see e.g.~\cite[III.6.5]{kato_pe}), $S(z)$
is even meromorphic on $D_{\delta,\epsilon}$. Let us finally set
$G(z) = \int_{X/\Gamma} f_1 S(z)f_2 \dd\mu \in \C$, this is well
defined by the last statement in Theorem
\ref{thm_control_spectral_radius_annonce}. The function $G$ is
meromorphic and defined on the set $D_{\delta,\epsilon}$, with
possible poles at the points $z_0-1/\lambda_1,
\dots,z_0-1/\lambda_I$. To conclude, we just have to check that $F$
and $G$ coincide in a neighborhood of $z_0$.

If $z$ is very close to $z_0$, $1/(z_0-z)$ is very large so
that all series expansions are valid. Then the formula
\eqref{defS} gives
  \begin{equation*}
  S(z) f_2 = \boM (1 - (z_0-z) \boM)^{-1} f_2
  = \sum_{k=0}^\infty (z_0-z)^k \boM^{k+1} f_2.
  \end{equation*}
Since $\boM^{k+1} f = \int_{t=0}^\infty \frac{t^{k}}{k!}
e^{-z_0 t} f\circ g_t \dd t$, we obtain
  \begin{equation*}
  S(z)f_2 = \int_{t=0}^\infty \sum_{k=0}^\infty (z_0-z)^k \frac{t^k}{k!} e^{-z_0 t} f\circ g_t \dd t
  = \int_{t=0}^\infty e^{(z_0-z)t} e^{-z_0 t} f_2 \circ g_t \dd t.
  \end{equation*}
This gives the desired result after multiplying by $f_1$ and
integrating.

Let us now compute the residue of $S(z)f_2$ around a point
$z_0-1/\lambda_i$. We have
  \begin{align*}
  S(z) & = \frac{1}{z_0-z} \left(\frac{1}{z_0-z} + \boM -\frac{1}{z_0-z} \right)
    \left( \frac{1}{z_0-z} -\boM \right)^{-1}
  \\&
  =\frac{1}{(z_0-z)^2} \left( \frac{1}{z_0-z} -\boM \right)^{-1} - \frac{1}{z_0-z}.
  \end{align*}
The term $-(z_0-z)^{-1}$ is holomorphic around
$z_0-1/\lambda_i$. Therefore, the residue of $S$ around this point is
given by
  \begin{multline*}
  \frac {1} {2\ic\pi} \int_{C(z_0-1/\lambda_i)}\frac{1}{(z_0-z)^2} \left( \frac{1}{z_0-z} -\boM \right)^{-1} \dd z
  \\
  = \frac {1} {2\ic\pi} \int_{C(\lambda_i)} w^2 (w-\boM)^{-1} \frac{\dd w}{w^2}
  = \Pi_{\lambda_i},
  \end{multline*}
where $C(u)$ denotes a positively oriented path around the point $u$
and we have written $w=1/(z_0-z)$. This concludes the proof.
\end{proof}

\subsection{Background on unitary representations of \texorpdfstring{$\SL$}{SL(2,R)}}

\label{subsec_background_SL}

Let us describe (somewhat informally) the notion of direct
decomposition of a representation. See e.g.~\cite{dixmier} for
all the details.

Let $H_\xi$ be a family of representations of $\SL$, depending
on a parameter $\xi$ in a space $\Xi$, and assume that this
family of representations is measurable (in a suitable sense).
If $m$ is a measure on $\Xi$, one can define the direct
integral $\int H_\xi \dd m(\xi)$: an element of this space is a
function $f$ defined on $\Xi$ such that $f(\xi) \in H_\xi$ for
all $\xi$, with $\norm{f}^2 \coloneqq \int
\norm{f(\xi)}^2_{H_\xi} \dd m(\xi) < \infty$. The group $\SL$
acts unitarily on this direct integral, by $(g\cdot f)(\xi)=g(
f(\xi))$. If $m'$ is another measure equivalent to $m$, then
the representations $\int H_\xi \dd m(\xi)$ and $\int H_\xi \dd
m'(\xi)$ are isomorphic.

From now on, let $\Xi$ be the space of all irreducible unitary
representations of $\SL$, with its canonical Borel structure (that we
will describe below). Any unitary representation $H$ of $\SL$ is
isomorphic to a direct integral $\int H_\xi \dd m(\xi)$, where the
space $H_\xi$ is a (finite or countable) direct sum of one or several
copies of the same representation $\xi$ (we say that $H_\xi$ is
quasi-irreducible). Moreover, the measure class of the measure $m$,
and the multiplicity of $\xi$ in $H_\xi$, are uniquely defined
(\cite[Th\'eor\`eme 8.6.6]{dixmier}), and the representation $H$ is
characterized by these data.

\smallskip

Let us now describe $\Xi$ more precisely. The irreducible
unitary representations of $\SL$ have been classified by
Bargmann, as follows. An irreducible unitary representation of
$\SL$ belongs to one of the following families:
\begin{itemize}
\item Representations $\boD_{m+1}^+$ and $\boD_{m+1}^-$, for $m \in
\N$. This is the discrete series (except for $m=0$, where
the situation is slightly different: these representations
form the ``mock discrete series'').
\item Representations $\boP^{+, \ic v}$ for $v\in [0,+\infty)$
and $\boP^{-,\ic v}$ for $v\in (0,\infty)$. This is the
principal series (these representations can also be defined
for $v<0$, but they are isomorphic to the same
representations with parameter $-v>0$).
\item Representations $\boC^u$ for $0<u<1$. This is the
complementary series.
\item The trivial representation.
\end{itemize}
These representations are described with more details in
\cite[II.5]{knapp}. They are all irreducible, no two of them
are isomorphic, and any irreducible unitary representation of
$\SL$ appears in this list. In particular, to any irreducible
representation $\xi$ of $\SL$ is canonically attached a complex
parameter $s(\xi)$ (equal to $m$ in the first case, $\ic v$ in
the second, $u$ in the third and $1$ in the fourth), and the
Borel structure of $\SL$ is given by this parameter (and the
discrete data $\pm$ in the first two cases).

\smallskip

The Casimir operator $\Omega$ is a generator of the center of
the enveloping algebra of $\SL$, i.e., it is a differential
operator on $\SL$, commuting with every translation, and of
minimal degree. It is unique up to scalar multiplication, and
we will normalize it as
  \begin{equation}
  \label{def_casimir}
  \Omega = (L_W^2 - L_\omega^2 - L_V^2) /4,
  \end{equation}
where $W = \smat 0 1 {-1} 0$, $\omega = \smat 1 0 0 {-1}$ and $V =
\smat 0 1 1 0$ are elements of the Lie algebra of $\SL$, and $L_Z$
denotes the Lie derivative on $\SL$ with respect to the left
invariant vector field equal to $Z$ at the identity.

The Casimir operator extends to an unbounded operator in every
unitary representation of $\SL$. Since it commutes with
translations, it has to be scalar on irreducible
representations. With the notations we have set up earlier, it
is equal to $(1-s(\xi)^2)/4 \in \R$ on an irreducible unitary
representation $\xi$ of parameter $s(\xi)$.

\smallskip

An irreducible unitary representation $\xi$ of $\SL$ is
\emph{spherical} if it contains an $\SO$-invariant non-trivial
vector. In this case, the $\SO$-invariant vectors have
dimension $1$, let $v$ be an element of unit norm in this set. The
spherical function $\phi_\xi$ is defined on $\SL$ by
  \begin{equation}
  \label{def_spherical}
  \phi_\xi(g)=\langle g\cdot v, v\rangle,
  \end{equation}
it is independent of the choice of $v$. Taking $g=g_t$, the
spherical function is simply the correlations of $v$ under the
diagonal flow.

The spherical unitary irreducible representations are the
representations $\boP^{+, \ic v}$ and $\boC^u$ (and the trivial
one, of course).

\smallskip

Assume now that $\SL$ acts on a space $Y$ and preserves a
probability measure $\mu$. Then $\SL$ acts unitarily on
$L^2(Y,\mu)$ by $g\cdot f(x)=f(g^{-1} x)$. Therefore, the
Casimir operator also acts $L^2(Y,\mu)$ (as an unbounded
operator). Since it commutes with translations, it leaves
invariant the space $L^2(\SO\backslash Y, \mu)$ (i.e., the
space of functions on $Y$ that are $\SO$-invariant and
square-integrable with respect to $\mu$). On this space,
$\Omega$ can also be described geometrically as a foliated
Laplacian, as follows.

For $x\in Y$, its orbit mod $\SO$ is identified with $\HH =
\SL/\SO$, by the map $g\SO \mapsto \SO g^{-1} x$ (and changing
the basepoint $x$ in the orbit changes the parametrization by
an  $\SL$ element). Therefore, any structure on $\HH$ which is
$\SL$ invariant can be transferred to $\SO\backslash Y$. This
is in particular the case of the hyperbolic metric of curvature
$-1$, and of the corresponding hyperbolic Laplacian $\Delta$
given in coordinates $(x_\HH,y_\HH) \in \HH$ by $-y_\HH \left(
\frac{\partial^2}{\partial x_\HH^2} +
\frac{\partial^2}{\partial y_\HH^2} \right)$.

Let $f_K$ be a function on $\SO \backslash Y$ belonging to the
domain of $\Delta$, and let $f$ be its canonical lift to $Y$.
Then $\Omega f$ is $\SO$-invariant, and is the lift of the
function $\Delta f_K$ on $\SO \backslash Y$. This follows at
once from the definitions (and our choice of normalization in
\eqref{def_casimir}).

Consider the decomposition $L^2(Y,\mu) \simeq \int_\Xi H_\xi
\dd m(\xi)$ of the representation of $\SL$ on $L^2(Y,\mu)$ into
an integral of quasi-irreducible representations. Denoting by
$H_\xi^{\SO}$ the $\SO$-invariant vectors in $H_\xi$, we have
$L^2(\SO\backslash Y, \mu)\simeq \int_\Xi H_\xi^{\SO} \dd
m(\xi)$. Therefore, the spectrum of $\Delta$ on
$L^2(\SO\backslash Y, \mu)$ is equal to the set $\{
(1-s(\xi)^2)/4\}$, for $\xi$ a spherical representation in the
support of $m$ (moreover, the spectral measure of $\Delta$ is
the image of $m$ under this map). Since the spectrum of the
Casimir operator in the interval $(0,1/4)$ only comes from the
complementary series representations, which are all spherical,
it follows that $\sigma( \Delta) \cap (0,1/4) = \sigma(\Omega)
\cap (0,1/4)$, and that the spectral measures coincide.
Therefore, it is equivalent to understand $\sigma( \Delta) \cap
(0,1/4)$ or to understand representations in the complementary
series arising in $L^2(Y,\mu)$. While the former point of view
is more elementary, the latter puts it in a larger (and,
perhaps, more significant) perspective.

%Consider $f \in L^2(K\backslash X,\mu)$ and $x\in X$. We define
%a function $\tilde f_x$ on $\HH = \SL/\SO$ by $\tilde f_x(g
%\SO) = f(\SO g^{-1} x)$ (this is well defined since $f$ is
%$\SO$ invariant. The hyperbolic Laplacian of $\tilde f_x$ at
%$\ic$ (the preimage of $x$ under the map $g\SO \mapsto \SO
%g^{-1}x$) is well defined, and we let $\Delta f(x)= \Delta_\HH
%\tilde f_x (\ic)$.

\subsection{Meromorphic extensions of Laplace transforms in
abstract \texorpdfstring{$\SL$}{SL(2,R)} representations}

Given $H$ a unitary representation of $\SL$, let us decompose
it as $\int_{\Xi} H_\xi \dd m(\xi)$ where $\Xi$ is the set of
unitary irreducible representations of $\SL$, and $H_\xi$ is a
direct sum of copies of the irreducible representation $\xi$.
For $f\in H$, we will denote by $f_\xi$ its component in
$H_\xi$.

Let us denote by $\Xi^\SO$ the set of spherical irreducible
unitary representations. Using the parameter $s$ of an
irreducible representation described in the previous section,
$\Xi^{\SO}$ is canonically in bijection with $(0,1] \cup \ic
[0,+\infty)$. We will denote by $\xi_s$ the representation
corresponding to a parameter $s$.

\begin{prop}
\label{prop_extension}
Let $f_1, f_2\in H$ be invariant under $\SO$. Let us define the
Laplace transform of the correlations of $f_1,f_2$  by
  \begin{equation*}
  F(z)=F_{f_1,f_2}(z) = \int_{t=0}^\infty e^{-zt} \langle g_t \cdot f_1, f_2\rangle,
  \end{equation*}
for $\Re(z)>0$.

The function $F$ admits an holomorphic extension to $\{
\Re(z)>-1, z\not\in (-1,0]\}$. Moreover, for every $\delta>0$,
the function $F$ can be written on the half-space
$\{\Re(z)>-1+2\delta\}$ as the sum of a bounded holomorphic
function $A_\delta$, and the function
  \begin{equation*}
  B_\delta(z)=\frac{1}{\sqrt{\pi}} \int_{s\in [\delta,1]} \frac{\Gamma(s/2)}{\Gamma((s+1)/2)}
  \langle (f_1)_{\xi_s}, (f_2)_{\xi_s} \rangle \frac{\dd m(\xi_s)}{z-s+1}.
  \end{equation*}
\end{prop}
\begin{proof}
We fix a decomposition of $H_\xi$ as an orthogonal sum
$\bigoplus_{0\leq i< n}\xi_i$, where $n=n(\xi)\in \N\cup\{+\infty\}$
is the multiplicity of $\xi$ in $H_\xi$, and $\xi_0,\dots,\xi_{n-1}$
are copies of the representation $\xi$. This decomposition is not
canonical, but it can be chosen to depend measurably on $\xi$ (see
\cite{dixmier}). If the decomposition $\xi$ is spherical, we fix in
every $\xi_j$ a vector $h(\xi,j) \in \xi_j$ of unit norm invariant
under $\SO$.

Let $f$ be a $\SO$-invariant element of $H$. For $\xi\in
\Xi^\SO$, the element $f_\xi$ of $H_\xi$ can uniquely be
decomposed as $\sum_{j<n(\xi)} \hat f(\xi,j) h(\xi,j)$, where
the coefficients $\hat f(\xi,j)\in \C$ depend measurably on
$\xi,j$.

We use this decomposition for $f_1$ and $f_2$. Let us recall that we
defined the spherical function $\phi_\xi$ of a representation $\xi$
in \eqref{def_spherical}. Since the functions $g_t\cdot h(\xi,j)$ and
$h(\xi,j')$ are orthogonal for $j\not=j'$, we have
  \begin{align*}
  \langle g_t \cdot &f_1,f_2 \rangle
  \\&
  = \int_{\Xi^\SO} \left\langle g_t \left(\sum_{j<n(\xi)} \hat f_1(\xi,j) h(\xi,j)\right),
	\left(\sum_{j'<n(\xi)} \hat f_2(\xi,j') h(\xi,'j)\right)\right\rangle \dd m(\xi)
  \\&
  = \int_{\Xi^\SO} \sum_{j<n(\xi)}\hat f_1(\xi,j) \overline{\hat f_2(\xi,j)}
	\langle g_t h(\xi,j), h(\xi,j) \rangle \dd m(\xi)
  \\&
  = \int_{\Xi^\SO} \left( \sum_{j<n(\xi)}\hat f_1(\xi,j) \overline{\hat f_2(\xi,j)}\right) \phi_\xi(g_t) \dd m(\xi)
  \\&
  = \int_{\Xi^\SO} \langle (f_1)_\xi, (f_2)_\xi \rangle \phi_\xi(g_t) \dd m(\xi).
  \end{align*}

To proceed, we will need fine asymptotics of the spherical
functions $\phi_\xi$. The first one is due to Ratner
\cite[Theorem 1]{ratner}: for all $\delta>0$, there exists a
constant $C$ such that, for any $\xi\in \Xi^\SO$ with
$s(\xi)\not\in  [\delta,1]$, and for any $t\geq 0$,
  \begin{equation}
  \label{ratner}
  |\phi_\xi(g_t)| \leq C e^{-(1-\delta)t}.
  \end{equation}
An important point in this estimate is that the constant $C$ is
uniform in $\xi$, even though $\xi$ varies in a non-compact
domain.

For representations in the complementary series, we will use a
more precise estimate, as follows. Define a function
  \begin{equation}
  \label{c(s)}
  c(s)=\frac{1}{\sqrt{\pi}}
  \frac{\Gamma(s/2)}{\Gamma((s+1)/2)},
  \end{equation}
for $s\in (0,1]$. This
function is known as Harish-Chandra's function. For all
$\delta>0$, there exists a constant $C>0$ such that, for all
$s\in [\delta,1]$ and all $t\geq 0$,
  \begin{equation}
  \label{harish}
  \left|\phi_{\xi_s} (g_t) - c(s)e^{(s-1)t}\right|
  \leq C e^{-t}.
  \end{equation}
This estimate is proved in Appendix \ref{appendix_cfunction}.

We will now conclude, using \eqref{ratner} and \eqref{harish}.
Let us decompose $\Xi^\SO$ (identified through the parameter
$s$ with a subset of $\C$) as the union of $[\delta,1]$ and its
complement. Then
  \begin{align*}
  F(z)&=\int_{\xi_s\in\Xi^\SO} \int_{t=0}^\infty e^{-zt} \langle (f_1)_{\xi_s}, (f_2)_{\xi_s} \rangle
  \phi_{\xi_s}(g_t) \dd t \dd m(\xi_s)
  \\&
  = \int_{s\in\Xi^\SO \setminus [\delta,1]} \int_{t=0}^\infty e^{-zt} \langle (f_1)_{\xi_s}, (f_2)_{\xi_s} \rangle
  \phi_{\xi_s}(g_t) \dd t \dd m(\xi_s)
  \\&\ \ \ %
  + \int_{s\in[\delta,1]} \int_{t=0}^\infty e^{-zt} \langle (f_1)_{\xi_s}, (f_2)_{\xi_s} \rangle
  (\phi_{\xi_s}(g_t)-c(s) e^{(s-1)t}) \dd t \dd m(\xi_s)
  \\&\ \ \ %
  + \int_{s\in [\delta,1]} \int_{t=0}^\infty e^{-zt} \langle (f_1)_{\xi_s}, (f_2)_{\xi_s} \rangle
  c(s) e^{(s-1)t} \dd t \dd m(\xi_s).
  \end{align*}
Let $B_\delta(z)$ be the last term in this expression, and
$A_\delta(z)$ the sum of the two other ones. In $A_\delta$, the
factors $\phi_{\xi_s}(g_t)$ and $\phi_{\xi_s}(g_t)- c(s)
e^{(s-1)t}$ are bounded, respectively, by $Ce^{-(1-\delta)t}$
and $Ce^{-t}$ (by \eqref{ratner} and \eqref{harish}).
Therefore, $A_\delta(z)$ extends to an holomorphic function on
$\{\Re(z)>-1+\delta\}$, which is bounded on the half-plane
$\{\Re(z)\geq-1+2\delta\}$. Since $\int_0^\infty e^{-at}\dd
t=1/a$ for $\Re(a)>0$, the function $B_\delta(z)$ is equal to
  \begin{equation*}
  \int_{s\in [\delta,1]} \langle (f_1)_{\xi_s}, (f_2)_{\xi_s} \rangle \frac{c(s)}{z-s+1} \dd m(\xi_s),
  \end{equation*}
for $\Re(z)>0$. This function can be holomorphically extended
to $z\not\in [-1+\delta,0]$, by the same formula. This proves
the proposition.
\end{proof}

\subsection{Proof of Theorem \ref{global_main_thm}}

We decompose the representation $H=L^2(X/\Gamma,\mu)$ of $\SL$
as a direct integral $\int_{\Xi} H_\xi \dd m(\xi)$, where the
representation $H_\xi$ is the direct sum of one or several
copies of the irreducible representation $\xi \in \Xi$. We
should prove that, for any $\delta>0$, the restriction of the
measure $m$ to $(\delta,1)$ (identified with the corresponding
set of representations in the complementary series) is made of
finitely many Dirac masses, and that at those points the
multiplicity of $\xi$ in $H_\xi$ is finite.

Let $\delta>0$ be small. Consider the eigenvalues $\lambda_i$
constructed in Proposition \ref{prop_mero_ext}. We claim that,
on the interval $(6\delta,1)$, the measure $m$ only gives mass
to the points $4\delta -1/\lambda_i + 1$, and that at such a
point the multiplicity of $\xi$ in $H_\xi$ is bounded by the
dimension of the image of the spectral projection
$\Pi_{\lambda_i}$ described in Proposition \ref{prop_mero_ext}.
This will conclude the proof of the theorem.

To proceed, we will use the fact that we have two different
expressions for the meromorphic extensions of Laplace
transforms, one related to the geometry of Teichm\"uller space
coming from Proposition \ref{prop_mero_ext}, and one given by
the abstract theory of representations of $\SL$ in Proposition
\ref{prop_extension}. Identifying these two expressions gives
the results, as follows.

\smallskip

\emph{First step: $m$ only gives weight to the points $4\delta
-1/\lambda_i+1$.} Assume by contradiction that $m$ gives positive
weight to an interval $[a,b]$ containing no such point. There exists
a function $f^{(0)}\in H$ invariant under $\SO$ such that the
corresponding components $f^{(0)}_\xi$ in $H_\xi$ satisfy
$\int_{[a,b]} \norm{ f^{(0)}_{\xi_s}}_{H_{\xi_s}}^2 \dd m(\xi_s) >
0$. Consider $\tilde f_{(n)}^{(0)} \in \boD^\Gamma$ a sequence of
smooth compactly supported functions converging to $f^{(0)}$ in $H$.
The functions $f_{(n)}^{(0)} = \int_{\theta \in \Sbb^1} k_\theta
\tilde f_{(n)}^{(0)} \dd\theta$ also belong to $\boD^\Gamma$, are
$\SO$-invariant, and converge to $f^{(0)}$ in $H$. In particular, if
$n$ is large enough, $\int_{[a,b]} \norm{
(f_{(n)}^{(0)})_{\xi_s}}_{H_{\xi_s}}^2 \dd m(\xi_s) > 0$. Let us fix
such a function $f=f_{(n)}^{(0)}$.

Consider the function $F_{f,f}(z)= \int_{t=0}^\infty e^{-zt}
\langle f, f\circ g_t\rangle \dd t$, initially defined for
$\Re(z)>0$. By Proposition \ref{prop_mero_ext}, it admits a
meromorphic extension to the domain $\boD_{\delta,\epsilon}$
for some $\epsilon>0$, with possible poles only at the points
$4\delta -1/\lambda_i$. Moreover, Proposition
\ref{prop_extension} shows that the same function can be
written, on the set $\{\Re(z)>-1+2\delta\}$, as the sum of a
bounded holomorphic function and the function
  \begin{equation*}
  B_\delta(z)=\frac{1}{\sqrt{\pi}} \int_{s\in [\delta,1]} \frac{\Gamma(s/2)}{\Gamma((s+1)/2)}
  \norm{f_{\xi_s}}^2_{H_{\xi_s}} \frac{\dd m(\xi_s)}{z-s+1}.
  \end{equation*}
It follows that this function $B_\delta$ can only have poles at
the points $4\delta -1/\lambda_i$. In particular, it is
continuous on the interval $[a-1,b-1]$. Lemma
\ref{lemchargepas} implies that the measure
$\dd\nu(s)=\frac{\Gamma(s/2)}{\Gamma((s+1)/2)}
\norm{f_{\xi_s}}^2_{H_{\xi_s}} \dd m(\xi_s)$ gives zero mass to
$[a,b]$. In particular, $\int_{[a,b]}
\norm{f_{\xi_s}}^2_{H_{\xi_s}} \dd m(\xi_s) = 0$. This is a
contradiction, and concludes the first step.

\smallskip

\emph{Second step: at a point $s=4\delta - 1/\lambda_i+1$, the
multiplicity of $\xi_s$ in $H_{\xi_s}$ is at most the dimension of
$\Ima \Pi_{\lambda_i}$ in the Banach space of Theorem
\ref{thm_control_spectral_radius_annonce}.} We argue again by
contradiction. Let $d=\dim \Ima \Pi_{\lambda_i}$, assume that the
multiplicity of $\xi_s$ in $H_{\xi_s}$ is at least $d+1$. Then one
can find in $H_{\xi_s}$ $d+1$ orthogonal functions $f^{(1)},\dots,
f^{(d+1)}$ which are $\SO$-invariant. Since $m$ has an atom at
$\xi_s$, these functions are elements of $H=L^2(X/\Gamma,\mu)$. As
above, we consider sequences $f_{(n)}^{(k)}\in \boD^\Gamma$ of
$\SO$-invariant functions that converge to $f^{(k)}$ in $H$.

Let $F_{f_{(n)}^{(k)}, f_{(n)}^{(\ell)}}(z)$ be the meromorphic
extension of the Laplace transform of the correlations of
$f_{(n)}^{(k)}$ and $f_{(n)}^{(\ell)}\circ g_t$, and let
$M^{k,\ell}_n$ denote its residue around the point $4\delta
-1/\lambda_i$. For each $n$, the residue $M^{k,\ell}_n$ is described
by Proposition \ref{prop_mero_ext}. Since the operator
$\Pi_{\lambda_i}$ has a $d$-dimensional image, it follows that the
rank of the matrix $M_n$ is at most $d$. On the other hand,
Proposition \ref{prop_extension} shows that $M^{k,\ell}_n = C \langle
(f_{(n)}^{(k)})_{\xi_s}, (f_{(n)}^{(\ell)})_{\xi_s} \rangle$ (where
$C>0$ depends only on $s$ and $m$). When $n$ tends to infinity, the
functions $f_{(n)}^{(k)}$ converge to $f^{(k)}$, hence $M_n$
converges to a diagonal matrix. In particular, $M_n$ is of rank $d+1$
for large enough $n$, a contradiction. \qed

\section{Measures with a local product structure on \texorpdfstring{$\Teich_1$}{Teich1}}
\label{section_local_product}

To construct the Banach space of
Theorem~\ref{thm_control_spectral_radius_annonce}, we will need
more geometric information on admissible measures, given by the
following proposition.

\begin{prop}
\label{prop_local_product}
Let $\tilde\mu$ be an admissible measure, supported on a
submanifold $X$ of $\Teich_1$. Then
\begin{enumerate}
\item
For every $x\in X$, there is a decomposition of the tangent
space $\boT_x X = \R\omega(x) \oplus E^u(x) \oplus E^s(x)$,
where $\omega(x)$ is the direction of the $g_t$-flow,
  \begin{equation*}
  E^u(x)=\boT_x X \cap D
  \Phi(x)^{-1}(H^1(M,\Sigma;\R)),
  \quad
  E^s(x)=\boT_x X \cap
  D\Phi(x)^{-1}(H^1(M,\Sigma;\ic\R)).
  \end{equation*}
\item The subspaces $E^s(x)$ and $E^u(x)$ depend in a
$C^\infty$ way on $x\in X$, are integrable, and the
integral leaves $W^u(x), W^s(x)$ are affine submanifolds of
$\Teich$.
\item
\label{pointmuu}
For every $x\in X$, there is a volume form $\mu_u$ on
$E^u(x)$ (defined up to sign), such that $x\mapsto
\mu_u(x)$ is $C^\infty$. Moreover, $x\mapsto \mu_u(x)$ is
constant along the unstable manifolds $W^u$. Additionally,
there exists a scalar $d\geq 0$ such that $(g_t)_*
\mu_u=e^{-dt}\mu_u$.
\item
For every $x\in X$, there is a volume form $\mu_s$ on
$E^s(x)$ (defined up to sign), such that $x\mapsto
\mu_s(x)$ is $C^\infty$. Moreover, $x\mapsto \mu_s(x)$ is
constant along the stable manifolds $W^s$. Additionally,
$(g_t)_* \mu_s=e^{dt}\mu_s$.
\item
For every $x\in X$, the volume form $\de\tilde\mu(x)$ on
$\boT_x X$ is equal to the product of $\dLeb$, $\mu_u(x)$
and $\mu_s(x)$ respectively in the directions $\omega(x)$,
$E^u(x)$ and $E^s(x)$.
\end{enumerate}
All these data are $\Gamma$-equivariant. We say that the
decomposition $\dd\tilde\mu=\dLeb \otimes \dd\mu_u \otimes
\dd\mu_s$ is the affine local product structure of $\mu$.
\end{prop}

Note that, since $W^u(x)$ is an affine submanifold, the tangent
spaces of $W^u(x)$ at two different points $y_1,y_2\in W^u(x)$
are canonically identified (i.e., their images under
$D\Phi(y_1)$ and $D\Phi(y_2)$ coincide), hence it is meaningful
to say in item \ref{pointmuu} of the above definition that
$y\mapsto \mu_u(y)$ is constant along $W^u(x)$. The same holds
for $\mu_s$ along $W^s$.

Note also that $E^u$ and $E^s$ are really the strong stable and
unstable manifolds. Indeed, $D \Phi(x)^{-1}(H^1(M,\Sigma;\R))$
is the weak unstable manifold for the flow on $\Teich$, but
since we are restricting to $\boT_x X$ we are excluding the
neutral directions (see the example of area-one surfaces
below).

If $x=a+\ic b$ in the chart $\Phi$, then for small $r$ we have
$h_r(x)=a+rb+\ic b$. In particular, the tangent vector to this
curve is always $b\in H^1(M,\Sigma;\R)$, hence $h_r(x)$ is in
the unstable manifold $W^u(x)$. Moreover, the differential of
$h_r$ sends $E^u(x)$ to $E^u(h_r x)$, and it is equal to the
identity in the chart $\Phi$. In particular, since $\mu_u$ is
constant along $W^u(x)$, this implies that $h_r$ leaves $\mu_u$
invariant, i.e., $(h_r)_* \mu_u=\mu_u$.

The family of volume forms $\mu_u(x)$ on $E^u(x)$ induces a positive
measure on each leaf $W^u$ of the unstable foliation, that we also
denote by $\mu_u$. In the same way, we get a measure $\mu_s$ on each
stable manifolds. Let us note that, although the volume forms
$\mu_u(x)$ are only defined up to sign, the induced positive measures
$\mu_u$ are canonical. If the manifolds $W^u$ and $W^s$ were
canonically oriented (or at least had a $\Gamma$ invariant
orientation), then $\mu_u(x)$ and $\mu_s(x)$ themselves would not be
defined only up to sign, but we do not know if this is always the
case.

The scalar $d$ in the above proposition can be identified, see
Remark \ref{identify_d}.

See \cite{babillot_ledrappier} for the notion of local product
structure in more complicated non-smooth settings.

\begin{example}
\label{ex_total}
Consider in $\Teich$ the subset $X=\Teich_1$ of area one surfaces,
with its canonical invariant Lebesgue measure $\tilde\mu$. We will
describe its affine local product structure. A similar construction
is given in \cite[Section 2]{athreya_et_al}, in more geometric terms.

First, assume $x\in X$ and $\Phi(x)=a+\ic b$. Around $x$, we
identify $\Teich$ and $H^1(M,\Sigma;\C)$ using $\Phi$. Then the
area of $a+a'+\ic b$ is $1+[a',b]$, where $[a',b]$ is the
intersection product of $a'$ and $b$ (this is initially defined
for elements of $H^1(M;\R)$, but since $H^1(M,\Sigma;\R)$
projects to $H^1(M;\R)$ it extends trivially to
$H^1(M,\Sigma;\R)$). Therefore, $E^u(x)=\{a' \in
H^1(M,\Sigma;\R) \st [a',b]=0\}$. This depends smoothly on $x$,
and the integral leaves of this distribution are locally the
sets $\{(a+a',b) \st [a',b]=0\}$. These are indeed affine
submanifolds of $\Teich$.

Let us now define $\mu_u$ at the point $a+\ic b$. The set
$H^1(M,\Sigma;\R)$ is endowed with a canonical volume form
$\vol$ (giving covolume $1$ to $H^1(M,\Sigma;\Z)$), we let
$\mu_u$ be the interior product of $a$ and $\vol$, i.e., if
$v_1,\dots,v_k$ is a basis of $E^u(x)$, then $\mu_u(x)=\vol(a,
v_1,\dots,v_k)$. At a nearby point $x'=a+a'+\ic b$ on the same
unstable manifold,
$\mu_u(x')(v_1,\dots,v_k)=\vol(a+a',v_1,\dots,v_k)=\vol(a,v_1,\dots,v_k)$
since $a'$ belongs to $E^u(x)$. Therefore, $\mu_u(x)=\mu_u(x')$
as claimed.

Let $d=k+1$ be the dimension of $H^1(M,\Sigma;\R)$. The
differential of $g_t$, mapping $E^u(x)$ to $E^u(g_t x)$, is
simply the multiplication by $e^t$, therefore $(g_t)_*
\mu_u(x)=e^{-(d-1)t}\mu_u(x)$. Since $\mu_u(g_t x) =
e^t\mu_u(x)$, we get $(g_t)_* \mu_u(x) = e^{-dt}\mu_u(g_t x)$.

In  the same way, we define a volume form $\mu_s(x)$ on
$E^s(x)$. It satisfies $(g_t)_* \mu_s=e^{dt}\mu_s$.

Let us finally define a volume form $\tilde\mu'$ on $\boT_x X$ as the
product of Lebesgue in the flow direction, $\mu_u$ and $\mu_s$. It
satisfies $(g_t)_* \tilde\mu'=\tilde\mu'$, since the factors
$e^{-dt}$ and $e^{dt}$ (coming respectively from $\mu_u$ and $\mu_s$)
cancel out.

All those data are intrinsically defined, and therefore
$\Gamma$-invariant. By ergodicity of $\mu$ in the quotient
$X/\Gamma$, we have $\tilde\mu'=c\tilde\mu$ for some $c\in
(0,+\infty)$.
\end{example}

We will prove simultaneously
Proposition~\ref{prop_local_product} (the fact that an
admissible measure has a local product structure) and
Proposition~\ref{prop_smooth_measure_is_projective} (the fact
that an absolutely continuous measure on a smooth submanifold
is automatically admissible): indeed, we will start from an
absolutely continuous measure and prove simultaneously that it
is admissible and that it has an affine local product
structure. For this proof, we will use the non-uniform
hyperbolicity of the Teichm\"uller flow. This property is
well-known, but we will need it later on in the following
precise form. Let us fix on $\Teich$ a $\Gamma$-invariant
Finsler metric. In later arguments, we will use a specific
metric which is well behaved at infinity (constructed in
Subsection~\ref{subsec_Finsler}), but the following statement
is valid for any metric.

\begin{prop}
\label{pliss_hyperbolicity}
For any set $K\subset \Teich_1$ which is compact mod $\Gamma$, there
exists $T=T(K)$ such that, for any point $x\in K$ and any time $t$
such that $g_t x\in K$ and
  \begin{equation*}
  \Leb \{ s\in [0,t] \st g_s(x)\in K\} \geq T,
  \end{equation*}
then $\norm{Dg_t(x) v}_{g_t x} \leq \norm{v}_x /2$ for any $v\in
E^s(x)$, and $\norm{Dg_t(x) v}_{g_t x} \geq 2\norm{v}_x $ for any
$v\in E^u(x)$.
\end{prop}
\begin{proof}
The uniform hyperbolicity of the Teichm\"uller flow in compact
subsets of $\Teich_1/\Gamma$ has been proved by Forni in
\cite[Lemma 2.1']{forni_deviation}, for a different norm, the
Hodge norm (and for vectors belonging to $H^1(M;\C)$ instead of
$H^1(M,\Sigma;\C)$). To obtain the result for the norm under
study, it is sufficient to use the following two facts:
\begin{enumerate}
\item Vectors in $H^1(M,\Sigma;\C)$ that vanish in $H^1(M;\C)$
are expanded at a constant rate $e^t$ in the unstable
direction, and contracted at a constant rate $e^{-t}$ in
the stable direction.
\item In a fixed compact subset of $\Teich_1/\Gamma$, any two continuous norms
are equivalent. \qedhere
\end{enumerate}
\end{proof}

\begin{proof}[Proof of
Propositions~\ref{prop_local_product} and
\ref{prop_smooth_measure_is_projective}]
Let us fix a measure $\tilde\mu$ as in the assumptions of
Proposition~\ref{prop_smooth_measure_is_projective}: it is
supported on a $C^1$ submanifold $X$ of $\Teich_1$, equivalent
to Lebesgue measure on $X$, and induces a Radon measure $\mu$
in $X/\Gamma$. We will prove that $\tilde\mu$ is admissible and
that it has an affine local product structure.

For $x\in \Teich_1$, denote by $\pi_\omega$, $\pi_u$ and
$\pi_s$ the projections respectively on the flow, unstable and
stable direction in the tangent space $\boT_x \Teich_1$.

\smallskip

\emph{First step: the measure $\mu$ has finite mass. In
particular, the flow $g_t$ is conservative in the measure space
$(X/\Gamma, \mu)$.}

Since $\mu$ is $\SL$-invariant, Athreya's Theorem
\cite{athreya} shows the existence of a compact set $K$ in
$X/\Gamma$ such that, under the iteration of $g_t$,
$\mu$-almost every point spends asymptotically at least half
its time in $K$. It follows from Hopf's ergodic theorem applied
to the ergodic components of $\mu$ that $\mu(X/\Gamma) \leq
2\mu(K)$. Since $\mu$ is a Radon measure, this quantity is
finite and the result follows.

\smallskip

\emph{Second step: at every point $x\in X$, we have
  \begin{equation}
  \label{good_decomposition}
  \boT_x X = \pi_\omega(\boT_x X) \oplus \pi_u(\boT_x X) \oplus
  \pi_s(\boT_x X).
  \end{equation}}
We will prove this property for $x$ in a dense subset of $X$, since
the general case follows by a limiting argument (using the
compactness of Grassmannians). The dimensions of $\pi_u(\boT_x X)$
and $\pi_s(\boT_x X)$ are semi-continuous, hence they are locally
constant on a dense subset of $X$. Moreover, since the flow $g_t$ is
conservative and $\mu$ has full support, Poincar\'e's recurrence
theorem ensures that almost every point of $X$ comes back close to
itself in the quotient $X/\Gamma$ infinitely often in forward and
backward time. We will prove \eqref{good_decomposition} for such a
point $x$.

First, since $g_t(x)\in X$ for all $t\geq 0$, we have
$\omega(x)= \partial g_t(x) / \partial t |_{t=0} \in \boT_x X$.
It is therefore sufficient to check that $\pi_u(\boT_x X)
\subset \boT_x X$ and $\pi_s(\boT_x X) \subset \boT_x X$. By
symmetry, it is even sufficient to prove the first inclusion.

Since the dimension of $\pi_u(\boT_y X)$ is locally constant
around $x$, there exists a constant $C$ such that, for any $y$
close to $x$, any vector $w_u\in \pi_u(\boT_y X)$ admits a lift
$w$ to $\boT_y X$ with $\norm{w}\leq C\norm{w_u}$.

Consider $v\in \boT_x X$, and write it as
$v=v_\omega+v_u+v_s\in \pi_\omega(\boT_x X) \oplus \pi_u(\boT_x
X) \oplus \pi_s(\boT_x X)$. We should prove that $v_u \in
\boT_x X$. Let $\epsilon>0$. Consider $t$ very large such that
$y=g_{-t}x$ is close to $x$. By Proposition
\ref{pliss_hyperbolicity}, if $t$ is large enough, the norm of
$w_u\coloneqq Dg_{-t}(x)\cdot v_u$ is bounded by $\epsilon$. We
may therefore find $w\in \boT_y X$ with $\pi_u(w)=w_u$, and
$\norm{w}\leq C\epsilon$. Write $w=w_\omega+w_u+w_s$. Then
$Dg_t(y)w \in \boT_x X$, and this vector can be written as
$Dg_t(y)(w_\omega+w_s) + v_u$ where
$\norm{Dg_t(y)(w_\omega+w_s)} \leq C\epsilon$. We have proved
that $v_u$ is a limit of points of $\boT_x X$, and therefore
that $v_u\in \boT_x X$. This concludes the proof of the second
step.

\smallskip

We can therefore define spaces $E^u(x)=\pi_u(\boT_x X) = \boT_x
X \cap \Phi^{-1}(H^1(M,\Sigma;\R))$ and $E^s(x)= \pi_s(\boT_x
X) = \boT_x X \cap \Phi^{-1}(H^1(M,\Sigma; \ic \R))$ such that
$\boT_x X=\R\omega \oplus E^u(x)\oplus E^s(x)$. Moreover, the
dimensions $d_u$ and $d_s$ of those spaces are locally
constant. Since the space $X/\Gamma$ is connected, they are in
fact constant. Finally, since the rotation $k_{\pi/2}$ maps
$E^u(x)$ to $E^s(k_{\pi/2}x)$, we have $d_u=d_s$.

Let $Y=\R_+^* X$. To simplify notations, we will omit $\Phi$
and identify locally $Y$ with a subset of $H^1(M,\Sigma;\C)$.
Since the tangent vector to the map $t\mapsto tx$ at $x=a+\ic
b$ is $a+\ic b$, we have $\boT_x Y=\R(a+\ic b) + \boT_x X$.
With the decomposition of $\boT_x X$ given at the first step
and the equality $\omega(x)=a-\ic b$, we obtain $\boT_x Y=
\tilde E^u(x) \oplus \tilde E^s(x)$, where $\tilde
E^u(x)=\boT_x Y\cap H^1(M,\Sigma;\R)=\R a \oplus E^u(x)$, and
$\tilde E^s(x)=\boT_x Y\cap H^1(M,\Sigma;\ic \R)=\ic\R b \oplus
E^s(x)$.

\smallskip

\emph{Third step: for any $x\in Y$, we have $\tilde E^s(x)= \ic
\tilde E^u(x)$.}

Let $e\in \tilde E^u(x)$ and $\ic f\in \tilde E^s(x)$. For
small $\theta$, the rotated vector $k_\theta(e+\ic f)=(\cos
(\theta) e+ \sin (\theta) f)+\ic (-\sin (\theta) e+\cos(\theta)
f)$ belongs to $\boT_{k_\theta x} Y$. Taking $f=0$ and
projecting to the real component, we deduce that $\tilde
E^u(k_\theta x)$ contains $\tilde E^u(x)$. Since they have the
same dimension, it follows that $\tilde E^u(k_\theta x)= \tilde
E^u(x)$. In the same way, taking $e=0$, we get $\tilde
E^u(k_\theta x)= \ic^{-1} \tilde E^s(x)$. Finally, $\tilde
E^s(x)=\ic \tilde E^u(x)$, as desired.

\smallskip

\emph{Fourth step: $Y$ is an affine submanifold of $\Teich$.}

At every point $x\in Y$, the tangent space $T_x Y = \tilde
E^u(x)\oplus \tilde E^s(x)$ is invariant by complex
multiplication, by the third step. This implies that $Y$ is a
complex (holomorphic) submanifold of $\Teich$, see
e.g.~\cite[Proposition 1.3.14]{CR_book}.

%M\"oller proved in \cite{moller_linear} that any $\GL$-invariant
%complex analytic submanifold of $\Teich$ is linear, therefore
%$Y$ is linear. It is also possible to give a direct proof, as
%follows (we thank S.~Cantat for the following argument).

Let us show that $Y$ is linear around any point $x_0\in Y$ (we
thank S.~Cantat for the following argument). Working in charts
and changing coordinates, we can assume that $x_0=0$ and that
$\boT_{0} Y = \C^k \subset \C^N$, for $k=d_s+1=d_u+1$. Around
$0$, the manifold $Y$ can therefore be written as a graph
$\{(z,f(z))\}$ for some holomorphic function from $\C^k$ to
$\C^{N-k}$. At a point $x$ close to $0$, the tangent space
$\boT_x Y$ is $\{(v, Df(x)v) \st v\in \C^k\}$. In particular,
the real part of this tangent space is included in $\{(v,
Df(x)v)\st v\in \R^k\}$. Since the dimension of the real part
of $\boT_x Y$ is exactly $k$, it follows that $Df(x)v$ is real
for any real vector $v$, i.e., all the matrix coefficients of
$Df(x)$ are real. Since a real valued holomorphic function is
constant, $Df$ is constant. Therefore, $Y$ is linear around
$x_0$.

\smallskip

\emph{Fifth step: the distributions of $d_u$-dimensional subspaces
$E^u$ and $E^s$ are integrable, and the integral leaves are affine
submanifolds of $\Teich$.}

The strong unstable manifolds form a foliation $\boF$ of $\Teich_1$
with affine leaves (see Example~\ref{ex_total}). Moreover, the
dimension of $T_x \boF \cap T_x X$ is independent of $x\in X$, by the
second step. It follows that the restriction of $\boF$ to $X$ defines
a foliation of $X$, integrating the distribution of subspaces $T_x
\boF \cap T_x X = E^u(x)$. In particular, the leaf $W^u(x)$
integrating $E^u(x)$ is locally given by $X \cap \boF_x$, which is
also equal to $Y\cap \boF_x$. Since $Y$ is affine by the fourth step
and $\boF_x$ is affine, $W^u(x)$ is also affine. The argument is the
same for $W^s$.

\smallskip

\emph{Sixth step: the measure $\tilde\mu_Y = \tilde\mu\otimes
\Leb$ on $Y$ is locally a multiple of the linear Lebesgue
measure on the linear manifold $Y$.}

Given $x\in Y$, fix a reference Lebesgue measure on $Y$ around
$x$ (there is a priori no canonical choice of normalization),
and denote by $\tilde\phi$ the density of $\tilde\mu_Y$ with
respect to this Lebesgue measure. We will prove that
$\tilde\phi$ is constant on strong stable and unstable
manifolds in a neighborhood of $x$. Since the foliations $W^s$
and $W^u$ are smooth and jointly non-integrable, it follows
from the classical Hopf argument that $\tilde\phi$ is constant.
We will work in $Y/\Gamma$, around the point $x\Gamma$. Let us
denote by $\phi$ the density of $\mu$ around $x\Gamma$.

Since $\mu$ has finite mass, we can consider a sequence of
compactly supported smooth measures $\mu_n$ converging (for the
total mass norm) to $\mu$ on $X/\Gamma$. For any $t\geq 0$,
  \begin{equation*}
  |(g_t)_* \mu_n - \mu| = |(g_t)_*\mu_n - (g_t)_*\mu|
  =|\mu_n - \mu|.
  \end{equation*}
Therefore, for any sequence $t_n$, the measures
$(g_{t_n})_*\mu_n$ converge to $\mu$.

Fix $M>0$. Let $\phi_{n,t}$ denote the density of $(g_t)_*(\mu_n
\otimes\Leb)$ in a ball $B$ around $x\Gamma$. Then, for any $n\in \N$
and any $M>0$, the integral
  \begin{equation*}
  \int_{y\in B} \int_{z\in W^u(y)\cap B}
  \min(|\phi_{n,t}(z)-\phi_{n,t}(y)|,M)\dLeb(z) \dLeb(y)
  \end{equation*}
converges to $0$ when $t$ tends to $+\infty$. Indeed, the
integrand is bounded by $M$, and converges almost everywhere to
$0$ since the flow is hyperbolic along almost every trajectory
and the measure $\mu_n$ is smooth. Let us choose $t_n$ such
that this integral is at most $2^{-n}$. Since
$(g_{t_n})_*\mu_n$ converges to $\mu$, the density
$\phi_{n,t_n}$ converges almost everywhere to $\phi$ along a
subsequence. This yields
  \begin{equation*}
  \int_{y\in B} \int_{z\in W^u(y)\cap B}
  \min(|\phi(z)-\phi(y)|,M)\dLeb(z) \dLeb(y) = 0.
  \end{equation*}
Letting $M$ tend to infinity, we obtain that $\phi$ is almost
everywhere constant along unstable manifolds in $B$, as desired.

\smallskip

\emph{Seventh step: the measure $\mu$ is ergodic.}

If $\mu$ is not ergodic, we can consider an invariant set $A$
for the action of $\SL$ on $X/\Gamma$, with positive but not
total measure. Consider $\tilde\nu$ the restriction of
$\tilde\mu$ to the lift of $A$ in $\Teich_1$. The argument in
the previous step applies to $\tilde\nu$ and shows that the
density of $\tilde\nu\otimes \Leb$ on $Y$ is locally constant.
Since $Y/\Gamma$ is connected, this implies that
$\nu\otimes\Leb$ is equivalent to Lebesgue measure on
$Y/\Gamma$. This is a contradiction, and proves the ergodicity
of $\mu$.

\smallskip

Among other things, we have shown that the measure $\tilde\mu$ is
admissible. This proves Proposition
\ref{prop_smooth_measure_is_projective}. To conclude the proof, we
have to complete the construction of the measures $\mu_u$ and $\mu_s$
forming the affine local product structure of $\mu$.

\smallskip

\emph{Eighth step: construction of canonical volume forms
$\mu_u(x)$ and $\mu_s(x)$, respectively on $E^u(x)$ and
$E^s(x)$, in terms of $\tilde\mu$, which are constant
respectively along $W^u$ and $W^s$. They are only defined up to
sign.}

Let $x\in X$. Identifying locally $\Teich$ with
$H^1(M,\Sigma;\C)$ thanks to the period map, we write $x=a+\ic
b$. If $\nu_u(x)$ is any volume form on $\tilde E^u(x)=\R a+
E^u(x)$, then it yields a volume form $\nu_s(x)$ on $\tilde
E^s(x)$ thanks to the identification of the third step. The
product $\nu_u(x)\wedge \nu_s(x)$ is a nonzero volume form on
$\tilde E^u(x)\oplus \tilde E^s(x)=\boT_x Y$, it is therefore
proportional to $\tilde\mu_Y(x)$. Multiplying $\nu_u(x)$ by a
unique (up to sign) normalization, we can ensure that
$\nu_u(x)\wedge\nu_s(x)=\pm\tilde\mu_Y$. Finally, let
$\mu_u(x)$ be the unique volume form on $E^u(x)$ such that
$\nu_u(x)$ is the product of $\mu_u(x)$ and Lebesgue measure on
$\R a$. Analogously, let $\mu_s(x)$ be the unique volume form
on $E^s(x)$ such that $\nu_s(x)$ is the product of $\mu_s(x)$
and Lebesgue measure on $\ic \R b$.

This construction is completely canonical up to sign, and
$\tilde\mu(x)=\pm\dLeb \wedge \mu_u(x)\wedge\mu_s(x)$ by
construction, where $\dLeb$ denotes Lebesgue measure along $\R
\omega(x)$. Since $\tilde\mu$ is $\Gamma$-invariant, it follows
that $\mu_u$ and $\mu_s$ are also $\Gamma$-invariant (possibly
up to sign).

Since $\mu_u$ is constructed in a canonical way in terms of
$\tilde\mu_Y$ and $\tilde\mu_Y$ is constant along unstable
manifolds (by the sixth step), it follows that $\mu_u$ is
constant along unstable manifolds. In the same way, $\mu_s$ is
constant along stable manifolds.

\smallskip

\emph{Ninth step: there exists $d > 0$ such that $(g_t)_* \mu_u
= e^{-dt}\mu_u$ and $(g_t)_* \mu_s = e^{dt}\mu_s$.}

Since the action of $\SL$ is ergodic, the action of the
horocycle flow is also ergodic by Howe-Moore's theorem
\cite{howe_moore}. In particular, we can choose $x$ whose orbit
is dense. For $t\geq 0$, the measure $(g_t)_* \mu_u(x)$ is a
volume form on $E^u(g_t x)$, and can therefore be written as
$e^{d(t)} \mu_u(g_t x)$ for some $d(t)\in \R$. Since the
measures $\mu_u$ are constant along the unstable manifolds of
$x$ and of $g_t x$, it follows that, for any point $y$ in the
horocycle through $x$, we also have $(g_t)_* \mu_u(y) =
e^{d(t)} \mu_u(g_t y)$. Since this horocycle is dense, $(g_t)_*
\mu_u(z)=e^{d(t)} \mu_u(g_t z)$ for any $z$.

The function $t\mapsto d(t)$ is continuous and satisfies
$d(t+t')=d(t)+d(t')$, we may therefore write $d(t)=-dt$ for
some $d\in \R$ (which has to be positive since the flow is
expanding along unstable directions). We obtain $(g_t)_* \mu_u
= e^{-dt} \mu_u$.

In the same way, we have $(g_t)_* \mu_s = e^{d't} \mu_s$ for
some $d'\geq 0$. Since $\tilde\mu=\dLeb \wedge \mu_u \wedge
\mu_s$ is $g_t$-invariant, it follows that $d=d'$.

This concludes the proofs of the ninth step and of
Propositions~\ref{prop_local_product} and
\ref{prop_smooth_measure_is_projective}.
\end{proof}

\begin{rmk}
\label{identify_d}
The scalar $d$ constructed in Proposition
\ref{prop_local_product} satisfies $d=d_u+1=d_s+1$, where $d_u$
and $d_s$ are the dimensions respectively of $E^u$ and $E^s$.

To prove this statement, let $\lambda_0,\dots,\lambda_{d_u}$ be
the Lyapunov exponents of the Kontsevich-Zorich cocycle (see
\cite{forni_deviation}) restricted to the bundle $\tilde E^u$,
for the measure $\mu$. The Lyapunov exponents of $g_t$ along
$\tilde E^u$ are given by $\nu_0=1+\lambda_0,\dots,
\nu_{d_u}=1+\lambda_{d_u}$, and their sum is equal to $d$ since
there is no expansion in the bundle $\tilde E^u/E^u$. Since
$\tilde E^s = \ic \tilde E^u$, the Lyapunov exponents of the
Kontsevich-Zorich cocycle along $\tilde E^s$ are also
$\lambda_0,\dots, \lambda_{d_u}$, and it follows that the
Lyapunov exponents of $g_t$ along $\tilde E^s$ are
$-1+\lambda_0,\dots, -1+\lambda_{d_u}$. Since $g_t$ preserves
the measure $\mu$ which is equivalent to Lebesgue measure, the
sum of its Lyapunov exponents vanishes. Hence, $\sum \lambda_i
= 0$. Finally, $d=\sum \nu_i = d_u+1+\sum \lambda_i = d_u+1$.

A suitably generalized Pesin formula also gives that $d$ is the
entropy of the measure $\mu$ for the flow $g_t$
\end{rmk}

\section{A good Finsler metric on \texorpdfstring{$\Teich$}{Teich}}

\label{sec_Finsler}

\subsection{Construction of the metric}
\label{subsec_Finsler}

To define the Banach space satisfying the conclusions of
Theorem \ref{thm_control_spectral_radius_annonce}, we will need
a Finsler metric on $\Teich$, with several good properties:
\begin{enumerate}
\item It should be complete and $\Gamma$--invariant.
\item It should behave in a controlled way close to infinity
(technically, it should be slowly varying, see the
definition below).
\item Under the Teichm\"uller flow, the metric should be
non-contracted in the unstable direction, and non-expanded
in the stable direction.
\end{enumerate}
It is certainly possible to cook up a metric satisfying these
requirements using the Hodge metric of Forni on $H^1(M;\R)$
\cite{forni_deviation} and extending it first to $H^1(M,\Sigma; \R)$
and then to $H^1(M,\Sigma;\C)$ (compare for instance
\cite{athreya_et_al}). However, \cite{AGY_teich} introduced a
geometrically defined metric that turns out to satisfy all the above
properties. This is the metric we will use for simplicity.

Let us describe this continuous Finsler metric on $\Teich$.
Since the tangent space of $\Teich$ is everywhere identified
with $H^1(M,\Sigma;\C)$ through the period map $\Phi$, it is
sufficient to define a family of norms on $H^1(M,\Sigma;\C)$,
depending continuously on the point $x\in \Teich$, as follows:
  \begin{equation*}
  \norm{v}_x = \sup \left|\frac{v(\gamma)}{\Phi(x)(\gamma)}\right|,
  \end{equation*}
where $\gamma$ runs over the saddle connections of the surface
$x$. It is proved in \cite{AGY_teich} that this is indeed a
norm, and that the corresponding Finsler metric is complete.
Let $d$ denote the distance on $\Teich$ coming from this
Finsler metric.

The two following straightforward lemmas show that this metric
behaves well with respect to the Teichm\"uller flow.

\begin{lem}
\label{SL_lipschitz}
The tangent vectors at $0$ to the families $t\mapsto g_t(x)$,
$r\mapsto h_r(x)$, $r\mapsto \tilde h_r(x)$ and $\theta \mapsto
k_\theta(x)$ are all bounded by $1$ in norm. Therefore, $d(x,g_t
x)\leq |t|$, $d(x, h_r x)\leq |r|$, $d(x, \tilde h_r(x))\leq |r|$ and
$d(x,k_\theta x) \leq |\theta|$.
\end{lem}
\begin{proof}
Given $x$ with $\Phi(x)=a+\ic b$, we have $\Phi(g_t x)=e^t
a+\ic e^{-t}b$, hence the tangent vector of the curve $t\mapsto
g_t x$ at $0$ is $a-\ic b$, which is clearly bounded by $1$
from the formula. Moreover, $\Phi(h_r x)=a+rb+\ic b$, hence the
tangent vector to this curve at $0$ is $b$, again bounded by
$1$. The computations are similar for $\tilde h_r$ and
$k_\theta$.
\end{proof}

\begin{lem}
\label{weakly_hyperbolic}
The Teichm\"uller flow is non-contracting in the unstable
direction and non-expanding in the stable direction, for the
above metric. More precisely, for any $t\geq 0$, for $v\in
H^1(M,\Sigma; \R)$ and $w\in H^1(M,\Sigma; \ic \R)$, we have
$\norm{ Dg_t(x) v}_{g_t x} \geq \norm{v}_x$ and $\norm{ Dg_t(x)
w}_{g_t x} \leq \norm{w}_x$.
\end{lem}
\begin{proof}
We have $Dg_t(x) v = e^t v$. Moreover, if $\Phi(x)=a+\ic b$, we
have $\Phi(g_t x)=e^t a+\ic e^{-t}b$. Therefore,
  \begin{equation*}
  \norm{ Dg_t(x) v}_{g_t x}=\sup_\gamma \frac{e^t |v(\gamma)|}{ | e^t a(\gamma) + \ic e^{-t} b(\gamma)|}
  \geq \sup_\gamma \frac{e^t |v(\gamma)|}{ | e^t a(\gamma) + \ic e^{t} b(\gamma)|}
  = \norm{v}_x.
  \end{equation*}
The argument for $w$ is the same.
\end{proof}
The same computation shows that $\norm{ Dg_t(x) v}_{g_t x} \leq
e^{2t} \norm{v}_x$ and $\norm{ Dg_t(x) w}_{g_t x} \geq
e^{-2t}\norm{w}_x$, which corresponds to the classical fact
that the upper and lower Lyapunov exponents of the Teichm\"uller
flow are respectively $2$ and $-2$.

\medskip

Let $\tilde\mu$ be an admissible measure, and let $X$ denote its
support. The above Finsler metric can be restricted to every stable
or unstable manifold in $X$, and therefore defines distances
$d_{W^u}$, $d_{W^s}$ on those manifolds. For $r>0$, we denote by
$W^u_r(x)$ the ball of radius $r$ around $x$ in $E^u(x)$ for the
distance $d_{W^u}$.

Fix $x\in X$. Let $\Psi=\Psi_x$ be the canonical local
parametrization of the affine manifold $W^u(x)$ by its tangent plane
$E^u(x)$. More formally, we define $\Psi(v)$ for $v\in E^u(x)$ as
follows. Consider the path $\paths$ starting from $x$ with
$\paths'(t)=v$ for all $t$. For small $t$, $\paths(t)$ is well
defined and belongs to $W^u(x)$. It is possible that $\paths(t)$ is
not defined for large $t$, since it could explode to infinity in
$\Teich$. If the path $\paths$ is well defined for all $t\in [0,1]$,
then we define $\Psi(v)=\paths(1)$.

Let us denote by $B(0,r)$ the ball of radius $r$ in $E^u(x)$, for the
norm $\norm{\cdot}_x$. The main result of this section is the following
proposition, showing that the norm $\norm{\cdot}_x$ varies slowly in
fixed size neighborhoods of any point in the non-compact space $X$.
This is a kind of bounded curvature behavior. Note however that this
metric depends only in a continuous way on the point, so we can not
use true curvature arguments.

\begin{prop}
\label{Psigood}
The map $\Psi$ is well defined on $B(0,1/2)$, and $d_{W^u}(x,
\Psi(v)) \leq 2 \norm{v}_x$ there. Moreover, for $v\in B(0,1/2)$, and
for every $w\in E^u(x)$,
  \begin{equation}
  \label{comparable_norms}
  1/2 \leq \frac{\norm{w}_x}{\norm{w}_{\Psi(v)}} \leq 2.
  \end{equation}
Finally, for $v\in B(0,1/25)$, we have $d_{W^u}(x,\Psi(v))\geq
\norm{v}_x/2$.
\end{prop}

Before proving this proposition, let us give a simple consequence for
the doubling property of $\mu_u$. Again, the interest of this
proposition is that the estimates are uniform, even though $X$ is not
compact.
\begin{cor}
\label{cor_doubling}
Let $\tilde\mu$ be a measure with an affine local product structure,
supported on a submanifold $X$. There exists $C>0$ such that, for
every $x\in X$ and every $r\leq 1/100$, $\mu_u(W^u_{2r}(x)) \leq C
\mu_u(W^u_{r}(x))$.
\end{cor}
\begin{proof}
By Proposition \ref{Psigood}, $\Psi^{-1}( W^u_{2r}(x)) \subset B(0,4r)$ and
$\Psi^{-1}( W^u_{r}(x)) \supset B(0,r/2)$. Since $y\mapsto \mu_u(y)$
is constant along $W^u(x)$, we have (denoting by $d_u$ the dimension
of $E^u(x)$)
  \begin{align*}
  \mu_u( W^u_{2r}(x)) &= \mu_u(x) (\Psi^{-1}( W^u_{2r}(x))) \leq \mu_u(x)(B(0,4r))
  \\&
  =8^{d_u} \mu_u(x)(B(0,r/2))
  \leq 8^{d_u} \mu_u(x) (\Psi^{-1}( W^u_{r}(x)))
  = 8^{d_u} \mu_u(W^u_{r}(x)).
  \qedhere
  \end{align*}
\end{proof}

The central point in the proof of Proposition~\ref{Psigood} is the
following proposition.

\begin{prop}
\label{norm_slowly_varies}
Let $\paths:[0,1] \to \Teich$ be a $C^1$ path. For each $v\in
H^1(M,\Sigma; \C)$,
  \begin{equation*}
  e^{-\length(\paths)}
  \leq
  \frac{ \norm{v}_{\paths(0)}}{\norm{v}_{\paths(1)}}
  \leq e^{\length(\paths)}.
  \end{equation*}
where $\length(\paths) = \int_{t=0}^1 \norm{
\paths'(t)}_{\paths(t)} \dd t$.
\end{prop}

By symmetry, it is sufficient to prove the upper bound. For the
proof, we start with the following lemma. We will write
$\paths(t)(\gamma)$ instead of $\Phi(\paths(t))(\gamma)$.
\begin{lem}
\label{when_survives}
Let $\gamma$ be a saddle connection surviving in the surface
$\paths(t)$, $t\in [t_1,t_2]$. Then
  \begin{equation*}
  \frac{|\paths(t_2)(\gamma)|}{|\paths(t_1)(\gamma)|}
  \leq e^{\int_{t_1}^{t_2} \norm{ \paths'(t)}_{\paths(t)} \dd t}.
  \end{equation*}
\end{lem}
\begin{proof}
Let $t\in [t_1,t_2]$. For small $h$,
  \begin{align*}
  \log |\paths(t+h)(\gamma)|&
  =\log |\paths(t)(\gamma) + h \paths'(t)(\gamma)+o(h)|
  \\&
  =\log |\paths(t)(\gamma)| + \log | 1+ h \paths'(t)(\gamma) / \paths(t)(\gamma) + o(h)|
  \\&
  =\log |\paths(t)(\gamma)| + h \Re(\paths'(t)(\gamma) / \paths(t)(\gamma))+o(h).
  \end{align*}
Hence, $t\mapsto \log |\paths(t)(\gamma)|$ is differentiable,
and its derivative $\Re(\paths'(t)(\gamma) /\paths(t)(\gamma))$
is bounded in norm by $\norm{ \paths'(t)}_{\paths(t)}$. The
result follows.
\end{proof}

\begin{proof}[Proof of Proposition \ref{norm_slowly_varies}]
For $0\leq t'_1\leq t'_2 \leq 1$, let us write
  \begin{equation*}
  I(t'_1,t'_2) = e^{\int_{t'_1}^{t'_2}
  \norm{\paths'(t)}_{\paths(t)} \dd t}.
  \end{equation*}
Let $\gamma$ be a fixed saddle connection in the surface
$\paths(0)$, we want to show that
  \begin{equation}
  \label{qslkdjflmqsdf}
  |v(\gamma)/ \paths(0)(\gamma)| \leq  I(0,1) \norm{v}_{\paths(1)}.
  \end{equation}

We define by induction a sequence of times $t_0<t_1<\dots$, and
sets $\Gamma_n$ of saddle connections on the surface
$\paths(t_n)$, as follows.

Let $t_0=0$ and $\Gamma_0=\{\gamma\}$. Assume $t_n$ and
$\Gamma_n$ are defined. If all the saddle connections in
$\Gamma_n$ survive in the surfaces $\paths(t)$, $t\in[t_n,1]$,
we let $t_{n+1}=1$ and stop the process here. Otherwise, let
$t_{n+1} \in (t_n,1]$ be the first time one or several saddle
connections in $\Gamma_n$ disappear. If $\tilde\gamma$ is such
a saddle connection, it means that other singularity points
arrive on $\tilde\gamma$, i.e., $\tilde\gamma$ is split in
$\paths(t_{n+1})$ into a finite set
$\{\gamma_1,\dots,\gamma_k\}$ of saddle connections, which are
all in the same direction. In particular, in homology,
$\tilde\gamma=\sum\gamma_i$, and moreover
$|\paths(t_{n+1})(\tilde\gamma)|=\sum
|\paths(t_{n+1})(\gamma_i)|$. We let $\Gamma_{n+1}$ be the
union of all the saddle connections in $\Gamma_n$ that survive
up to time $t_{n+1}$, and all the newly created saddle
connections $\gamma_i$.

We now show that this inductive construction reaches $t=1$ in a
finite number of steps. Let $S_n=\sum_{\tilde\gamma\in
\Gamma_n} |\paths(t_n)(\tilde\gamma)|$. For $\tilde\gamma\in
\Gamma_n$, Lemma \ref{when_survives} shows that
$|\paths(t_{n+1}-\epsilon)(\tilde\gamma)|\leq I(t_n,
t_{n+1}-\epsilon)|\paths(t_n)(\tilde\gamma)|$. Summing over
$\tilde\gamma$ and letting $\epsilon$ tend to $0$, we get
$S_{n+1} \leq I(t_n,t_{n+1}) S_n$. In particular, $S_n$ is
uniformly bounded, since $S_n\leq I(0, t_n) S_0\leq I(0, 1)
S_0$. Moreover, the length of saddle connections in all the
surfaces $\paths(t)$ is bounded from below, since
$\paths([0,1])$ is a compact subset of the Teichm\"uller space.
This implies that the cardinality of $\Gamma_n$ is uniformly
bounded. Since $\Card \Gamma_{n+1} \geq \Card\Gamma_n +1$, this
would give a contradiction if the inductive process did not
stop after finitely many steps.

We claim that, for all $n$,
  \begin{equation}
  \label{firsttime}
  \sup_{\tilde\gamma\in \Gamma_n} |v(\tilde\gamma)/ \paths(t_n)(\tilde\gamma)|
  \leq I(t_n,t_{n+1}) \sup_{\tilde\gamma\in \Gamma_{n+1}} |v(\tilde\gamma)/ \paths(t_{n+1})(\tilde\gamma)|.
  \end{equation}
Let $N$ be such that $t_N=1$. Multiplying these inequalities
for $n=0,\dots,N-1$, we obtain \eqref{qslkdjflmqsdf},
concluding the proof. We now prove \eqref{firsttime}. Let
$\tilde\gamma\in \Gamma_n$. If $\tilde\gamma$ survives up to
time $t_{n+1}$, Lemma \ref{when_survives} gives
$|v(\tilde\gamma)/ \paths(t_n)(\tilde\gamma)| \leq
I(t_n,t_{n+1}) |v(\tilde\gamma)/
\paths(t_{n+1})(\tilde\gamma)|$, as desired. Otherwise,
$\tilde\gamma$ is split at time $t_{n+1}$ into finitely many
saddle connections $\gamma_1,\dots,\gamma_k$. For small
$\epsilon>0$, the saddle connection $\tilde\gamma$ survives
from time $t_n$ to time $t_{n+1}-\epsilon$. Therefore, Lemma
\ref{when_survives} gives $|v(\tilde\gamma)/
\paths(t_n)(\tilde\gamma)| \leq I(t_n,
t_{n+1}-\epsilon)|v(\tilde\gamma)/
\paths(t_{n+1}-\epsilon)(\tilde\gamma)|$. When $\epsilon$ tends
to $0$, this tends to
  \begin{align*}
  I(t_n,t_{n+1}) \frac{ |v(\tilde\gamma)|}{ |\paths(t_{n+1}) (\tilde\gamma)|}&
  =I(t_n,t_{n+1}) \frac{ |\sum v(\gamma_i)|}{ \sum |\paths(t_{n+1}) (\gamma_i)|}
  \leq I(t_n,t_{n+1}) \frac{ \sum |v(\gamma_i)|}{\sum |\paths(t_{n+1}) (\gamma_i)|}
  \\&
  \leq I(t_n,t_{n+1}) \sup \frac{ |v(\gamma_i)|}{ |\paths(t_{n+1})(\gamma_i)|}.
  \end{align*}
This proves \eqref{firsttime}.
\end{proof}

\begin{proof}[Proof of Proposition~\ref{Psigood}]
Let $\paths$ be the path starting from $x$ with $\paths'=v$. If
$\paths$ is well defined on an interval $[0,t_0]$, then for
$t\in [0,t_0]$
  \begin{equation*}
  \norm{\paths'(t)}_{\paths(t)} = \norm{v}_{\paths(t)}
  \leq \norm{v}_{x} e^{\int_0^t \norm{\paths'(r)}_{\paths(r)} \dd r},
  \end{equation*}
by Proposition \ref{norm_slowly_varies}. Therefore, the
function $t\mapsto G(t)=\int_0^t \norm{\paths'(r)}_{\paths(r)}
\dd r$ satisfies $G'(t) \leq e^{G(t)} \norm{v}_x$, i.e., $
(-e^{-G(t)})' \leq  \norm{v}_x$. Integrating this inequality
gives $G(t) \leq -\log(1-t\norm{v}_x)$, and therefore
  \begin{equation*}
  \norm{\paths'(t)}_{\paths(t)} \leq \frac{\norm{v}_x}{1-t \norm{v}_x}.
  \end{equation*}
If $\norm{v}_x <1$, this quantity remains bounded for $t\in
[0,1]$. Therefore, $\Psi$ is well defined on such vectors $v$.
In particular, $\Psi$ is well defined on the ball $B(0,1/2)$.
Moreover, $d_{W^u}(x,\Psi(v))\leq \int_0^1
\norm{\paths'(t)}_{\paths(t)} \leq \norm{v}_x/(1- \norm{v}_x)$.
For $v\in B(0,1/2)$, this gives
  \begin{equation}
  \label{Psicontracte}
  d_{W^u}(x,\Psi(v)) \leq 2 \norm{v}_x.
  \end{equation}
Using the same notation $G$ as above, Proposition
\ref{norm_slowly_varies} shows that, for every $v\in B(0,1/2)$
and every $w\in E^u(x)$, we have $e^{-G(1)}\leq
\frac{\norm{w}_x}{\norm{w}_{\Psi(v)}} \leq e^{G(1)}$. Since
$G(1)\leq \log 2$, this proves \eqref{comparable_norms}.

Let us now prove that, for $v\in B(0,1/25)$, we also have
  \begin{equation}
  \label{Psidilate}
  d_{W^u}(x,\Psi(v)) \geq \norm{v}_x/2.
  \end{equation}
Consider $\paths:[0,1] \to W^u(x)$ an almost minimizing path for the
distance $d_{W^u}$, between $x$ and $\Psi(v)$. By
\eqref{Psicontracte}, its length is less than $1/10$. Let us lift
$\paths$ to a path $\tilde\paths$ taking values in $E^u(x)$, starting
from $0$ and such that $\paths=\Psi\circ \tilde\paths$, as long as
$\tilde\paths$ stays in $B(0,1/2)$.

While $\tilde\paths(t)$ is defined, we have by
\eqref{comparable_norms} $\norm{\tilde\paths'(t)}_x \leq 2
\norm{\tilde\paths'(t)}_{\paths(t)}$. Integrating this
inequality from $0$ to $t$, we get
  \begin{equation*}
  \norm{\tilde\paths(t)}_x \leq \int_0^t \norm{\tilde\paths'(r)}_x \dd r
  \leq 2 \int_0^t \norm{\tilde\paths'(r)}_{\paths(r)} \dd r
  \leq 2 \length(\paths)
  \leq 1/5.
  \end{equation*}
Therefore, $\tilde\paths(t)$ stays in $B(0,1/2)$, and the
lifting process may be continued up to $t=1$, where
$\tilde\paths(1)=v$. We get $\norm{v}_x \leq 2
\length(\paths)$. Hence, $\norm{v}_x \leq 2
d_{W^u}(x,\Psi(v))$, proving \eqref{Psidilate}.
\end{proof}

\subsection{\texorpdfstring{$C^k$}{Ck} norm and partitions of unity}

When $(E,\norm{\cdot})$ is a normed vector space and $f$ is a $C^k$
function on an open subset of $E$, let $c_k(f)=\sup |D^k f(x ;
v_1,\dots,v_k)|$ where the supremum is taken on the points $x$ in the
domain of $f$, and the tangent vectors $v_1,\dots, v_k$ of norm at
most $1$.

If an affine manifold has a Finsler metric, we can define in
the same way the $c_k$ coefficients of a function, using the
affine structure to define the $k$-th differential at every
point, and the Finsler metric to measure the tangent vectors.
Note that the (possibly non-smooth) variation of the Finsler
metric from point to point plays no role in this definition,
since it only uses the Finsler metric at a fixed point. Those
coefficients behave well under the composition with affine
maps.

We can then define the $C^k$ norm of a function by
$\norm{f}_{C^k}=\sum_{j=0}^k c_j(f)$. When we say that a
function is $C^k$ on a non-compact space, we really mean that
its $C^k$ norm is finite.

\begin{rmk}
\label{rem_def_norme_vecteur}
There are several more general situations where this definition has a
natural extension. Consider for example the following case: $W$ is an
affine submanifold of an affine Finsler manifold $Z$, and $v$ is a
vector field defined on $W$ (but pointing in any direction in $Z$).
Then, for $x\in W$ and $v_1,\dots, v_k \in \boT_x W$, the $k$-th
differential $D^k v(x ; v_1,\dots,v_k)$ is well defined and belongs
to the normed vector space $\boT_x Z$. We can therefore define
$c_k(v)$ as the supremum of the quantities $\norm{ D^k
v(x;v_1,\dots,v_k)}_x$, for $x\in W$ and $v_1,\dots,v_k \in \boT_x W$
with $\norm{v_i}_x\leq 1$. Finally, we set as above
$\norm{v}_{C^k}=\sum_{j=0}^k c_j(v)$.
\end{rmk}

Note however that there are several situations where it is not
possible to canonically define a $C^k$ norm as above. For
instance, on a general Finsler manifold, there is no canonical
connection, and therefore $D^k f$ is not well defined. In the
same way, in Remark \ref{rem_def_norme_vecteur}, if $W$ is not
affine or if $Z$ is not affine, then we can not define
$\norm{v}_{C^k}$. Of course, in a compact subset of $W$, one
could choose charts to define such a norm, but it would depend
on the choice of the charts -- the equivalence class of the
$C^k$ norm is well defined, but the $C^k$ norm itself is not.
Further on, we will need to control constants precisely, and it
will be very important for us to have a canonical norm.

Consider now an admissible measure $\tilde\mu$, supported on a
manifold $X$. Since the local unstable manifolds $W^u(x)$ are
affine manifolds, the previous discussion applies to them.

The next proposition constructs good partitions of unity on
pieces of such unstable manifolds.
\begin{prop}
\label{prop_partition}
There exists a constant $C$ with the following property. Let
$W$ be a compact subset of an unstable leaf $W^u(x)$. Then
there exist finitely many $C^\infty$ functions
$(\partition_i)_{i\in I}$ on $W^u(x)$, taking values in
$[0,1]$, with $\sum \partition_i=1$ on $W$,
$\sum\partition_i=0$ outside of $\{y\in W^u(x) \st d_{W^u}(y,W)
\leq 1/200\}$, and each $\partition_i$ is supported in a ball
$W^u_{1/200}(x_i)$ for some $x_i \in W$. Moreover, we can ensure that
$c_k(\partition_i) \leq C (k!)^2$, and every point of $W^u(x)$
belongs to at most $C$ sets $W^u_{1/200}(x_i)$.
\end{prop}
The precise bound $C(k!)^2$ is not important for the
applications we have in mind, what really matters is that we
have a bound depending only on $k$, uniform in $x$.
\begin{proof}
By Proposition \ref{Psigood}, the norm $\norm{\cdot}_x$ is slowly varying in the
sense of \cite[Definition 1.4.7]{hormander}. Applying Theorem 1.4.10
there to the sequence $d_k=c/k^{3/2}$ for some $c>0$, we get a
sequence of functions $\partition_i$ satisfying the conclusion of our
proposition: they satisfy $c_k(\partition_i) \leq C^k (k!)^{3/2}$ for
a constant $C$ depending only on the dimension, so
$c_k(\partition_i)\leq C' (k!)^2$, and moreover the assertions on the
support are also satisfied. One should only be a little careful since
the supports in \cite[Theorem 1.4.10]{hormander} are controlled in
terms of fixed norms $\norm{\cdot}_x$, while our conclusion deals with the
Finsler metric $d_{W^u}$. Since Proposition \ref{Psigood} shows that they
are uniformly equivalent in small neighborhoods of the points, this
is not an issue.
\end{proof}

The next lemma is a particular case of Proposition
\ref{prop_partition} (obtained by letting $W=W^u_{1/200}(x)$), and
will be needed later on.
\begin{lem}
\label{lem_cutonW}
There exists a constant $C$ with the following property. For
any $x\in X$, there exists a function $\partition$ on $W^u(x)$,
supported in $W^u_{1/100}(x)$, taking values in $[0,1]$, equal
to $1$ on $W^u_{1/200}(x)$, with $c_k(\partition) \leq C
(k!)^2$.
\end{lem}
The interest of this lemma is, again, that the estimates are
uniform in $x$ while this point lives in a noncompact space.

In the next statement, we do not use the distance induced by
the Finsler metric on unstable manifolds, but the global
distance. Since the previous arguments only rely on Proposition
\ref{norm_slowly_varies}, which is satisfied in $W^u$ as well
as in the whole space, this lemma follows again from the same
techniques.
\begin{lem}
\label{lem_phiV}
There exists a constant $C$ with the following property. Let $F :
\Teich \to [1,\infty)$ be a function such that $|\log F(x)-\log
F(y)|\leq 2 d(x,y)$ for any $x,y\in \Teich$. For any $V\geq 1$, there
exists a $C^\infty$ function $\partition_V : \Teich \to [0,1]$ such
that $\partition_V(x) = 1$ if $F(x)\leq V$ and $\partition_V(x)=0$ if
$F(x) \geq 2V$, satisfying $c_k(\partition_V) \leq C (k!)^2$.
\end{lem}

\section{Recurrence estimates}
\label{section_recurrence}

Our goal in this section is to prove the following exponential
recurrence estimate. Consider an admissible measure $\tilde\mu$
with its affine local product structure, supported on a
submanifold $X$. If $x$ is a translation surface, let $\sys(x)$
be its systole, i.e., the length of the shortest saddle
connection in $x$.
\begin{prop}
\label{prop_Vdelta_OK_sur_Wu}
Let $\delta\in (0,1/4)$. There exists $C>0$ such that, for any
$x\in X$ and any $t\geq 0$,
  \begin{equation*}
  %\label{eqVdeltaWu}
  \frac{1}{\mu_u(W^u_{1/100}(x))}\int_{W^u_{1/100}(x)} V_\delta( g_t y) \dd\mu_u(y)
  \leq C e^{-(1-2\delta) t} V_\delta(x)+C,
  \end{equation*}
where $V_\delta(x)=\max(1/\sys(x)^{1+\delta}, 1)$. Moreover,
the function $\log V_\delta$ is $(1+\delta)$-Lipschitz for the
Finsler norm of the previous section.
\end{prop}

We will use the following lemma, which is due to Eskin-Masur
\cite{eskin_masur} and Athreya \cite{athreya}.
\begin{lem}
Fix a neighborhood $\mathcal V$ of the identity in $\SL$.
For every $\delta>0$, there exists $C>0$ such that, for all
$t>0$, there exist a function $V^{(t)}_\delta : \Teich \to
[1,\infty)$ and a scalar $b(t)>0$ satisfying the following
property. For all $x\in \Teich_1$,
  \begin{equation*}
  %\label{good_vdelta}
  \int_0^{2\pi} V_\delta^{(t)}(g_t k_\theta x) \dd \theta
  \leq C e^{-(1-\delta)t} V_\delta^{(t)}(x) + b(t).
  \end{equation*}
Moreover,
  \begin{equation}
  \label{Vdeltasmooth}
  V_\delta^{(t)}(g x) \leq C V_\delta^{(t)}(x)
  \end{equation}
for all $x\in \Teich$ and all $g\in \mathcal V$.
Finally, there exists a constant
$C_{\delta,t}$ such that $V_\delta^{(t)}/ V_\delta\in
[C_{\delta,t}^{-1}, C_{\delta,t}]$.
\end{lem}
The order of quantifiers in our statement corrects a mistake in
Athreya's Lemma 2.10.

In the next lemma, we transfer the previous estimate on circle
averages to estimates on horocycle averages.
\begin{lem}
\label{interm}
For every $\delta>0$, there exists $C$ such that, for any large
enough $t$, there exists $b(t)>0$ such that, for any $x\in
\Teich_1$,
  \begin{equation*}
  \int_0^1 V_\delta^{(t)}( g_t h_r x) \dd r \leq C e^{-(1-\delta) t}
V^{(t)}_\delta(x)+b(t)\, .
  \end{equation*}
\end{lem}
\begin{proof}
Using the decomposition $ANK$ of $\SL$, we can write uniquely
$h_r=g_{\tau(r)} \tilde h_{\tilde r(r)} k_{\theta(r)}$, where
the functions $\tau$, $\tilde r$ and $\theta$ depend smoothly
on $r$. One easily checks that $\theta'(0)\not=0$. In
particular, if $n$ is large enough, $r\mapsto \theta(r)$ is a
diffeomorphism on $[0,1/n]$. Using the commutation relation
$g_\tau \tilde h_{\tilde r}=\tilde h_{e^{-2\tau} \tilde{r}}
g_\tau$, we get
  \begin{align*}
  \int_0^{1/n} V_\delta^{(t)}(g_t h_r x)\dd r
  &
  =\int_0^{1/n} V_\delta^{(t)}(g_t g_\tau \tilde h_{\tilde r} k_\theta x)\dd r
  \\&
  =\int_0^{1/n} V_\delta^{(t)}(\tilde h_{\tilde r e^{-2(t+\tau)}} g_\tau g_t  k_\theta x) \dd r.
  \end{align*}
By \eqref{Vdeltasmooth}, this is bounded by
  \begin{align*}
  C \int_0^{1/n} V_\delta^{(t)}( g_t k_{\theta} x) \dd r
  &
  = C \int_{\theta([0,1/n])}V_\delta^{(t)}( g_t k_u x) (\theta^{-1})'(u) \dd u
  \\&
  \leq C \int_0^{2\pi} V_\delta^{(t)}(g_t k_u x) \dd u
  \\&
  \leq C e^{-(1-\delta) t} V_\delta^{(t)}(x)+b(t).
  \end{align*}
Therefore,
  \begin{align*}
  \int_0^1 V_\delta^{(t)}(g_t h_r x) \dd r
  &=\sum_{j=0}^{n-1} \int_0^{1/n} V_\delta^{(t)}(g_t h_r h_{j/n} x) \dd r
  \\
  \leq \sum_{j=0}^{n-1} C e^{-(1-\delta) t} V_\delta^{(t)}(h_{j/n} x)+b(t).
  \end{align*}
With \eqref{Vdeltasmooth}, this gives the conclusion of the
lemma.
\end{proof}

\begin{lem}
\label{lemme_alltimes}
For every $\delta>0$,
there exist $C$ and $\tau$ such that, for any $t\geq 0$ and any $x\in
\Teich_1$,
  \begin{equation}
  \label{Vdeltg_alltimes}
  \int_0^1 V_\delta^{(\tau)}( g_t h_r x) \dd r \leq C e^{-(1-2\delta) t} V_\delta^{(\tau)}(x)+C\, .
  \end{equation}
\end{lem}
The difference with the previous lemma is that we obtain a
result valid for all times, with constants independent of the
time (while $b$ depends on $t$ in the statement of Lemma
\ref{interm}).
\begin{proof}
Let us fix $\tau$ and $b$ such that, for every $x\in \Teich_1$,
  \begin{equation}
  \label{qlskdjflsmqfd}
  \int_0^1 V_\delta^{(\tau)}( g_\tau h_r x) \dd r \leq e^{-(1-2\delta)
  \tau} \int_0^1 V_\delta^{(\tau)}(h_r x)+b.
  \end{equation}
Their existence follows from Lemma \ref{interm} and
\eqref{Vdeltasmooth}. We can also assume that $e^{2\tau}$ is a
(large) integer $N$.

Let us now prove that, for all $n\in \N$,
  \begin{equation}
  \label{Vdeltg_induc}
  \int_0^1 V_\delta^{(\tau)}(g_{(n+1)\tau} h_r x)\dd r
  \leq e^{-(1-2\delta) \tau} \int_0^1 V_\delta^{(\tau)}(g_{n\tau}h_r x)\dd r+b.
  \end{equation}
A geometric series then shows \eqref{Vdeltg_alltimes} for times of
the form $n\tau$, and the general result follows from
\eqref{Vdeltasmooth}.

To prove \eqref{Vdeltg_induc}, write $g_{(n+1)\tau} h_r =
g_\tau g_{n\tau} h_r = g_\tau h_{e^{2n\tau} r} g_{n\tau}$ with
$e^{2n\tau}=N^n=M$. Then, writing $r'=Mr$,
  \begin{align*}
  \int_0^1 V_\delta^{(\tau)}(g_{(n+1)\tau} h_r x)\dd r
  &=\sum_{j=0}^{M-1} \int_0^{1/M} V_\delta^{(\tau)}( g_{(n+1) \tau} h_r h_{j/M} x)\dd r
  \\&
  =\sum_{j=0}^{M-1} \int_0^{1/M} V_\delta^{(\tau)}( g_{\tau} h_{M r} g_{n\tau} h_{j/M} x)\dd r
  \\&
  = \frac{1}{M}\sum_{j=0}^{M-1} \int_0^1 V_\delta^{(\tau)}( g_{\tau} h_{r'} g_{n\tau} h_{j/M} x)\dd r'
  \\&
  \leq \frac{1}{M} \sum_{j=0}^{M-1}\left( e^{-(1-2\delta)\tau} \int_0^1 V_\delta^{(\tau)}( h_{r'} g_{n\tau} h_{j/M} x)\dd r' +b\right),
  \end{align*}
where the last inequality follows from \eqref{qlskdjflsmqfd}
applied to the point $g_{n\tau} h_{j/M} x$. Changing again
variables in the opposite direction, we get
\eqref{Vdeltg_induc}.
\end{proof}

\begin{proof}[Proof of Proposition \ref{prop_Vdelta_OK_sur_Wu}]
The log-smoothness of $V_\delta$ readily follows from the fact
that $\log \sys$ is $1$-Lipschitz by \cite[Lemma
2.12]{AGY_teich}.

Let $\tau$ be given by Lemma \ref{lemme_alltimes}. Since $V_\delta$
is within a multiplicative constant of $V_\delta^{(\tau)}$, it also
satisfies the inequality \eqref{Vdeltg_alltimes} (with a different
constant $C$).

Fix $r \in [0,1/100]$. Since $\mu_u$ is invariant under $h_r$,
  \begin{multline*}
  \int_{W^u_{1/100}(x)} V_\delta( g_t y) \dd\mu_u(y)
  =\int_{W^u_{1/100}(x)} V_\delta( g_t h_r h_{-r} y) \dd\mu_u(y)
  \\
  =\int_{ h_{-r} W^u_{1/100}(x)} V_\delta(g_t h_r z) \dd\mu_u(z)
  \leq \int_{ W^u_{1/50}(x)} V_\delta(g_t h_r z) \dd\mu_u(z).
  \end{multline*}
Averaging over $r$, we get
  \begin{align*}
  \int_{W^u_{1/100}(x)} V_\delta( g_t y) \dd\mu_u(y)
  &
  \leq 100\int_{r=0}^{1/100}\int_{ W^u_{1/50}(x)} V_\delta(g_t h_r z) \dd\mu_u(z) \dd r
  \\&
  \leq 100 \int_{ W^u_{1/50}(x)} \left( \int_0^1 V_\delta(g_t h_r z) \dd r\right) \dd\mu_u(z).
  \end{align*}
This is bounded by $ \mu_u(W^u_{1/50}(x)) (C e^{-(1-2\delta)t}
V_\delta(x) +C)$, using \eqref{Vdeltg_alltimes} for $V_\delta$ and
the fact that $V_\delta(z)/V_\delta(x)$ is uniformly bounded for all
$z\in W^u_{1/50}(x)$ (since $\log V_\delta$ is Lipschitz). The result
follows since the measures of $W^u_{1/50}(x)$ and $W^u_{1/100}(x)$
are comparable by Corollary \ref{cor_doubling}.
\end{proof}

\section{Distributional coefficients}
\label{par_construct_norms}

In this section, we introduce a distributional norm on smooth
functions, similar in many respects to the norms introduced in
\cite{gouezel_liverani} (the differences are the control at
infinity, and the fact that we only use vector fields pointing
in the stable direction or the flow direction -- this is
simpler than the approach of \cite{gouezel_liverani}, and is
made possible here by the smooth structure of the stable
foliation). Let us fix $\tilde\mu$ an admissible measure with
its affine local product structure, supported by a manifold
$X$. Let also $\delta>0$ be a fixed small number, as in the
previous section.

Consider a smooth vector field $v^s$ on a piece of unstable manifold
$W^u_{1/100}(x)$, such that for every $y\in W^u_{1/100}(x)$,
$v^s(y)\in E^s(y)$. We can define its $c_k$ coefficients as in Remark
\ref{rem_def_norme_vecteur}. For a vector field
$v^\omega(y)=\psi(y)\omega(y)$ defined on $W^u_{1/100}(x)$, we let
its $c_k$ coefficient be $c_k(\psi)$. The definitions of
$\norm{v^s}_{C^k}$ and $\norm{v^\omega}_{C^k}$ follow. Let us stress
that these definitions only involve base points that are located on
an unstable manifold: this implies that these norms behave well under
$g_{-t}$, which is contracting along such an unstable manifold, and
is at the heart of the proof of Lemma \ref{lemme_basic_estimate}
below.

We want to use such vector fields to differentiate functions, several
times. However, the Lie derivative $L_{v_1} L_{v_2} f$ of a function
$f$ can only be defined if $L_{v_2}f$ is defined on an open set,
which means that $v_2$ has to be defined on an open set. Therefore,
we will need to extend the above vector fields to whole open sets, as
follows.

Consider first a smooth vector field $v^s$ on $W^u_{1/100}(x)$,
pointing everywhere in the stable direction. We will now
construct an extension $\extension{v^s}$ of $v^s$ to a
neighborhood of $W^u_{1/200}(x)$ in $X$.

For $y\in W^u_{1/100}(x)$, the stable manifold $W^s(y)$ is
affine, its tangent space is everywhere equal to $E^s(y)$, and
we may therefore define $\extension{v^s}(z)=v^s(y)$ for $z\in
W^s(y)$: this extended vector field is still tangent to the
direction $E^s$. Finally, for small $t$, we define
$\extension{v^s}(g_t z)=D g_t(z) \cdot \extension{v^s}(z)$,
i.e., we push the vector field by $g_t$. Since $g_t$ sends
stable direction to stable direction, $\extension{v^s}$ is
everywhere tangent to the stable direction. Since the unstable
direction, the stable direction and the flow direction are
transverse at every point, we can uniquely parameterize a point
in a neighborhood of $W^u_{1/200}(x)$ as $g_t(z)$ for some
$z\in W_\epsilon^s(y)$, $y\in W^u_{1/200+\epsilon}(x)$. This
defines the extension of $v^s$.

If $v^\omega$ is a vector field along $W^u_{1/100}(x)$ pointing
everywhere in the flow direction, we can also define an
extension $\extension{v^\omega}$ as follows. Along $W^u$, write
$v^\omega(y)=\psi(y) \omega(y)$, where the function $\psi$ is
smooth. Let $\extension{v^\omega} (g_t z)= \psi(y) \omega(g_t
z)$ for $z\in W_\epsilon^s(y)$ as above, this defines a smooth
vector field extending $v^\omega$ as desired.

%
%
%It is also possible to define the extension of a vector field
%pointing in the unstable direction, using (a generalization of) the
%previous procedure: if $x$ and $y$ are two nearby points of $X$,
%their tangent spaces are identified by the composition
%  \begin{equation*}
%  \boT_x X \hookrightarrow \boT_x Y = \boT_y Y = \R\oplus  \boT_y X
%  \to \boT_y X,
%  \end{equation*}
%where $Y=\R_+^* \times X \subset \Teich$ is affine (hence the
%equality $\boT_x Y=\boT_y Y$), and the last map is the second
%projection. Using this connection, one may transport any vector field
%on $W^u(x)$, from $y$ to any point $g_t(z)$ for $z\in
%W_\epsilon^s(y)$. This process respects the flow, stable and unstable
%directions, coincides with the previous description for vectors in
%the stable direction or in the flow direction, and also works for
%vectors in the unstable direction. However, since we need anyway to
%separate the flow direction and the stable direction in forthcoming
%arguments, the previous simple description is more useful to us.

\medskip

For $k,\ell\in \N$, $\alpha\in \{s,\omega\}^\ell$ and $x \in
X$, we can now define a distributional coefficient of the
$C^\infty$ function $f$ at $x$, as follows (the function
$V_\delta$ has been defined in Proposition
\ref{prop_Vdelta_OK_sur_Wu}):
  \begin{equation}
  \label{defekl}
  e_{k,\ell,\alpha}(f;x)\coloneqq \frac{1}{V_\delta(x)} \frac{1}{\mu_u(W^u_{1/200}(x))}\sup\left|\int_{W^u_{1/200}(x)}
  \phi \cdot L_{\extension{v_1}}\cdots L_{\extension{v_\ell}} f \dd\mu_u \right|,
  \end{equation}
where the supremum is over all compactly supported functions $\phi:
W^u_{1/200}(x) \to \C$ with $\norm{\phi}_{C^{k+\ell}}\leq 1$, and all
vector fields $v_1,\dots,v_\ell$ defined on $W^u_{1/100}(x)$ such
that $v_j(y)\in E^s(y)$ if $\alpha_j = s$ and $v_j(y)\in \R\omega(y)$
if $\alpha_j=\omega$, and
$\norm{v_j}_{C^{k+\ell+1}(W^u_{1/100}(x))}\leq 1$. Note that the
domain of definition of the vector fields is \emph{larger} than the
domain of integration in \eqref{defekl} -- this will be useful for
extension purposes below. Note also that we use the Lie derivative
with respect to the extended vector fields $\extension{v_j}$, but the
norm requirements on the vector fields $v_j$ are only along $W^u$ and
not in the transverse direction.

Define $e_{k,\ell,\alpha}(f)=\sup_x e_{k,\ell,\alpha}(f;x)$.
Let $e_{k,\ell}(f)=\sum_{\alpha \in \{s,\omega\}^\ell}
e_{k,\ell,\alpha}(f)$. Finally, let
  \begin{equation}
  \label{defnorm}
  \norm{f}_k = \sup_{0\leq \ell \leq k} e_{k,\ell}(f).
  \end{equation}

\begin{rmk}
\label{continuity f_1}
If $f_1 \in \boD^\Gamma$ then we have the estimate
  \begin{equation*}
  \int_{X/\Gamma} f_1 f \dd \mu
  \leq C(f_1) e_{k,0}(f) \leq C(f_1) \norm{f}_k, \quad f \in \boD^\Gamma,
  \end{equation*}
where $C(f_1)$ depends on the
support of $f_1$ as well as its $C^k$ norm therein.  This is readily obtained by
decomposing $f_1$ as a sum of finitely many functions with small support
(using partitions of unity), using locally the disintegration of $\mu$ along
local unstable manifolds, and applying the definition of $e_{k,0}$ to bound the
integrals along those.

\end{rmk}

We will also need a weaker norm, that we denote by
$\normdeux{\cdot}_k$, given by
  \begin{equation}
  \label{defnormdeux}
  \normdeux{f}_k = \sup \frac{1}{V_\delta(x)} \frac{1}{\mu_u(W^u_{1/200}(x))}\left|\int_{W^u_{1/200}(x)}
  \phi \cdot L_{\extension{v_1}}\cdots L_{\extension{v_\ell}} f \dd\mu_u \right|,
  \end{equation}
where the supremum is over $0\leq \ell\leq k-1$, over all points
$x\in X$, all compactly supported functions $\phi: W^u_{1/200}(x) \to
\C$ with $\norm{\phi}_{C^{k+\ell+1}}\leq 1$, and all vector fields
$v_1,\dots,v_\ell$ defined on $W^u_{1/100}(x)$ and pointing either in
the stable direction or in the flow direction, such that
$\norm{v_j}_{C^{k+\ell+1}(W^u_{1/100}(x))}\leq 1$. Apart from
constants, the difference with the norm $\norm{f}_k$ is that we allow
less derivatives (at most $k-1$ instead of $k$), and that the test
function $\phi$ has one more degree of smoothness (it is in
$C^{k+\ell+1}$ instead of $C^{k+\ell}$). Therefore, the norm
$\normdeux{f}_k$ is weaker in all directions than the norm
$\norm{f}_k$. Hence, the following compactness result is not
surprising.

\begin{prop}
\label{prop_compactness}
Let $K$ be a compact set mod $\Gamma$, and let $k\in \N$. Let
$f_n$ be a sequence of functions in $\boD^\Gamma$, supported in
$K$, and with $\norm{f_n}_k \leq 1$. Then there exists a
subsequence $f_{j(n)}$ which is Cauchy for the norm
$\normdeux{\cdot}_k$.
\end{prop}
In other words, if we work with the completions of the spaces,
then the unit ball for the norm $\norm{\cdot}_k$ is relatively
compact for the norm $\normdeux{\cdot}_k$ if we consider only
functions on $X/\Gamma$ that are supported in a fixed compact
set.

The rest of this subsection is devoted to the proof of this
proposition (it is similar to the proof of Lemma 2.1 in
\cite{gouezel_liverani}). We will need a preliminary lemma.

Let us fix for any $r$ a $C^r$ norm on the functions supported
in $K$, such that this norm is $\Gamma$-invariant. Such a norm
is not canonically defined, but this will not be a problem in
the statements or results to follow since multiplicative
constants do not matter.

\begin{lem}
\label{lem_equiv_norms}
There exists a constant $C(k,\ell,K)$ such that any smooth
function $f$ supported in $K$ satisfies the following property.
For any $x\in K$, any $C^{k+\ell}$ vector fields $v_1,\dots,
v_\ell$ defined on a neighborhood of $W^u_{1/100}(x)$ with
$\norm{v_j}_{C^{k+\ell}} \leq 1$, and any $C^{k+\ell}$ function
$\phi$, compactly supported on $W^u_{1/200}(x)$ with
$\norm{\phi}_{C^{k+\ell}} \leq 1$,
  \begin{equation*}
  \left| \int_{W^u_{1/200}(x)} \phi\cdot L_{v_1}\cdots L_{v_\ell} f\dd\mu_u\right|
  \leq C \sum_{\ell'\leq \ell} e_{k,\ell'}(f).
  \end{equation*}
\end{lem}
The interest of this lemma is that the vector fields $v_j$ can
be any vector fields, not only canonical extensions of vector
fields pointing in the stable direction or in the flow
direction. Moreover, we also weaken the smoothness of the
vector fields $v_j$, requiring them only to be $C^{k+\ell}$
instead of $C^{k+\ell+1}$.
\begin{proof}
We prove the statement of the lemma by induction on $\ell$. For $\ell=0$,
this is clear from the definitions.  Let us
decompose the vector field $v_1$ as $v_1^u + v_1^s + v_1^\omega$
where those three components point, respectively, in the unstable
direction, in the stable direction and in the flow direction. Along
$W^u_{1/100}(x)$, decomposing $v_1^s$ along coordinates vector
fields, we can write it as a linear combination of vector fields of
the form $\psi_1^s w_1^s$ where $\psi_1^s$ is a function bounded in
$C^{k+\ell}$ and $w_1^s$ is a $C^{\infty}$ vector field with
$\norm{w_1^s}_{C^{k+\ell+1}}\leq C$. To simplify notations, we will
omit a summation and assume that we can write $v_1^s(y)=\psi_1^s(y)
w_1^s(y)$. In the same way, we write $v_1^\omega(y)=\psi_1^\omega(y)
\omega(y)$ where $\norm{\psi_1^\omega}_{C^{k+\ell}}\leq C$. For
convenience, we introduce the notation $w_1^\omega = \omega$.

Let $g= L_{v_2}\dots L_{v_\ell} f$. Since $L_{v_1} g$ only
depends on the value of the vector field $v_1$ (and not its
derivatives), we have, along $W^u_{1/200}(x)$, $L_{v_1} g =
L_{v_1^u} g + \psi_1^s L_{\extension{w_1^s}} g+\psi_1^\omega
L_{\extension{w_1^\omega}}g$. Moreover,
  \begin{equation*}
  \int_{W^u_{1/200}(x)} \phi \cdot L_{v_1^u} g \dd\mu_u
  = -\int_{W^u_{1/200}(x)} L_{v_1^u}\phi \cdot g\dd\mu_u,
  \end{equation*}
which is bounded by $C \sum_{\ell'\leq \ell-1} e_{k,\ell'}(f)$
by the induction hypothesis, since the function $L_{v_1^u}\phi$
is $C^{k+\ell-1}$ and is multiplied by $\ell-1$ derivatives of
$f$ against $C^{k+\ell-1}$ vector fields.

It remains to bound $\int_{W^u_{1/200}(x)} \phi
\psi_1^{\alpha_1}\cdot
L_{\extension{w_1^{\alpha_1}}}L_{v_2}\cdots L_{v_\ell}
f\dd\mu_u$, for some $\alpha_1 \in \{s,\omega\}$. Let us
exchange the vector fields to put
$L_{\extension{w_1^{\alpha_1}}}$ in the last position. Since
$[L_v, L_w]=L_{[v,w]}$, the error we make is bounded by the
integral of a $C^{k+\ell}$ function multiplied by $\ell-1$
derivatives of $f$ against $C^{k+\ell-1}$ vector fields. By the
induction hypothesis, this is again bounded by $C
\sum_{\ell'\leq \ell-1} e_{k,\ell'}(f)$.

It remains to bound $\int_{W^u_{1/200}(x)}
\phi\psi_1^{\alpha_1}\cdot L_{v_2}\cdots L_{v_\ell}
L_{\extension{w_1^{\alpha_1}}}f\dd\mu_u$. In the same way as
above, we decompose $v_2$ into its unstable, stable and flow
part, integrate by parts to get rid of the unstable part, and
exchange the vector fields to put the remaining parts of $v_2$
at the end. Iterating this process $\ell$ times, we end up with
an estimate
  \begin{multline*}
  \left| \int_{W^u_{1/200}(x)} \phi\cdot L_{v_1}\cdots L_{v_\ell} f\dd\mu_u\right|
  \\
  \leq C \sum_{\ell'\leq \ell-1} e_{k,\ell'}(f)
  + C\sup_{\alpha\in \{s,\omega\}^\ell}
  \left|\int_{W^u_{1/200}(x)} \phi \psi_1^{\alpha_1}\cdots \psi_\ell^{\alpha_\ell}
  \cdot L_{\extension{w_1^{\alpha_1}}}\cdots L_{\extension{w_\ell^{\alpha_\ell}}} f\dd\mu_u\right|.
  \end{multline*}
By construction, the vector fields $\extension{w_j^{\alpha_j}}$
are canonical extensions of $C^{k+\ell+1}$ vector fields
defined along $W^u_{1/200}(x)$ and pointing in the stable or
flow direction. Therefore, the latter integrals are bounded by
$C e_{k,\ell}(f)$ by definition of this coefficient.
\end{proof}

\begin{proof}[Proof of Proposition \ref{prop_compactness}]
The first step of the proof is to show that, to estimate
$\normdeux{f}_{k}$, it is sufficient to work with finitely many
unstable manifolds. More precisely, we will show that, for any
$\epsilon>0$, there exist finitely many points $(x_i)_{i\in I}$
such that, for any function $f$ supported in $K$ and
$\Gamma$-invariant,
  \begin{equation}
  \label{compactprime}
  \normdeux{f}_k \leq C \epsilon \norm{f}_k + C \sup \left|\int_{W^u_{1/200}(x_i)}
  \phi \cdot L_{v_1}\cdots L_{v_\ell} f \dd\mu_u \right|,
  \end{equation}
where the supremum is taken over all $0\leq \ell\leq k-1$, all
$i\in I$, all functions $\phi$ compacly supported on $W^u_{1/200}(x_i)$
and all vector fields $v_j$
defined in some fixed neighborhood $\mathcal {U}_i$
of $W^u_{1/100}(x_i)$ with $C^{k+\ell+1}$ norm
bounded by $1$.

Since $K/\Gamma$ is compact, it is sufficient to show that
integrals along the unstable manifold of a point $x_1$ can be
controlled by similar integrals along the unstable manifold of
a nearby point $x_0$. Let $x_0$, $x_1$ be two nearby points in
$K$ (so that their unstable spaces $E^u(x_0)$ and $E^u(x_1)$
are also close). Consider a smooth path $x_t$ from $x_0$ to
$x_1$, and a smooth family of maps sending $E^u(x_0)$ to
$E^u(x_t)$. Parameterizing locally the (affine) unstable
manifold of the point $x_t$ by its tangent space (by the map
$\Psi_{x_t}$ introduced before Proposition \ref{Psigood}), we obtain
a family of affine maps $\Phi_t : W^u_{1/50}(x_0) \to W^u(x_t)$ with
$\Phi_0=\id$, that we extend smoothly to diffeomorphisms
defined on a neighborhood of $W^u_{1/50}(x_0)$.

Fix $0 \leq \ell \leq k-1$ and consider a $C^{k+\ell+1}$ function $\phi$
compactly supported on
$W^u_{1/400}(x_1)$, and $C^{k+\ell+1}$ vector fields
$v_1,\dots,v_\ell$ along $W^u_{1/50}(x_1)$, each of them pointing
either in the stable direction or in the flow direction, with $C^{k+\ell+1}$
norm bounded by $1$.  We
want to bound the integral
  \begin{equation*}
  I_1=\int_{W^u(x_1)} \phi \cdot L_{\extension{v_1}}\cdots L_{\extension{v_\ell}}f \dd\mu_u,
  \end{equation*}
using data along $W^u(x_0)$.

For each $t$, we define vector
fields $v_j^t$ on a neighborhood of $W^u_{1/75}(x_t)$ by $v_j^0 =
(\Phi_1)^* \extension{v_j}$, and $v_j^t= (\Phi_t)_* v_j^0$.
Letting $J_t\in (0,+\infty)$ be the jacobian of $\Phi_t$ from
$W^u(x_0)$ to $W^u(x_t)$, we can rewrite $I_1$ as a sum of two
terms
  \begin{multline*}
  I_1=\int_{W^u(x_0)} \phi \circ \Phi_1 \cdot L_{v_1^0}\cdots L_{v_\ell^0} (f\circ \Phi_1) J_1 \dd \mu_u
  = \int_{W^u(x_0)} \phi \circ \Phi_1 \cdot L_{v_1^0}\cdots L_{v_\ell^0} f \cdot J_1 \dd \mu_u
  \\
  +\int_{t=0}^1 \frac{\partial}{\partial t} \left( \int_{W^u(x_0)}
  \phi \circ \Phi_1 \cdot L_{v_1^0}\cdots L_{v_\ell^0} (f\circ \Phi_t) \cdot J_1 \dd \mu_u\right)
  \dd t.
  \end{multline*}
The first term is bounded by the second term in the right hand side
of \eqref{compactprime}. Writing $w_t = (\partial \Phi_t/\partial
t)\circ \Phi_t^{-1}$, the integrand of the second term at fixed $t$
is
  \begin{multline*}
  \int_{W^u(x_0)}
  \phi \circ \Phi_1 \cdot L_{v_1^0}\cdots L_{v_\ell^0}
  ((L_{w_t} f)\circ \Phi_t) \cdot J_1 \dd \mu_u
  \\
  =\int_{W^u(x_t)}
  \phi \circ \Phi_1 \circ \Phi_t^{-1}\cdot  L_{v_1^t}\cdots L_{v_\ell^t}
  L_{w_t} f \cdot J_1 J_t^{-1}
  \dd\mu_u.
  \end{multline*}
This is an integral along an unstable manifold of a $C^{k+\ell+1}$
function multiplied by $\ell+1$ derivatives of $f$ against
$C^{k+\ell+1}$ vector fields. By Lemma \ref{lem_equiv_norms} (applied
to $\ell'=\ell+1$, which is licit since $\ell<k$ by assumption), this
is bounded in terms of $\norm{f}_k$. Moreover, if $x_0$ and $x_1$ are
close enough, the $C^{k+\ell+1}$ norm of the vector field $w_t$ is
arbitrarily small, and we get that this integral is bounded by
$C\epsilon \norm{f}_k$.  Putting together the two terms, we see that
$I_1$ is bounded by the right hand side of \eqref{compactprime}. Up
to constants (which do depend on $K$), the norm $\normdeux {f}_k$ is
defined using integrals similar to $I_1$, but where $\phi$ is allowed
to have a larger support $W^u_{1/200}(x_1)$ and the $v_j$ may have a
smaller domain of definition $W^u_{1/100}(x_1)$.  However, this is
not a problem, since those more general integrals can be decomposed
as sums of a bounded number of integrals like $I_1$, using partitions
of unity. This concludes the proof of \eqref{compactprime}.

\smallskip

It is now easy to conclude the proof.  Fix smooth bump functions
$\partition_i$ compactly supported in $\mathcal {U}_i$ (the domain of
definition of the $v_j$ in \eqref{compactprime}) and equal to $1$ in
a neighborhood of $W^u_{1/200}(x_i)$.  Since $C^{k+\ell+1}$ is
compactly included in $C^{k+\ell}$, for each $x_i$, $i \in I$, we can
choose finitely many functions $\phi_{m,i}$ compactly supported in
$W^u_{1/200}(x_i)$ and finitely many vector fields $v_{j,m,i}$
defined in $\mathcal {U}_i$, such that for  all functions $\phi$ and
vector fields $v_j$ which are bounded by $1$ in $C^{k+\ell+1}$, there
exists $m$ such that $\phi$ and $\partition_i v_j$ are
$\epsilon$-close to $\phi_{m,i}$ and $\partition_i v_{j,m,i}$ in
$C^{k+\ell}$.  By Lemma \ref{lem_equiv_norms}, this gives with
\eqref{compactprime}
  \begin{equation*}
  \normdeux{f}_k \leq C'\epsilon \norm{f}_k + \sup_{i,m}
  \left|\int_{W^u_{1/200}(x_i)} \phi_{m,i} \cdot L_{v_{1,m,i}}\cdots
  L_{v_{\ell,m,i}} f \dd\mu_u\right|.
  \end{equation*}

Consider now a sequence $f_n$ with $\norm{f_n}_k\leq 1$. We extract a
subsequence $f_{j(n)}$ along which all the finitely many quantities
$\int_{W^u_{1/200}(x_i)} \phi_{m,i} \cdot L_{v_{1,m,i}}\cdots
L_{v_{\ell,m,i}} f_{j(n)} \dd\mu_u$ converge. It follows that
$\limsup_{n,n'\to \infty} \normdeux{f_{j(n)} -f_{j(n')}}_k \leq
2C'\epsilon$. Letting $\epsilon$ tend to $0$ and using a standard
diagonal argument, we get the required Cauchy sequence.
\end{proof}

\section{A good bound on the essential spectral radius of
\texorpdfstring{$\boM$}{M}}

\label{sec_first_step}

Let $\tilde\mu$ be an admissible measure with its affine local
product structure, supported by a submanifold $X$ of $\Teich_1$. In
this section, we prove Theorem
\ref{thm_control_spectral_radius_annonce}. As in the statement of
this theorem, let us write $\boM f = \int_{t=0}^\infty e^{-4\delta t}
\boL_t f\dd t$ (to be interpreted as explained in \S
\ref{subsec_main_result_first_step}), where $\delta>0$ is fixed and
$\boL_t f=f\circ g_t$.

To prove Theorem \ref{thm_control_spectral_radius_annonce}, we have
to construct a good norm on $\boD^\Gamma$. It turns out that the
norms $\norm{\cdot}_k$ that we have constructed in the previous section in
\eqref{defnorm} are suitable for this purpose. The following
statement contains Theorem \ref{thm_control_spectral_radius_annonce}
(see also Remark \ref{continuity f_1}).

\begin{thm}
\label{thm_control_spectral_radius}
For all $k$, there exists $C>0$ such that $\norm{\boL_t f}_k
\leq C \norm{f}_k$, uniformly in $t\geq 0$. Therefore, $\boM$
acts continuously on the completion of $\boD^\Gamma$ for the
norm $\norm{\cdot}_k$.

Moreover, if $k$ is large enough, then the essential spectral
radius of $\boM$ on this space is at most $1+\delta$.
\end{thm}

This section is devoted to the proof of this result. Until the end of
its proof, we will always specify if a constant depends on $k$, by
using a subscript as in $C_k$. Most constants will be independent of
$k$, and this will be very important for the argument, since $k$ will
be chosen only at the very end of the proof.

\smallskip

For technical reasons, it is convenient to work with another norm
that is equivalent to $\norm{\cdot}_k$. For $A\geq 1$, let us first define
a norm equivalent to $\norm{\cdot}_{C^k}$, by
$\norm{f}_{C_A^k}=\sum_{j=0}^k c_j(f)/ (j! A^j)$. Since $c_j(fg)\leq
\sum_{m=0}^j \binom j m c_m(f) c_{j-m}(g)$, it follows that
$\norm{fg}_{C_A^k} \leq \norm{f}_{C_A^k} \norm{g}_{C_A^k}$. Moreover,
for any fixed $C^k$ function $f$ and any $\epsilon>0$, if $A$ is
large enough, then $\norm{f}_{C_A^k} \leq (1+\epsilon) \sup |f|$. Let
us define $e_{k,\ell,\alpha}^A(f;x)$ like $e_{k,\ell,\alpha}(f;x)$,
but replacing the requirements $\norm{\phi}_{C^{k+\ell}}\leq 1$ and
$\norm{v_j}_{C^{k+\ell+1}} \leq 1$ (for the supremum taken in
\eqref{defekl}) by $\norm{\phi}_{C_A^{k+\ell}}\leq 1$ and
$\norm{v_j}_{C_A^{k+\ell+1}} \leq 1$.

We will need to deal separately with the case where all the
vector fields in the definition of $e_{k,\ell,\alpha}^A$ point
in the stable direction, and the case where at least one vector
field points in the flow direction. Let us therefore define
$e_{k,\ell, s}^A(f;x)=e_{k,\ell, \{s,\dots,s\}}^A(f;x)$, and
$e_{k,\ell,\omega}^A(f;x) = \sup e_{k,\ell,\alpha}^A(f;x)$,
where the supremum is over all $\alpha\in \{s,\omega\}^\ell$
different from $\{s,\dots,s\}$. Let $e_{k,\ell,s}^A(f)=\sup_x
e_{k,\ell,s}^A(f;x)$, and similarly for
$e_{k,\ell,\omega}^A(f)$. For $B\geq 1$, let
$\norm{f}^{A,B}_{k,s}=\sum_{\ell=0}^k B^{-\ell}
e_{k,\ell,s}^A(f)$, and similarly for
$\norm{f}^{A,B}_{k,\omega}$. Finally, let $\norm{f}_k^{A,B} =
\norm{f}_{k,s}^{A,B} + \norm{f}_{k,\omega}^{A,B}$. This norm is
equivalent to $\norm{f}_k$, but more convenient for a lot of
inequalities.

In the statements below, when we say ``for all large enough
$A,B$...'', we mean: if $A$ is large enough, then, if $B$ is
large enough (possibly depending on $A$), then... The
assumption ``for all large enough $k,A,B$'' should be
interpreted in the same way.

\bigskip

We now start the proof. Some arguments are borrowed from
\cite{gouezel_liverani}. We write $\boD$ for the set of
$C^\infty$ functions supported in a compact set mod $\Gamma$.
It contains the previously defined set $\boD^\Gamma$ of
functions in $\boD$ that are $\Gamma$-invariant.

\begin{lem}
\label{lemme_basic_estimate}
There exists a constant $C_0\geq 1$ satisfying the following
property. For every $k,\ell\in \N$ and every $\alpha\in
\{s,\omega\}^\ell$, if $A$ is large enough, then for every
$t\geq 0$, every $f\in \boD$ and every $x\in X$,
  \begin{equation}
  \label{borne_ekl_triviale}
  e_{k,\ell,\alpha}^A(f\circ g_t;x) \leq C_0 e_{k,\ell,\alpha}^A(f) \left( e^{-(1-2\delta)t} + 1/V_\delta(x)\right).
  \end{equation}
\end{lem}
\begin{proof}
We first give the proof for $\ell=0$.

Fix some point $x$, and some compactly supported function $\phi
: W^u_{1/200}(x) \to \C$ with $\norm{\phi}_{C^k_A}\leq 1$. We
want to estimate $\int_{W^u_{1/200}(x)} \phi(y) \cdot f\circ
g_t(y) \dd\mu_u(y)$. We change variables, letting $z=g_t(y)$.
By Proposition \ref{prop_local_product}, the resulting jacobian
has the form $e^{-dt}$ for some $d>0$. The integral becomes an
integral over $g_t(W^u_{1/200}(x))$. Proposition
\ref{prop_partition} provides a partition of unity
$(\partition_i)_{i\in I}$ on this set, with good properties. In
particular, $\partition_i$ is supported in a ball
$W^u_{1/200}(x_i)$. The integral becomes
  \begin{equation*}
  \sum_i \int_{W^u_{1/200}(x_i)} \partition_i(z) \phi(g_{-t}z)\cdot f(z) \, e^{-dt}\dd\mu_u(z).
  \end{equation*}
Since $g_{-t}$ is affinely contracting along $W^u$,
$\norm{\phi\circ g_{-t}}_{C^k_A}\leq \norm{\phi}_{C^k_A}\leq
1$. Therefore, the $C^k_A$ norm of $\partition_i \cdot
\phi\circ g_{-t}$ is bounded by $\norm{\partition_i}_{C^k_A}$.
If $A$ is large enough, this is at most $2$ (since the
coefficients $c_m$ of $\partition_i$, for $1\leq m\leq k$, are
uniformly bounded by Proposition \ref{prop_partition}).
Therefore, the above integral is bounded by
  \begin{multline}
  \label{above integral}
  \sum_i C e_{k,0}^A(f) V_\delta(x_i) \mu_u(W^u_{1/200}(x_i)) e^{-dt}
  \\
  \leq   C e_{k,0}^A(f) \sum_i \int_{W^u_{1/200}(x_i)} V_\delta(z) \, e^{-dt}\dd\mu_u(z),
  \end{multline}
since $\log V_\delta$ is Lipschitz by Proposition
\ref{prop_Vdelta_OK_sur_Wu}. The covering multiplicity of the sets
$W^u_{1/200}(x_i)$ is uniformly bounded, by Proposition
\ref{prop_partition}. Moreover, all those sets are included in $\{ z
\st d(z, g_t (W^u_{1/200}(x))) \leq 1/200\}$, which is itself
included in $g_t(W^u_{1/100}(x))$ since $g_{-t}$ contracts the
distance along $W^u$. Therefore, \eqref{above integral} is bounded by
  \begin{equation*}
  C e_{k,0}^A(f) \int_{g_t(W^u_{1/100}(x))} V_\delta(z) \, e^{-dt} \dd\mu_u(z)
  =Ce_{k,0}^A(f) \int_{W^u_{1/100}(x)} V_\delta(g_t y) \dd\mu_u(y).
  \end{equation*}
By Proposition \ref{prop_Vdelta_OK_sur_Wu}, this is bounded by
$C e_{k,0}^A(f) \mu_u(W^u_{1/100}(x))
(e^{-(1-2\delta)t}V_\delta(x) + 1)$. Finally,
  \begin{multline*}
  \frac{1}{V_\delta(x)} \frac{1}{\mu_u(W^u_{1/200}(x))}\left|\int_{W^u_{1/200}(x)}
  \phi \cdot f \circ g_t \dd\mu_u(y) \right|
  \\
  \leq C e_{k,0}^A (f) \frac{\mu_u(W^u_{1/100}(x))}{\mu_u(W^u_{1/200}(x))} \left( e^{-(1-2\delta)t} + 1/V_\delta(x)\right).
  \end{multline*}
The ratio of the measures is bounded, by Corollary
\ref{cor_doubling}. This proves \eqref{borne_ekl_triviale} when
$\ell=0$.

Assume now $\ell>0$, we have to estimate
  \begin{equation}
  \label{qsmiufqsdfui}
  \int_{W^u_{1/200}(x)} \phi \cdot L_{\extension{v_1}}\cdots L_{\extension{v_\ell}}(f\circ g_t)
  \dd\mu_u,
  \end{equation}
where the vector fields $v_j$ are defined on $W^u_{1/100}(x)$,
satisfy $\norm{v_j}_{C^{k+\ell+1}_A} \leq 1$, and point in the
direction $E^s$ or $\R\omega$. Consider a function $\partition$ equal
to $1$ in $W^u_{1/200}(x)$ and compactly supported in
$W^u_{1/100}(x)$ (as constructed in Lemma \ref{lem_cutonW}), and
define a new vector field $v_{j,1}=\partition\cdot v_j$. It coincides
with $v_j$ on $W^u_{1/200}(x)$, therefore the integral
\eqref{qsmiufqsdfui} can also be written using $v_{j,1}$ instead of
$v_j$. Moreover, if $A$ is large enough, the definition of the
$C_A^{k+\ell+1}$ norm ensures that
  \begin{equation*}
  \norm{v_{j,1}}_{C^{k+\ell+1}_A}
  = \norm{ \partition \cdot v_j}_{C^{k+\ell+1}_A}
  \leq \norm{ \partition}_{C^{k+\ell+1}_A}
	   \norm{ v_j       }_{C^{k+\ell+1}_A}
  \leq 2^{1/\ell}.
  \end{equation*}

Let $w_j=(g_t)_* v_{j,1}$. Since the extension
$\extension{w_j}$ is defined using the affine structure and the
flow direction, which are invariant under the affine flow
$g_t$, it follows that $\extension{w_j}=(g_t)_*
\extension{v_{j,1}}$. Therefore,
  \begin{equation*}
  L_{\extension{v_{1,1}}}\cdots L_{\extension{v_{\ell,1}}}(f\circ g_t)(y)
  =L_{\extension{w_1}}\cdots L_{\extension{w_\ell}} f (g_t y).
  \end{equation*}
We claim that the vector fields $w_j$ are bounded by
$2^{1/\ell}$ in $C^{k+\ell+1}_A$ (even better, $c_m(w_j) \leq
c_m(v_{j,1})$ for all $m$). We can then proceed as in the
$\ell=0$ case, getting simply an additional error factor equal
to $\prod_{j=1}^\ell \norm{w_j}_{C^{k+\ell+1}_A}\leq 2$. One
should pay attention to the fact that, with the above
definition, the vector fields $w_j$ are not always defined on
all the balls $W_{1/100}(x_i)$, for those $x_i$ that are close
to the boundary of $g_t( W_{1/200}(x))$. This is not a problem
since $w_j$ is compactly supported in $g_t( W_{1/100}(x))$ by
construction: one may therefore extend it by $0$ wherever it is
not defined (this is why we had to use $v_{j,1}$ and not $v_j$
in this construction).

It remains to check the formula $c_m(w_j) \leq c_m(v_{j,1})$. It
comes from the fact that the definition of $c_m$ only involves
differentiation along directions in $W^u$, and that $g_{-t}$ is
contracting along this manifold. If $\alpha_j=\omega$, i.e., $v_j$
points in the flow direction, this estimate is
straightforward.
Let us therefore assume that $\alpha_j=s$, i.e., $v_j$ points in the
stable direction. Consider a point $z$ in the domain of definition of
$w_j$, and $m$ vectors $u_1,\dots,u_m$ at that point which are
tangent to $W^u(x)$, with $\norm{u_m}_z \leq 1$. Write $y=g_{-t}z$.
We get
  \begin{equation*}
  D^m w_j(z; u_1, \dots, u_m)
  =e^{-t} D^m v_{j,1}( g_{-t}z; D g_{-t}(z)\cdot u_1,\dots, Dg_{-t}(z)\cdot u_m).
  \end{equation*}
Therefore,
  \begin{align*}
  \norm{ D^m w_j(z; u_1, \dots, u_m)}_y & = e^{-t} \norm{D^m v_{j,1}
	( g_{-t}z; D g_{-t}(z)\cdot u_1,\dots, D g_{-t}(z)\cdot u_m)}_y
  \\&
  \leq e^{-t} c_m(v_{j,1}) \norm{D g_{-t}(z) u_1}_y \cdots \norm{D g_{-t}(z) u_m}_y.
  \end{align*}
Since the differential $D g_{-t}(z)$ contracts in the direction
of $W^u$ by Lemma \ref{weakly_hyperbolic}, we have $\norm{D
g_{-t}(z) u_n}_y \leq \norm{u_n}_z\leq 1$. This yields
  \begin{equation}
  \label{mlwkjxcv}
  \norm{ D^m w_j(z; u_1, \dots, u_m)}_y \leq
  e^{-t} c_m(v_{j,1}).
  \end{equation}
We are interested in bounding $\norm{ D^m w_j(z; u_1, \dots,
u_m)}_z$. Since $d(y,z)\leq |t|$ by Lemma \ref{SL_lipschitz},
Proposition \ref{norm_slowly_varies} shows that the ratio
between $\norm{\cdot}_y$ and $\norm{\cdot}_z$ is at most $e^t$.
This cancels the factor $e^{-t}$ in \eqref{mlwkjxcv}, and we
get the conclusion.
\end{proof}

\begin{cor}
For every $k\in \N$, for every large enough $A$ and $B$, for
every $t\geq 0$ and every $f\in \boD$, we have $\norm{f\circ
g_t}_{k}^{A,B} \leq 2C_0 \norm{f}_{k}^{A,B}$.
\end{cor}
\begin{proof}
The function $V_\delta$ is bounded from below by $1$. Taking
the supremum over $x$ in \eqref{borne_ekl_triviale}, we get
$e_{k,\ell,\alpha}^A(f\circ g_t) \leq 2C_0
e_{k,\ell,\alpha}^A(f)$. The result follows from the definition
of the $\norm{\cdot}_{k}^{A,B}$ norm.
\end{proof}

It follows from this corollary that we can define the operator
$\boM$ on $\boDb$. Let $N\in \N$, we will study the norm of
$\boM^N$. We have
  \begin{equation}
  \label{formuleMN}
  \boM^N f=\int_{t=0}^\infty
  \frac{t^{N-1}}{(N-1)!} e^{-4\delta t} \boL_t f \dd t.
  \end{equation}

We will estimate differently the contributions $\norm{\boM^n
f}_{k,\omega}^{A,B}$ and $\norm{\boM^n f}_{k,s}^{A,B}$ to
$\norm{\boM^n f}_{k}^{A,B}$. Let us first deal with the former.

\begin{lem}
\label{lem_omega}
For any $N\in \N$, for any $k$, if $A$ and $B$ are large
enough, we have
  \begin{equation*}
  \norm{\boM^N f}_{k,\omega}^{A,B} \leq 5C_0 \norm{f}_k^{A,B}.
  \end{equation*}
\end{lem}
\begin{proof}
We will prove that, for any $N,k,\ell$ and $A$ sufficiently large,
there exists a constant $C_{N,k,\ell,A}$ such that
  \begin{equation}
  \label{mlkjvxwlvcjlwlxjcvv}
  e^A_{k,\ell,\omega}(\boM^N f) \leq C_{N,k,\ell,A} \sum_{\ell'<\ell} e_{k,\ell'}^A(f)
  + 4C_0 e^A_{k,\ell}(f).
  \end{equation}
Taking $B$ much larger than all $C_{N,k,\ell,A}$ for $0\leq
\ell \leq k$, this implies directly the statement of the lemma.

Let us fix $N, k,\ell,A$. We split $\boM^N$ as the sum of
$\boM_1\coloneqq \int_0^D \frac{t^{N-1}}{(N-1)!} e^{-4\delta
t}\boL_t \dd t$ and $\boM_2\coloneqq \int_{D}^\infty
\frac{t^{N-1}}{(N-1)!} e^{-4\delta t}\boL_t \dd t$, where $D$
is suitably large.

Lemma \ref{lemme_basic_estimate} shows that
$e^A_{k,\ell,\omega}(\boL_t f) \leq 2 C_0 e^A_{k,\ell,\omega}(f)$.
Hence, if $D$ is large enough (depending on $N$), we have
$e^A_{k,\ell,\omega}(\boM_2 f) \leq C_0 e^A_{k,\ell,\omega}(f)$. The
term $\boM_2$ is therefore not a problem to prove
\eqref{mlkjvxwlvcjlwlxjcvv}.

Let us handle $\boM_1$. Consider first a point $x$ such that
$V_\delta(x) \geq e^{(1-2\delta)D}$. For such a point $x$, Lemma
\ref{lemme_basic_estimate} gives $e^A_{k,\ell,\omega}(\boL_t f; x)
\leq 2C_0 e^A_{k,\ell,\omega}(f) e^{-(1-2\delta)t}$ for $t\leq D$. In
particular,
  \begin{align*}
  e^A_{k,\ell,\omega}(\boM_1 f;x)
  &
  \leq \int_{t=0}^D \frac{t^{N-1}}{(N-1)!} e^{-4\delta t}
	e^A_{k,\ell,\omega}(\boL_t f;x)\dd t
  \\&
  \leq 2C_0 \int_{t=0}^D \frac{t^{N-1}}{(N-1)!} e^{-4\delta t}
	e^A_{k,\ell,\omega}(f) e^{-(1-2\delta)t}\dd t
  \leq 2C_0 e^A_{k,\ell,\omega}(f)
  \end{align*}
since $\int_{t=0}^\infty \frac{t^{N-1}}{(N-1)!}
e^{-(1+2\delta)t}\dd t \leq \int_{t=0}^\infty
\frac{t^{N-1}}{(N-1)!} e^{-t}\dd t = 1$. This concludes the
proof for such points $x$.

It remains to consider points $x$ with $V_\delta(x) \leq
e^{(1-2\delta)D}$. This set is very large if $D$ is large, but it is
compact mod $\Gamma$. Fix such a point $x$, we have to estimate
integrals of the form
  $
  \int_{W^u_{1/200}(x)} \phi \cdot L_{\extension{v_1}}\cdots L_{\extension{v_\ell}}(\boM_1 f)\dd\mu_u
  $,
where $\norm{\phi}_{C^{k+\ell}_A}\leq 1$ and
$\norm{v_j}_{C^{k+\ell+1}_A}\leq 1$, and at least one of the
vector fields $v_j$ points in the flow direction.

To begin, assume that the last vector field $v_\ell$ points in
the flow direction, i.e., $v_\ell(y)=\psi(y) \omega(y)$ for
some function $\psi$ with $\norm{\psi}_{C^{k+\ell+1}_A}\leq 1$.
In the expression $L_{\extension{v_1}}\cdots
L_{\extension{v_{\ell-1}}}( \psi L_\omega (\boM_1 f))$, if we
use at least one of the Lie derivatives to differentiate
$\psi$, we obtain a term bounded by $C_{k,\ell,A}
e^A_{k,\ell'}(\boM_1 f)$ for some $\ell'<\ell$. This is bounded
by $C_{N,k,\ell,A,D} e^A_{k,\ell'}(f)$ by Lemma
\ref{lemme_basic_estimate}. This error term is compatible with
\eqref{mlkjvxwlvcjlwlxjcvv}. The remaining term is
  $
  \psi L_{\extension{v_1}}\cdots
  L_{\extension{v_{\ell-1}}}( L_\omega
  (\boM_1 f))
  $.
Since $\boM_1 f=\int_{t=0}^D h(t) \boL_t f \dd t$ for some
smooth function $h$, we have $L_\omega(\boM_1 f)=h(D) \boL_D f
- h(0) f-\int_{t=0}^D h'(t)\boL_t f\dd t$. Therefore, the
integral we are studying can be bounded in terms of $\ell-1$
derivatives of $f$ (or images of $f$ under operators $\boL_t$),
and this is bounded by $C_{N,k,\ell,A,D} e^A_{k,\ell - 1}(f)$.
This error term is again compatible with
\eqref{mlkjvxwlvcjlwlxjcvv}.

Assume now that one of the vector fields $v_j$ points in the
flow direction, but that it is not necessarily the last one. We
can exchange the vector fields to put the vector field
$\extension{v_j}$ in the last position and conclude as above.
Since $[L_{w_1}, L_{w_2}]=L_{[w_1,w_2]}$, the additional error
corresponds to the integration of $\ell-1$ derivatives of
$\boM_1 f$ against a $C^{k+\ell}$ function, but one of the
vector fields is not the canonical extension of a vector field
defined on $W^u_{1/100}(x)$. Since we work in the set
$\{V_\delta \leq e^{(1-2\delta)D}\}$ which is compact mod
$\Gamma$, Lemma \ref{lem_equiv_norms} shows that this error is
bounded in terms of $\sup_{\ell'<\ell} e_{k,\ell'}(f)$, and is
again compatible with \eqref{mlkjvxwlvcjlwlxjcvv}.
\end{proof}

It remains to study $\norm{\boM^N f}_{k,s}^{A,B}$. We will
rather estimate $\norm{\boL_t f}_{k,s}^{A,B}$ if $t$ is large
enough, this will readily gives estimates for $\norm{\boM^N
f}_{k,s}^{A,B}$ by \eqref{formuleMN}.

Let us fix some constants. First, we recall that $C_0$ has been
defined in Lemma \ref{lemme_basic_estimate}. Let $T_0>0$ be
large enough so that $40C_0 \leq e^{\delta T_0}$. Let
$V=2e^{(3-2\delta)T_0}$, and define
  \begin{equation}
  \label{defK}
  K=\{x\in \Teich_1 \st V_\delta(x) \leq 4V e^{2T_0}\}.
  \end{equation}
This set is compact mod $\Gamma$. Finally, applying Proposition
\ref{pliss_hyperbolicity} to $K$, we get a time $T=T(K)$.

We will study the operator $\boL_{nT_0}$, for all $n$ large
enough so that $nT_0 \geq T/\delta$. By Lemma \ref{lem_phiV},
we can define a $C^\infty$ function $\partition_V$ such that
$\partition_V(x)=1$ if $V_\delta(x)\leq V$ and
$\partition_V(x)=0$ if $V_\delta(x)\geq 2V$. Write
$\psi_1=\partition_V$ and $\psi_2 = 1-\partition_V$ so that
$\psi_1+\psi_2=1$. We decompose $\boL_{T_0}(f) =
\boL_{T_0}(\psi_1 f)+ \boL_{T_0}(\psi_2 f) = \tilde\boL_1 f +
\tilde \boL_2 f$. Therefore, $\boL_{nT_0} = \sum_{\gamma \in
\{1,2\}^n} \tilde\boL_{\gamma_1} \cdots \tilde
\boL_{\gamma_n}$.

We first give a lemma ensuring that the multiplication by
$\partition_V$ or $1-\partition_V$ in the definition of
$\tilde\boL_1$ and $\tilde\boL_2$ is not harmful, and then we will
turn to the study of $\tilde\boL_{\gamma_1} \cdots \tilde
\boL_{\gamma_n}$ for $\gamma \in \{1,2\}^n$. We will handle in Lemma
~ \ref{lem_close_infinity} the case where most $\gamma_i$ are equal
to $2$ (i.e., most time is spent close to infinity, and we can use
the good recurrence estimates of
Proposition~\ref{prop_Vdelta_OK_sur_Wu}), and in
Lemma~\ref{maincontraction} the case where a definite proportion of
the $\gamma_i$ is equal to $1$ (i.e., some time is spent in the
compact set $K$, and we can take advantage of the hyperbolicity of
the flow there).

\begin{lem}
\label{largeB_hence_OK}
Let $k\in \N$, and let $\psi:X \to [0,1]$ be a $C^{2k}$
function supported in a compact set mod $\Gamma$. If $A$ and
$B$ are large enough, then for any $f\in \boD$, we have
$\norm{\psi f}_{k,s}^{A,B} \leq 3 \norm{f}_{k,s}^{A,B}$.
\end{lem}
\begin{proof}
Let us prove that, for every $k$,
$\ell \leq k$ and every large enough
$A$, there exists a constant $C_{k,\ell,A}$ such that, for any
$f\in \boD$,
  \begin{equation}
  \label{lkjsmlkjsfd}
  e_{k,\ell,s}^A(\psi f) \leq 2 e_{k,\ell,s}^A(f)+ C_{k,\ell,A} \sum_{\ell'<\ell}e_{k,\ell',s}^A(f).
  \end{equation}
The statement of the lemma follows directly from this estimate
if $B$ is much larger than any of the $C_{k,\ell,A}$.

To estimate $e_{k,\ell,s}^A(\psi f)$, we have to compute
integrals of the form
  \begin{equation*}
  \int_{W^u_{1/200}(x)} \phi \cdot L_{\extension{v_1}}\cdots L_{\extension{v_\ell}}(\psi f)\dd \mu_u,
  \end{equation*}
where $\norm{\phi}_{C^{k+\ell}_A} \leq 1$ and
$v_1,\dots,v_\ell$ have a $C_A^{k+\ell+1}$--norm along
$W^u_{1/100}(x)$ bounded by $1$. We can use each
$L_{\extension{v_i}}$ to differentiate either $\psi$ or $f$. If
we differentiate $\psi$ $m$ times for some $m>0$, we obtain an
integral of $\ell-m$ derivatives of $f$ against a
$C^{k+\ell-m}$ function, hence this is bounded by $C
e_{k,\ell',s}^A(f)$ for $\ell'=\ell-m$ (note that we are
working in the lift of a compact subset of $\Teich_1/\Gamma$,
hence the $C^{k+\ell}$ norm of the extended vector fields
$\extension{v_j}$ is bounded). The remaining term is $\int \phi
\psi \cdot L_{\extension{v_1}}\cdots L_{\extension{v_\ell}}f$.
If $A$ is large enough, $\norm{ \phi \psi}_{C_A^{k+\ell}} \leq
\norm{\phi}_{C_A^{k+\ell}} \norm{\psi}_{C_A^{k+\ell}} \leq 2$,
hence this integral is bounded by $2 e_{k,\ell,s}^A(f)$. We
have proved \eqref{lkjsmlkjsfd}.
\end{proof}

\begin{lem}
\label{lem_close_infinity}
For every $k,n\in \N$, for every $\gamma\in \{1,2\}^n$, for
every large enough $A,B$, we have for every $f\in \boD$
  \begin{equation*}
  \norm{\tilde\boL_{\gamma_1} \cdots \tilde\boL_{\gamma_n} f}_{k,s}^{A,B}
  \leq (10C_0)^n e^{-(1-2\delta) T_0 \Card\{i \st \gamma_i=2\}}
  \norm{f}_{k,s}^{A,B}.
  \end{equation*}
\end{lem}
\begin{proof}
It is sufficient to prove that
  \begin{equation*}
  \norm{\tilde\boL_1 f}_{k,s}^{A,B}\leq 10C_0 \norm{f}_{k,s}^{A,B} \quad \text{and}\quad
  \norm{\tilde \boL_2 f}_{k,s}^{A,B}\leq 10C_0 e^{-(1-2\delta)T_0} \norm{f}_{k,s}^{A,B}.
  \end{equation*}

Since $V_\delta$ is bounded from below by $1$, Lemma
\ref{lemme_basic_estimate} shows that $\norm{\boL_{T_0}
f}_{k,s}^{A,B} \leq 2C_0 \norm{f}_{k,s}^{A,B}$ if $A$ is large
enough. Therefore,
  \begin{equation*}
  \norm{\tilde \boL_1 f}_{k,s}^{A,B}
  =\norm{\boL_{T_0} (\partition_V f)}_{k,s}^{A,B}
  \leq 2C_0 \norm{\partition_V f}_{k,s}^{A,B}
  \leq 6C_0 \norm{f}_{k,s}^{A,B},
  \end{equation*}
by Lemma \ref{largeB_hence_OK}, if $A$ and $B$ are large
enough.

We turn to $\tilde \boL_2 f= \boL_{T_0}( (1-\partition_V) f)$.
Let $x\in X$. Since $\log V_\delta$ is $2$-Lipschitz,
$V_\delta(g_{T_0} y) \leq e^{2T_0} V_\delta(y)$ for all $y$. If
$V_\delta(x) \leq e^{-2T_0}V/2$, it follows that $V_\delta(y)
\leq  e^{-2T_0} V$ on $W^u_{1/100}(x)$, and therefore that
$V_\delta(g_{T_0} y) \leq V$ on $g_{T_0} ( W^u_{1/100}(x))$.
Hence, $1-\partition_V=0$ on this set. The definition of
$e_{k,\ell,s}^A$ gives $e_{k,\ell,s}^{A}(
\boL_{T_0}((1-\partition_V) f) ; x) = 0$. We therefore obtain
by Lemma \ref{lemme_basic_estimate}
  \begin{align*}
  e_{k,\ell,s}^A( \boL_{T_0}((1-\partition_V) f) &= \sup_{ V_\delta(x) \geq e^{-2T_0}V/2}
  e_{k,\ell,s}^A( \boL_{T_0}((1-\partition_V) f ; x)
  \\&
  \leq \sup_{ V_\delta(x) \geq e^{-2T_0}V/2}
  C_0 e_{k,\ell,s}^A((1-\partition_V) f) \left( e^{-(1-2\delta)T_0} + 1/V_\delta(x)\right)
  \\&
  \leq C_0 e_{k,\ell,s}^A((1-\partition_V) f) \left( e^{-(1-2\delta)T_0} + 2 e^{2T_0}/V\right).
  \end{align*}
Taking into account the definition of
$\norm{\cdot}_{k,s}^{A,B}$ and the equality $2 e^{2T_0}/V =
e^{-(1-2\delta) T_0}$, we obtain
  \begin{equation*}
  \norm{\boL_{T_0} ((1-\partition_V)f)}_{k,s}^{A,B}
  \leq 2C_0 \norm{(1-\partition_V)f}_{k,s}^{A,B} e^{-(1-2\delta)T_0}.
  \end{equation*}
By Lemma \ref{largeB_hence_OK},
$\norm{(1-\partition_V)f}_{k,s}^{A,B} \leq \norm{f}_{k,s}^{A,B}
+\norm{\partition_V f}_{k,s}^{A,B} \leq 4 \norm{f}_{k,s}^{A,B}$
if $A,B$ are large enough. We obtain $\norm{\tilde \boL_2
f}_{k,s}^{A,B} \leq 8C_0 e^{-(1-2\delta)T_0}
\norm{f}_{k,s}^{A,B}$ as desired.
\end{proof}

We defined an auxiliary norm $\normdeux{\cdot}_k$ in
\eqref{defnormdeux}.

\begin{lem}
\label{maincontraction}
Consider $\gamma= (\gamma_1,\dots,\gamma_n)$ with $\Card\{i \st
\gamma_i = 1\} \geq T/T_0$. Then, for all $k$, if $A$ and $B$
are large enough,
  \begin{equation}
  \label{mlxkjvwmxvi}
  \norm{\tilde\boL_{\gamma_1} \cdots \tilde\boL_{\gamma_n} f}_{k,s}^{A,B}
  \leq 2^{-k/2}\cdot 12C_0 \norm{f}_k^{A,B}
  + C_{n,k,A,B} \normdeux{\psi_{\gamma} f}_k,
  \end{equation}
where the function $\psi_{\gamma}$ is $C^\infty$ and supported
in a compact set mod $\Gamma$.
\end{lem}
The point of this lemma is that, if $\gamma$ is fixed, we can
choose $k$ very large to make the first term in
\eqref{mlxkjvwmxvi} arbitrarily small, while the second term
gives a compact contribution (thanks to Proposition
\ref{prop_compactness}), and will therefore not be an issue to
control the essential spectral radius.
\begin{proof}
We can write $\tilde\boL_{\gamma_1}\dots \tilde \boL_{\gamma_n} f =
\boL_{nT_0} (\psi f)$, where $\psi=\psi_\gamma=\prod_{j=1}^n \psi_{\gamma_j}
\circ g_{-(n-j)T_0}$ is $C^\infty$ and compactly supported.

To estimate $e_{k,\ell,s}^A (\tilde\boL_{\gamma_1}\dots \tilde
\boL_{\gamma_n} f)$ for some $0\leq \ell \leq k$, we should
estimate integrals of the form
  \begin{equation}
  \label{qsmiufqsdfui2}
  \int_{W^u_{1/200}(x)} \phi \cdot L_{\extension{v_1}}\cdots L_{\extension{v_\ell}}(\boL_{nT_0} (\psi f))\dd\mu_u,
  \end{equation}
where $\norm{\phi}_{C^{k+\ell}_A}\leq 1$, the vector fields
$v_j$ all point in the stable direction and
$\norm{v_j}_{C^{k+\ell+1}_A}\leq 1$.

As in the proof of Lemma \ref{lemme_basic_estimate}, we first replace
$v_j$ by a compactly supported vector field $v_{j,1}$ on
$W^u_{1/100}(x)$, with $\norm{v_{j,1}}_{C^{k+\ell+1}_A}\leq 2^{1/2}$
(assuming $A$ is large enough). Let $w_j$ be the push-forward of
$v_{j,1}$ under $g_{nT_0}$, and let $(\partition_i)$ be a partition
of unity on $g_{nT_0}(W^u_{1/200}(x))$ (c.f.~Proposition
\ref{prop_partition}). The integral \eqref{qsmiufqsdfui2} becomes
  \begin{equation*}
  \sum_{i\in I} \int_{W^u_{1/200}(x_i)} \partition_i(z) \phi(g_{-nT_0}z)
  \cdot
  L_{\extension{w_1}}\cdots L_{\extension{w_\ell}}(\psi f) (z) \, e^{-dnT_0}\dd\mu_u(z).
  \end{equation*}
Let $I'\subset I$ be the set of $i$s such that $\psi$ is not
identically zero on $W^u_{1/200}(x_i)$. We claim that, for
$i\in I'$, for all $y\in W^u_{1/200}(x_i)$,
  \begin{equation}
  \label{spend_a_lot_of_time_in_K}
  \Leb \{ s\in [0, nT_0] \st g_{-s}(y) \in K\} \geq T,
  \end{equation}
where $K$ is defined in \eqref{defK}. Indeed, let $z\in
W^u_{1/200}(x_i)$ satisfy $\psi(z)\not=0$. For all $j$ with
$\gamma_j=1$, we have $\psi_1( g_{-(n-j)T_0} z)\not=0$,
therefore $V_\delta(g_{-(n-j)T_0} z)\leq 2V$. Since
$g_{-(n-j)T_0}$ is a contraction along $W^u$, we obtain
$V_\delta( g_{-(n-j)T_0} y) \leq 4V$ for any $y\in
W^u_{1/200}(x_i)$. For any $s\in [0,T_0]$,
$V_\delta(g_{-s}g_{-(n-j)T_0} y) \leq e^{2s}
V_\delta(g_{-(n-j)T_0} y) \leq e^{2T_0} 4V$, i.e.,
$g_{-s}g_{-(n-j)T_0} y \in K$. This implies that
  \begin{equation*}
  \Leb \{ s\in [0, nT_0] \st g_{-s}(y) \in K\}
  \geq T_0 \Card\{j \st \gamma_j=1\},
  \end{equation*}
which is greater than or equal to $T$, by the assumptions of
the lemma. This proves \eqref{spend_a_lot_of_time_in_K}.

Fix now $i\in I'$, we work along $W^u_{1/200}(x_i)$. Since
$g_t$ is uniformly hyperbolic along trajectories that spend a
time at least $T$ in $K$ (by Proposition
\ref{pliss_hyperbolicity}), we have $c_m(\phi \circ g_{-n T_0})
\leq 2^{-m} c_m(\phi)$, and $c_m(w_j) \leq 2^{-m-1}
c_m(v_{j,1})$ (note that we have a gain even for $m=0$ since
the vector itself is contracted by the differential of $g_{n T_0}$).
This gives $\norm{\phi \circ g_{-n T_0}}_{C^{k+\ell}_A} \leq
\norm{\phi}_{C^{k+\ell}_A}$ (there is no gain here at level
$m=0$, so no gain overall) and $\norm{w_j}_{C^{k+\ell+1}_A}
\leq 2^{-1} \norm{v_{j,1}}_{C^{k+\ell+1}_A}\leq 2^{-1/2}$. This
gives a gain of $2^{-1/2}$ with respect to the non-contracting
situation of Lemma \ref{lemme_basic_estimate}, and we end up
after the same computations with
  \begin{equation}
  \label{Irbasic}
  e_{k,\ell,s}^A(\boL_{nT_0}(\psi f)) \leq 2C_0 \cdot 2^{-\ell/2} e_{k,\ell,s}^A(\psi f).
  \end{equation}
This gives a certain gain if $\ell$ is large. In particular,
for $\ell=k$, we obtain a gain of $2^{-k/2}$, as in the
estimate \eqref{mlxkjvwmxvi} we are trying to prove. However,
this is not sufficient for smaller $\ell$. Assume now $\ell<k$,
we will regularize the function $\phi$ by convolution in this
case.

For $\epsilon>0$, we consider a function $\tilde \phi$ on
$W^u_{1/100}(x_i)$ such that $c_m(\phi-\tilde \phi)\leq \epsilon$ for
$m<k+\ell$, $c_{k+\ell}(\tilde \phi) \leq 2 c_{k+\ell}(\phi)$ and
$c_{k+\ell+1}(\tilde \phi) \leq C_{k,A}/\epsilon$. Note that, since
$\tilde\phi$ is obtained by convolution between $\phi$ and a kernel
of support of size $\epsilon$, the support of $\tilde\phi$ is larger
than that of $\phi$. Since all the functions we are considering are
multiplied by the partition of unity $\partition_i$, this is not a
problem.

Along $W^u_{1/200}(x_i)$, the function $\phi'=(\phi-\tilde
\phi)\circ g_{-nT_0}$ satisfies $c_m(\phi')\leq \epsilon$ for
$m<k+\ell$ and $c_{k+\ell}(\phi') \leq 2\cdot 2^{-(k+\ell)}
c_{k+\ell}(\phi)$. Choosing $\epsilon=2^{-(4+k+\ell)}$, we have
finally $\norm{\phi'}_{C^{k+\ell}_A} \leq
e^{1/A} 2^{-4-k-\ell}+2^{1-k-\ell}\norm{\phi}_{C^{k+\ell}_A} \leq 2^{(3/2)-k-\ell}$ for
any $A \geq 1$. Let us decompose in \eqref{qsmiufqsdfui2} the function
$\phi$ as $\phi'+\tilde \phi$. The resulting term coming from
$\phi'$ is similar to \eqref{Irbasic} but with an additional
factor $2^{(3/2)-k-\ell}$, while the term coming from $\tilde\phi$
is bounded in terms of $\normdeux{f}_k$, since there are at
most $\ell<k$ derivatives of $f$ integrated against a function
in $C^{k+\ell+1}$. In the end, we get
  \begin{equation*}
  e_{k,\ell,s}^A(\boL_{nT_0}(\psi f)) \leq
  2C_0 \cdot 2^{-\ell/2} \cdot 2^{(3/2)-k-\ell} e_{k,\ell,s}^A(\psi f)
  +C_{n,\gamma,A,k} \normdeux{\psi f}_k.
  \end{equation*}
Summing the last equation for $\ell=0,\dots,k-1$ and
\eqref{Irbasic} for $\ell=k$, we obtain
  \begin{equation*}
  \norm{\boL_{nT_0}(\psi f)}_{k,s}^{A,B}
  \leq 4C_0 2^{-k/2} \norm{\psi f}_{k,s}^{A,B}
  + C_{n,\gamma,A,k} \normdeux{\psi f}_k.
  \end{equation*}
Since the function $\psi$ is $C^\infty$ and compactly
supported, Lemma \ref{largeB_hence_OK} applies if $B$ is large
enough. This concludes the proof.
\end{proof}

To simplify notations, we write $O^{comp}(f)$ for terms bounded
by $\normdeux{\psi f}_k$, for some $C^\infty$ function $\psi$
in $\Teich_1$ that is supported in a compact set mod $\Gamma$.
This notation is invariant under $\boL_t$ for fixed $t$ (since
this operator acts continuously for $\normdeux{\cdot}_k$), and
under addition (if $\psi_1$ and $\psi_2$ are two $C^\infty$
functions whose support is compact mod $\Gamma$, consider a
function $\psi$ with the same properties which is equal to $1$
on $\supp(\psi_1) \cup \supp(\psi_2)$, then $\normdeux{\psi_1
f}_k=\normdeux{\psi_1 \psi f}_k \leq C(\psi_1) \normdeux{\psi
f}_k$, and a similar inequality holds for $\psi_2$).

\begin{cor}
For every $n\in \N$ with $n\geq T/(\delta T_0)$, if $k,A,B$ are
large enough, we have
  \begin{equation*}
  \norm{ \boL_{nT_0} f}_{k,s}^{A,B} \leq e^{-(1-4\delta)nT_0} \norm{f}_k^{A,B}
  + O^{comp}(f).
  \end{equation*}
\end{cor}
\begin{proof}
We write $\boL_{nT_0} f= \sum_{\gamma\in \{1,2\}^n} \tilde
\boL_{\gamma_1} \cdots \tilde\boL_{\gamma_n} f$, and estimate
the terms coming from each $\gamma$.

If $\Card\{j\st \gamma_j=1\} \geq \delta n$, then the resulting
term is bounded by Lemma \ref{maincontraction}. Otherwise,
$\Card\{j\st \gamma_j=2\} \geq (1-\delta)n$, and Lemma
\ref{lem_close_infinity} gives an upper bound of the form
$(10C_0)^n e^{-(1-2\delta) T_0 (1-\delta) n} \norm{f}_k^{A,B}$.
Since $(1-2\delta)(1-\delta) \geq 1-3\delta$, we obtain after
summing over the $2^n$ possible values of $\gamma$
  \begin{align*}
  \norm{ \boL_{nT_0} f}_{k,s}^{A,B} &
  \leq 2^n (10 C_0)^n e^{-(1-3\delta) T_0 n} \norm{f}_{k}^{A,B}
  + 2^n \cdot 12C_0 2^{-k/2} \norm{f}_k^{A,B}
  \\ &\ \ \ \ %
  + C_{n,k,A,B} \sum \normdeux{\psi_{n,\gamma} f}_k.
  \end{align*}
Choosing $k$ large enough, we can make sure that $12 C_0
2^{-k/2} \leq (10 C_0)^n e^{-(1-3\delta) T_0 n}$, and we obtain
a bound of the form
  \begin{equation*}
  (40 C_0)^n e^{-(1-3\delta)T_0 n} \norm{f}_{k}^{A,B}
  +C_{n,k,A,B} \sum \normdeux{\psi_{n,\gamma} f}_k.
  \end{equation*}
Since $40C_0 \leq e^{\delta T_0}$, this implies the statement
of the corollary.
\end{proof}

\begin{cor}
\label{cor_s}
For any large enough $N$, if $k,A,B$ are large enough, we
have
  \begin{equation*}
  \norm{\boM^N f}_{k,s}^{A,B} \leq 2C_0 (e^{(1-4\delta) T_0}+2) \norm{f}_k^{A,B} +O^{comp}(f).
  \end{equation*}
\end{cor}
\begin{proof}
We start from the formula
  \begin{equation*}
  \boM^N f=\int_{t=0}^\infty \frac{t^{N-1}}{(N-1)!} e^{-4\delta t} \boL_t f
\dd t  =\sum_{n=0}^\infty \int_{nT_0}^{(n+1)T_0} \frac{t^{N-1}}{(N-1)!} e^{-4\delta t} \boL_t f
\dd t.
  \end{equation*}

On an interval $[nT_0, (n+1)T_0]$ with small $n$ (i.e., $n<
T/(\delta T_0)$), we use the simple bound $\norm{\boL_t
f}_{k,s}^{A,B} \leq 2C_0 \norm{f}_k^{A,B}$ coming from Lemma
\ref{lemme_basic_estimate}. Since for any fixed $T_*>0$,
$\int_0^{T_*}
\frac{t^{N-1}}{(N-1)!} \dd t$ tends to zero when $N\to\infty$,
the contribution of those intervals is bounded, say, by $2C_0
\norm{f}_k^{A,B}$ if $N$ is large enough.

We use the same trivial bound on the intervals $[nT_0,
(n+1)T_0]$ with very large $n$ ($n\geq n_0(N)$ to be chosen
later). The contribution of these intervals is then bounded by
  \begin{equation*}
  \int_{ n_0(N)T_0}^\infty \frac{t^{N-1}}{(N-1)!} e^{-4\delta t} 2C_0 \norm{f}_k^{A,B} \dd t.
  \end{equation*}
Choosing $n_0(N)$ large enough, we can ensure that this is
bounded by $2C_0 \norm{f}_k^{A,B}$.

Consider now $n$ in between. For $t\in [nT_0, (n+1)T_0]$, we
have
  \begin{align*}
  \norm{ \boL_t f}_{k,s}^{A,B} &
  \leq 2C_0 \norm{\boL_{nT_0} f}_{k,s}^{A,B}
  \leq 2C_0 e^{-(1-4\delta) nT_0} \norm{f}_k^{A,B}
  +O^{comp}(f)
  \\&
  \leq 2C_0 e^{(1-4\delta) T_0} e^{-(1-4\delta)t}
  \norm{f}_k^{A,B}  +O^{comp}(f).
  \end{align*}
Integrating over $t$ and then summing over $n$, we get a
contribution bounded by
  \begin{equation*}
  2C_0 e^{(1-4\delta) T_0} \int_{t=0}^{\infty}
  \frac{t^{N-1}}{(N-1)!} e^{-4\delta t} e^{-(1-4\delta)t} \norm{f}_k^{A,B} \dd t
  +O^{comp}(f),
  \end{equation*}
which is bounded by $ 2C_0 e^{(1-4\delta) T_0} \norm{f}_k^{A,B}
+ O^{comp}(f)$ since $\int_{t=0}^\infty \frac{t^{N-1}}{(N-1)!}
e^{-t} \dd t= 1$.
\end{proof}

\begin{proof}[Proof of Theorem \ref{thm_control_spectral_radius}]
The first part of the statement is contained in Lemma
\ref{lemme_basic_estimate}.  It remains to estimate the essential
spectral radius of $\boM$.  Adding the estimates of Lemma
\ref{lem_omega} and of Corollary \ref{cor_s}, we have for large
enough $N,k,A,B$
  \begin{equation*}
  \norm{\boM^N f}_k^{A,B} \leq 2C_0(e^{(1-4\delta) T_0}+5) \norm{f}_k^{A,B}
  +O^{comp}(f).
  \end{equation*}
Let us fix once and for all $N$ large enough so that
$2C_0(e^{(1-4\delta) T_0}+5) \leq (1+\delta)^N$, and $k,A,B$
such that the previous estimate holds. This estimate translates
into the following: there exists a $C^\infty$ function $\psi$
supported in a compact set mod $\Gamma$ such that, for any
function $f$ in $\boD$,
  \begin{equation*}
  \norm{\boM^N f}_k^{A,B} \leq (1+\delta)^N \norm{f}_k^{A,B}
  +\normdeux{\psi f}_k.
  \end{equation*}
The unit ball of $\boD^\Gamma$ for the norm $\norm{\cdot}_k$ is
relatively compact for the semi-norm $\normdeux{f}\coloneqq
\normdeux{\psi f}_k$, by Proposition ~ \ref{prop_compactness}. By
Hennion's Theorem (Lemma ~ \ref{lemme_hennion}), it follows
that the essential spectral radius of $\boM$ for the norm
$\norm{\cdot}_k^{A,B}$ on the space $\boD^\Gamma$ is at most
$1+\delta$. Since this norm is equivalent to $\norm{\cdot}_k$,
this concludes the proof of Theorem ~
\ref{thm_control_spectral_radius}.
\end{proof}

\appendix
\section{}
\subsection{Spherical functions}
\label{appendix_cfunction}
In this section, we prove the estimate \eqref{harish} on the
behavior of the spherical function $\phi_{\xi_s}$ when $\xi_s$
is a representation of $\SL$ in the complementary series. It is
a consequence of classical estimates on spherical functions,
let us for instance follow the computations in
\cite{helgason2}. For $\Re s\in [-1,1]$, let us define coefficients
$\Gamma_n(s)$ by $\Gamma_0=1$, $\Gamma_n=0$ if $n$ is odd, and
$n(n-s)\Gamma_n(s)=\sum_{0<k\leq n/2}
\Gamma_{n-2k}(s)(2n-4k-s+1)$ if $n$ is even. It is easy to check by induction
that these coefficients grow more slowly than any exponential.
In particular (see, e.g., \cite[Lemma 4.13]{helgason2}), for
every $\epsilon>0$, there exists a constant $C>0$ such that
  \begin{equation}
  \label{Gamma_bound}
  \forall s \in [-1, 1],\ \forall n\in \N,\quad |\Gamma_n(s)|\leq Ce^{\epsilon n}.
  \end{equation}

These coefficients are chosen so that $t\mapsto e^{(s-1)t}\sum
\Gamma_n(s) e^{-2nt}$ satisfies an explicit differential equation of
order $2$ which is also satisfied by $\phi_{\xi_s}$. Another solution
of the same equation is $t\mapsto e^{(-s-1)t}\sum \Gamma_n(-s)
e^{-2nt}$. It follows that $\phi_{\xi_s}$ is a linear combination of
those two functions. One can identify the coefficients in this linear
combination (they are given by the $c$ function \eqref{c(s)}), to
obtain the following formula for $\phi_{\xi_s}$: for every $s\in
(0,1] \cup \ic(0,+\infty)$,
  \begin{equation*}
  \phi_{\xi_s}(g_t)=c(s) e^{(s -1)t}\sum_{n\geq 0}\Gamma_n(s) e^{-2nt}
  +c(-s) e^{(-s -1)t} \sum_{n\geq 0} \Gamma_n(-s) e^{-2nt}.
  \end{equation*}
This is \cite[Theorem IV.5.5]{helgason2} in the case of $\SL$
(the formula for $c$ is given in \cite[Theorem
IV.6.4]{helgason2}).

For $s\in [\delta,1]$, the dominating term in this formula is
$c(s) e^{(s-1)t}$, and the sum of the other terms is bounded by
$C e^{-t}$ if $t\geq 1$, by \eqref{Gamma_bound}. Since
$\phi_{\xi_s}(g_t) - c(s) e^{(s-1)t}$ is uniformly bounded for
$t\in [0,1]$ and $s\in [\delta,1]$, the estimate \eqref{harish}
follows.

\subsection{Boundary behavior of Cauchy transforms}
\begin{lem}
\label{lemchargepas}
Let $\nu$ be a nonnegative measure on $[0,1]$, with finite
mass. Assume that the function $F(z)=\int_{s\in [0,1]}
\frac{\dd\nu(s)}{z-s+1}$, defined for $z\in \C - [-1,0]$,
admits a continuous extension to an interval $[a-1,b-1]\subset
[-1,0]$. Then $\nu[a,b]=0$.
\end{lem}
\begin{proof}
Let us first show that, if $F$ is continuous at a point $x-1$,
with $x\in [0,1]$, then
  \begin{equation}
  \label{prolonge}
  \nu[x-\epsilon, x+\epsilon]=o(\epsilon).
  \end{equation}
We have
  \begin{equation*}
  F(x-1+\ic y)=\int \frac{\dd\nu(s)}{x+\ic y-s}
  =\int \frac{x-s-\ic y}{ (x-s)^2+y^2} \dd\nu(s).
  \end{equation*}
As a consequence,
  \begin{equation*}
  \Ima (F(x-1+\ic y)-F(x-1-\ic y))=-2\int \frac{y}{ (x-s)^2+y^2} \dd\nu(s).
  \end{equation*}
If $F$ can be extended continuously to $x-1$, this quantity
tends to $0$. For $s\in[x-y,x+y]$, the integrand is at least
$y/(2y^2)$, therefore
  \begin{multline*}
  \nu[x-y,x+y]/y \leq 2\int_{s=x-y}^{x+y} \frac{y}{ (x-s)^2+y^2} \dd\nu(s)
  \\
  \leq | \Ima( F(x-1+\ic y)-F(x-1-\ic y))|\to 0.
  \end{multline*}
This proves \eqref{prolonge}.

Assume now that $F$ can be continuously extended to a whole
interval $[a-1,b-1]$. For any $x\in [a,b]$, we have
$\nu[x-\epsilon,x+\epsilon]=o(\epsilon)$. By \cite[Theorem
2.12]{mattila}, for any $\rho>0$, we can cover $[a,b]$ with
intervals $I_n$ with bounded overlap, with $\nu(I_n) \leq \rho
|I_n|$. Therefore, $\nu[a,b]\leq \sum \nu(I_n) \leq \rho
\sum|I_n|\leq \rho C(b-a)$. Letting $\rho$ tend to $0$, we
obtain $\nu[a,b]=0$.
\end{proof}

\bibliography{biblio}

\providecommand{\bysame}{\leavevmode\hbox to3em{\hrulefill}\thinspace}
\providecommand{\MR}{\relax\ifhmode\unskip\space\fi MR }
% \MRhref is called by the amsart/book/proc definition of \MR.
\providecommand{\MRhref}[2]{%
  \href{http://www.ams.org/mathscinet-getitem?mr=#1}{#2}
}
\providecommand{\href}[2]{#2}
\begin{thebibliography}{ABEM06}

\bibitem[ABEM06]{athreya_et_al}
Jayadev Athreya, Alexander Bufetov, Alex Eskin, and Maryam Mirzakhani,
  \emph{Lattice point asymptotics and volume growth on {T}eichm{\"u}ller
  space}, Preprint, 2006.

\bibitem[AGY06]{AGY_teich}
Artur Avila, S{\'e}bastien Gou{\"e}zel, and Jean-Christophe Yoccoz,
  \emph{Exponential mixing for the {T}eichm\"uller flow}, Publ. Math. Inst.
  Hautes \'Etudes Sci. (2006), 143--211. \MR{MR2264836}

\bibitem[AR09]{avila_resende}
Artur Avila and Maria~Jo{\~a}o Resende, \emph{Exponential mixing for the
  {T}eichmuller flow in the space of quadratic differentials}, Preprint, 2009.

\bibitem[Ath06]{athreya}
Jayadev~S. Athreya, \emph{Quantitative recurrence and large deviations for
  {T}eichmuller geodesic flow}, Geom. Dedicata \textbf{119} (2006), 121--140.
  \MR{MR2247652}

\bibitem[BER99]{CR_book}
M.~Salah Baouendi, Peter Ebenfelt, and Linda~Preiss Rothschild, \emph{Real
  submanifolds in complex space and their mappings}, Princeton Mathematical
  Series, vol.~47, Princeton University Press, Princeton, NJ, 1999.
  \MR{MR1668103}

\bibitem[BGK07]{BGK_coupling}
Jean-Baptiste Bardet, S{\'e}bastien Gou{\"e}zel, and Gerhard Keller,
  \emph{Limit theorems for coupled interval maps}, Stoch. Dyn. \textbf{7}
  (2007), 17--36. \MR{MR2293068}

\bibitem[BL98]{babillot_ledrappier}
Martine Babillot and Fran{\c{c}}ois Ledrappier, \emph{Geodesic paths and
  horocycle flow on abelian covers}, Lie groups and ergodic theory (Mumbai,
  1996), Tata Inst. Fund. Res. Stud. Math., vol.~14, Tata Inst. Fund. Res.,
  Bombay, 1998, pp.~1--32. \MR{MR1699356}

\bibitem[BL07]{butterley_liverani}
Oliver Butterley and Carlangelo Liverani, \emph{Smooth {A}nosov flows:
  correlation spectra and stability}, J. Mod. Dyn. \textbf{1} (2007), 301--322.
  \MR{MR2285731}

\bibitem[BT07]{bt_aniso}
Viviane Baladi and Masato Tsujii, \emph{Anisotropic {H}\"older and {S}obolev
  spaces for hyperbolic diffeomorphisms}, Ann. Inst. Fourier (Grenoble)
  \textbf{57} (2007), 127--154. \MR{MR2313087}

\bibitem[CM82]{casselman_milicic}
William Casselman and Dragan Mili{\v{c}}i{\'c}, \emph{Asymptotic behavior of
  matrix coefficients of admissible representations}, Duke Math. J. \textbf{49}
  (1982), 869--930. \MR{MR683007}

\bibitem[Dix69]{dixmier}
Jacques Dixmier, \emph{Les {$C\sp{\ast} $}-alg\`ebres et leurs
  repr\'esentations}, Deuxi\`eme \'edition. Cahiers Scientifiques, Fasc. XXIX,
  Gauthier-Villars \'Editeur, Paris, 1969. \MR{MR0246136}

\bibitem[EM01]{eskin_masur}
Alex Eskin and Howard Masur, \emph{Asymptotic formulas on flat surfaces},
  Ergodic Theory Dynam. Systems \textbf{21} (2001), 443--478. \MR{MR1827113}

\bibitem[EM09]{every_curve_is_teich}
Jordan~S. Ellenberg and David~B. McReynolds, \emph{Every curve is a
  {T}eichmuller curve}, Preprint, 2009.

\bibitem[For02]{forni_deviation}
Giovanni Forni, \emph{Deviation of ergodic averages for area-preserving flows
  on surfaces of higher genus}, Ann. of Math. (2) \textbf{155} (2002), 1--103.
  \MR{MR1888794}

\bibitem[Fri86]{fried_analytique}
David Fried, \emph{The zeta functions of {R}uelle and {S}elberg. {I}}, Ann.
  Sci. \'Ecole Norm. Sup. (4) \textbf{19} (1986), 491--517. \MR{MR875085}

\bibitem[GL06]{gouezel_liverani}
S{\'e}bastien Gou{\"e}zel and Carlangelo Liverani, \emph{Banach spaces adapted
  to {A}nosov systems}, Ergodic Theory Dynam. Systems \textbf{26} (2006),
  189--217. \MR{MR2201945}

\bibitem[Hej83]{hejhal_2}
Dennis~A. Hejhal, \emph{The {S}elberg trace formula for {${\rm PSL}(2,\,{\bf
  R})$}. {V}ol. 2}, Lecture Notes in Mathematics, vol. 1001, Springer-Verlag,
  Berlin, 1983. \MR{MR711197}

\bibitem[Hel00]{helgason2}
Sigurdur Helgason, \emph{Groups and geometric analysis}, Mathematical Surveys
  and Monographs, vol.~83, American Mathematical Society, Providence, RI, 2000,
  Integral geometry, invariant differential operators, and spherical functions,
  Corrected reprint of the 1984 original. \MR{MR1790156}

\bibitem[Hen93]{hennion}
Hubert Hennion, \emph{Sur un th\'eor\`eme spectral et son application aux
  noyaux lipchitziens}, Proc. Amer. Math. Soc. \textbf{118} (1993), 627--634.
  \MR{MR1129880}

\bibitem[HM79]{howe_moore}
Roger~E. Howe and Calvin~C. Moore, \emph{Asymptotic properties of unitary
  representations}, J. Funct. Anal. \textbf{32} (1979), 72--96. \MR{MR533220}

\bibitem[H{\"o}r03]{hormander}
Lars H{\"o}rmander, \emph{The analysis of linear partial differential
  operators. {I}}, Classics in Mathematics, Springer-Verlag, Berlin, 2003,
  Distribution theory and Fourier analysis, Reprint of the second (1990)
  edition. \MR{MR1996773}

\bibitem[Iwa95]{Iwaniec}
Henryk Iwaniec, \emph{Introduction to the spectral theory of automorphic
  forms}, Biblioteca de la Revista Matem\'atica Iberoamericana. [Library of the
  Revista Matem\'atica Iberoamericana], Revista Matem\'atica Iberoamericana,
  Madrid, 1995. \MR{MR1325466}

\bibitem[Kat66]{kato_pe}
Tosio Kato, \emph{Perturbation theory for linear operators}, Die Grundlehren
  der mathematischen Wissenschaften, Band 132, Springer-Verlag New York, Inc.,
  New York, 1966. \MR{MR0203473}

\bibitem[Kna01]{knapp}
Anthony~W. Knapp, \emph{Representation theory of semisimple groups}, Princeton
  Landmarks in Mathematics, Princeton University Press, Princeton, NJ, 2001, An
  overview based on examples, Reprint of the 1986 original. \MR{MR1880691}

\bibitem[KS09]{sarnak_strong_spectral_gap}
Dubi Kelmer and Peter Sarnak, \emph{Strong spectral gaps for compact quotients
  of products of {$\mathrm{PSL}(2,\mathbb{R})$}}, J. Eur. Math. Soc. (JEMS)
  \textbf{11} (2009), 283--313. \MR{MR2486935}

\bibitem[KZ03]{Kontsevich_Zorich}
Maxim Kontsevich and Anton Zorich, \emph{Connected components of the moduli
  spaces of {A}belian differentials with prescribed singularities}, Invent.
  Math. \textbf{153} (2003), 631--678. \MR{MR2000471}

\bibitem[Lan08]{Lanneau_connected}
Erwan Lanneau, \emph{Connected components of the strata of the moduli spaces of
  quadratic differentials}, Ann. Sci. \'Ec. Norm. Sup\'er. (4) \textbf{41}
  (2008), 1--56. \MR{MR2423309}

\bibitem[Liv04]{liverani_contact}
Carlangelo Liverani, \emph{On contact {A}nosov flows}, Ann. of Math. (2)
  \textbf{159} (2004), 1275--1312. \MR{MR2113022}

\bibitem[Mas82]{Masur}
Howard Masur, \emph{Interval exchange transformations and measured foliations},
  Ann. of Math. (2) \textbf{115} (1982), 169--200. \MR{MR644018}

\bibitem[Mat95]{mattila}
Pertti Mattila, \emph{Geometry of sets and measures in {E}uclidean spaces},
  Cambridge Studies in Advanced Mathematics, vol.~44, Cambridge University
  Press, Cambridge, 1995, Fractals and rectifiability. \MR{MR1333890}

\bibitem[McM07]{mcmullen}
Curtis~T. McMullen, \emph{Dynamics of {${\rm SL}_2(\mathbb{R})$} over moduli
  space in genus two}, Ann. of Math. (2) \textbf{165} (2007), 397--456.
  \MR{MR2299738}

\bibitem[M{\"o}l08]{moller_linear}
Martin M{\"o}ller, \emph{Linear manifolds in the moduli space of one-forms},
  Duke Math. J. \textbf{144} (2008), 447--487. \MR{MR2444303}

\bibitem[Rat87]{ratner}
Marina Ratner, \emph{The rate of mixing for geodesic and horocycle flows},
  Ergodic Theory Dynam. Systems \textbf{7} (1987), 267--288. \MR{MR896798}

\bibitem[Rat92]{Ratner_sl}
\bysame, \emph{Raghunathan's conjectures for {$\mathrm{SL}(2,\mathbf{R})$}},
  Israel J. Math. \textbf{80} (1992), 1--31. \MR{MR1248925}

\bibitem[Sel65]{Selberg}
Atle Selberg, \emph{On the estimation of {F}ourier coefficients of modular
  forms}, Proc. {S}ympos. {P}ure {M}ath., {V}ol. {VIII}, Amer. Math. Soc.,
  Providence, R.I., 1965, pp.~1--15. \MR{MR0182610}

\bibitem[Vee82]{Veech_flow}
William~A. Veech, \emph{Gauss measures for transformations on the space of
  interval exchange maps}, Ann. of Math. (2) \textbf{115} (1982), 201--242.
  \MR{MR644019}

\end{thebibliography}
\bibliographystyle{amsalpha}

\end{document}